\documentclass[11pt]{amsart}
\usepackage{graphicx,pdflscape,rotating}
\usepackage[usenames]{color}
\usepackage{amssymb,amscd,amsthm,array}
\usepackage[normalem]{ulem}
\usepackage{graphicx,tikz,subfigure}
\usepackage{enumitem}
\usepackage{comment}
\usetikzlibrary{arrows,automata,chains,fit,shapes}
\usepackage{xcolor}
\newcommand{\Z}{\mathbb{Z}}
\newcommand{\N}{\mathbb{N}}

\newcommand{\R}{\mathbb{R}}

\newcommand{\BB}{{\mathcal B}}
\newcommand{\Bvuw}{{\mathcal B}_v^{u,w}}
\newcommand{\Lv}{\mathcal{L}_v}

\newcommand{\LL}{\mathcal{L}}

\newcommand{\DD}{\mathcal D}
\newcommand{\OO}{\mathcal O}
\newcommand{\QQ}{\mathcal Q}
\newcommand{\VV}{\mathcal V}

\newcommand{\ti}{{t^{-1}}}
\newcommand{\ai}{{a^{-1}}}

\newcommand{\ba}{{\bf a}}
\newcommand{\bb}{{\bf b}}
\newcommand{\x}{{\bf x}}
\newcommand{\xx}{{\bf x}}

\def\r{{\bf r}}
\newcommand{\wi}{{\bf w}^{(i)}}
\newcommand{\wf}[1]{{\bf w}^{(#1)}}
\newcommand{\wwi}{w^{(i)}}
\newcommand{\z}{{\bf z}}
\newcommand{\y}{{\bf y}}

\newcommand{\sgn}{\ensuremath{\textnormal{sign}}}

\newcommand{\fn}{\left\lfloor\frac{n}{2}\right\rfloor}
\newcommand{\fne}{\frac{n}{2}}
\newcommand{\ord}{\le_{u,w}}

\newcommand{\ords}{<_{u,w}}

\newcommand{\kx}{{k_{\xx}}}
\newcommand{\ky}{{k_{\y}}}
\newcommand{\K}[1]{{k_{#1}}}
\newcommand{\wt}{\ensuremath{\textnormal{weight}}}

\newcommand{\s}{{\bf s}}
\newcommand{\p}{{\bf p}}
\newcommand\OOn{\{|\OO_n(N)|\}_{N \in \N}}

\newtheorem{theorem}{Theorem}[section]
\newtheorem{corollary}[theorem]{Corollary}
\newtheorem{proposition}[theorem]{Proposition}
\newtheorem{lemma}[theorem]{Lemma}
\newtheorem{definition}{Definition}
\theoremstyle{definition}
\newtheorem{ex}[theorem]{Example}
\theoremstyle{remark}
\newtheorem{remark}[definition]{Remark}

\title{Conjugation Curvature in Solvable Baumslag-Solitar Groups}

\author[Jennifer Taback] {Jennifer Taback}
\address{Department of Mathematics, Bowdoin College, Brunswick, ME 04011}
\email{jtaback@bowdoin.edu}

\author[Alden Walker] {Alden Walker}
\address{Center for Communications Research, La Jolla, CA 92121}
\email{akwalke@ccrwest.org}

\thanks{The first author acknowledges support from Simons Foundation
grant 31736 to Bowdoin College.  Both authors thank Moon Duchin, Rob Kropholler and Murray Elder for insightful conversations.}

\date{\today}

\begin{document}

\begin{abstract}
For an element in $BS(1,n) = \langle t,a | ta\ti = a^n \rangle$ written in the normal form $t^{-u}a^vt^w$ with $u,w \geq 0$ and $v \in \Z$, we exhibit a geodesic word representing the element and give a formula for its word length with respect to the generating set $\{t,a\}$.  Using this word length formula, we prove that there are sets of elements of positive density of positive, negative and zero conjugation curvature, as defined by Bar Natan, Duchin and Kropholler. 
\end{abstract}

\maketitle

\setcounter{tocdepth}{1}
\tableofcontents

\section{Introduction}
The notion of discrete Ricci curvature for Cayley graphs of finitely generated groups was introduced by Bar-Natan, Duchin and Kropholler in \cite{BDK} as {\em conjugation curvature}.  Their work is based on that of Ollivier on metric Ricci curvature for graphs and non-manifold geometries \cite{YO1,YO2,YO3,YO4}.  One considers whether on average, ``corresponding points" on spheres of the same radius are closer together or farther apart than the centers of the spheres.  Negative conjugation curvature occurs when, on average, such points are farther apart the centers of the spheres, and positive curvature when they are closer together.

In the context of the Cayley graph of a finitely generated group, there is a natural interpretation of corresponding points; if $g_1,g_2 \in G=\langle S | R \rangle$ are the centers of two spheres of the same radius, then we consider the distance between $g_1w$ and $g_2w$ for $w \in G$.  Without loss of generality, we use the isometric action of a group on its Cayley graph to translate $g_1$ to the identity, and then $d_S(w,hw) = d_S(e,w^{-1}hw)=d_S(e,h^w)$ where $h=g_1^{-1}g_2$.  The conjugation curvature $\kappa_r(h)$ is then defined to be
\[
\kappa_r(h) = \frac{l(h) - \frac{1}{|n(r)|} \sum_{w \in S_n(r)} l(h^w)}{l(h)}
\]
that is, the difference between the word length of $h$ with respect to $S$ and the average word length of the conjugates of $h$ by all $w$ in the sphere $S_n(r)$ of radius $r$ centered at the identity in the Cayley graph $\Gamma(G,S)$,
scaled by the word length of $h$.

Bar-Natan, Duchin and Kropholler prove a variety of results using this definition of the conjugation curvature of a Cayley graph when $r=1$.  In particular, it is always zero at central elements, and any finite group has identically zero conjugation curvature when $S=G$.  This definition depends heavily on the generating set, and most groups considered in \cite{BDK} are viewed with respect to a natural generating set.  For some specific groups, they obtain strong conclusions.  For example, if $G$ is a right angled Artin group with the standard generating set, they obtain the dichotomy that for all $g$ in the group, $\kappa(g) = 0$ if and only if $g$ is central, otherwise $\kappa(g)<0$.  Additionally, if every element of a group has zero conjugation curvature, then the group is virtually abelian. For the Heisenberg group they show that there is a set of elements of positive density with each type of conjugation curvature: positive, negative and zero.

In this paper we show that the solvable Baumslag-Solitar groups
\[
BS(1,n)= \langle t,a | tat^{-1} = a^n \rangle
\]
for $n>1$ have a positive density of elements of positive, negative and zero conjugation curvature for $r$ in a bounded set of values.  To prove this, we require a detailed understanding of the shape of geodesics in the Cayley graph of $BS(1,n)$ with respect to the standard generating set $\{t,a\}$. Multiple people have studied geodesics in $BS(1,n)$, and while all reach similar conclusions, in each case the motivating questions frame the results in a unique way.

We begin with the standard normal form on $BS(1,n)$ and express each element uniquely as $g=t^{-u}a^vt^w$ for $u,v,w \in \Z$ with $u,w \geq 0$ where $n | v$ implies that $uw=0$.
We describe a deterministic algorithm which takes as input the triple $u,v,w$ and
produces a geodesic representative of the element.  
These geodesics come in four basic ``shapes''.  Our algorithm is lattice-based, and yields a succinct formula for
the word length of $g$ with respect to the generating set $\{t,a\}$.

In \cite{EH}, Elder and Hermiller produce a rubric for geodesics in $BS(1,n)$ and show that each $g \in BS(1,n)$ can be represented by a geodesic path which has one of their specified forms.  This exhaustive and detailed work is illuminating, but it does not link a given element of $BS(1,n)$ of the form $g=t^{-u}a^vt^w$ with a particular geodesic, which is what we require.

Elder in \cite{Elder} takes the first constructive approach to producing a geodesic for a given $g \in BS(1,n)$, exhibiting an algorithm to do so which runs in linear time and $O(n \log n)$ space, where $n$ is the length of an initial string of group generators.  Elder is motivated by the complexity of this algorithm, as it allows him to conclude that the bounded geodesic length problem can be solved in linear time for $BS(1,n)$.  This problem asks whether given a word of length $n$ in the generators of $G=\langle S | R \rangle$ and a nonnegative integer $k$, one can decide whether the geodesic length of the word is at most $k$.  This problem is NP-complete for free metabelian groups in their standard generating set \cite{MRUV} and hence the problem of finding an explicit geodesic representative for a given string of generators is NP-hard in the general case. While Elder produces the same geodesic paths that we find below, the fact that we are unconcerned with the complexity of the process streamlines our exposition, and we extend these common ideas by producing a word length formula at the conclusion of the algorithm.

Diekert and Laun in \cite{DL} exhibit an algorithm which produces geodesic paths for elements of $BS(m,n) = \langle t,a | ta^m\ti = a^n \rangle$ when $m |n$.  When $m=1$ they produce identical geodesic paths to those in \cite{Elder} and below.  They are also mainly concerned with the complexity of their algorithm; in general it is quadratic in the length of the initial string of generators and simplifies to linear when $m=1$.  However their methods are quite different from those of Elder in \cite{Elder}.  Neither \cite{DL} nor \cite{Elder} draw conclusions about word length in $BS(1,n)$.  Burillo and Elder in \cite{BE} prove metric estimates for word length in $BS(m,n)$ and use them to compute a lower bound on the growth rate of $BS(m,n)$.

Our method for producing a geodesic representative for $g = t^{-u}a^vt^w$ in $BS(1,n)$ allows us to investigate the growth rate of $BS(1,n)$ in \cite{TW_growth}.  
We use our techniques to show that the set of paths describing one shape of geodesics forms a regular language, and we exhibit a finite state automaton which accepts it.  It is sufficient to understand this set of geodesic paths, as it has the same growth rate as the entire group.  As an immediate
consequence, $BS(1,n)$ has rational growth, and we are able to obtain a simple expression for its growth rate, which was first computed by Collins, Edjvet and Gill in \cite{CEG}.

This paper is organized as follows.  Section~\ref{section:background} presents a brief introduction to the solvable Baumslag-Solitar groups. 
Section~\ref{section:min_rep} introduces our lattice-based methods and constructs a geodesic path for each $g = t^{-u}a^vt^w$ in $BS(1,n)$.  The results in this section provide numerical conditions on recognizing when certain paths are geodesic; these conditions are used in \cite{TW_growth} to compute the growth rate of $BS(1,n)$.
Section~\ref{sec:growth} contains some introductory remarks on growth, as well as an overview of the results of \cite{TW_growth}, where it is shown that a certain set of geodesics forms a regular language whose growth rate is identical to the growth rate of $BS(1,n)$.  
Section~\ref{sec:curvature} includes explicit descriptions of three infinite families of elements which have, respectively, positive, zero and negative conjugation curvature, when $n \geq 3$. Using our results
about growth rate from \cite{TW_growth}, we prove that these families have positive density in $BS(1,n)$.
Section~\ref{sec:n=2_curvature} contains analogous results for $n=2$.
In Section~\ref{sec:technical_lemmas} we prove several technical lemmas stated in Section~\ref{section:min_rep}.

\section{A geometric model for solvable Baumslag-Solitar groups}
\label{section:background}

For $n \in \N$ with $n > 1$, the solvable Baumslag-Solitar group $BS(1,n)$ has presentation $$BS(1,n) = \langle a,t | tat^{-1} = a^n \rangle.$$  
We consider elements of $BS(1,n)$ in the standard normal form, namely each $g \in BS(1,n)$ can be written uniquely as $t^{-u}a^vt^w$ where  $u,v,w \in \Z$ and $u,w \geq 0$, with the additional requirement that if $n|v$ then $uw=0$. 
If $n|v$ but $uw \neq 0$ then the group relator can be applied to simplify the normal form expression.  When we write $g=t^{-u}a^vt^w$ we assume that these conditions are satisfied.

The group $BS(1,n)$ for $n > 1$ acts property discontinuously and cocompactly by isometries on a metric 2-complex $X_n$ which is well described in the literature; see, for example, \cite{FM1} or \cite{freden}.  Topologically, this complex is the product $T_n \times \R$ where $T_n$ is a regular tree of valence $n+1$.  We equip $T_n$ with a height function $h:T_n \rightarrow \R$ so that vertices which differ by a single edge map to adjacent integers; this is well defined after an initial choice of vertex at height $0$.
Metrically, for any line $l \subset T_n$ where the heights of the vertices along $l$ map bijectively to $\Z$, the plane $l \times \R \subset T_n \times \R$ is a combinatorial model of the hyperbolic plane.  The 1-skeleton of this plane is tiled with the ``horobrick'' labeled by the group relator, depicted in Figure~\ref{fig:horobrick}; a part of this plane when $n=2$ is depicted in Figure~\ref{fig:plane}.

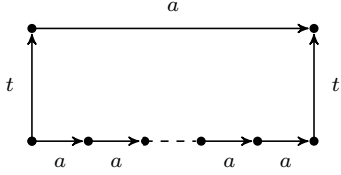
\begin{figure}[ht!]
\begin{tikzpicture}[>=stealth',shorten >=2pt,auto, node distance=5cm, semithick,scale=.75]
\tikzstyle{every node}=[draw,shape=circle]
\draw [->] (0,0) node[circle,fill,inner sep=1pt](a){} - - (1,0) node[draw=none,fill=none,font=\scriptsize,midway,below] {$a$};
\draw [->] (1,0) node[circle,fill,inner sep=1pt](b){} - - (2,0) node[draw=none,fill=none,font=\scriptsize,midway,below] {$a$};
\draw [dashed] (2,0) node[circle,fill,inner sep=1pt](b){} - - (3,0) node[draw=none,fill=none,font=\scriptsize,midway,below] {};
\draw [->] (3,0) node[circle,fill,inner sep=1pt](b){} - - (4,0) node[draw=none,fill=none,font=\scriptsize,midway,below] {$a$};
\draw [->] (4,0) node[circle,fill,inner sep=1pt](b){} - - (5,0) node[draw=none,fill=none,font=\scriptsize,midway,below] {$a$};
\draw [->] (0,0) node[circle,fill,inner sep=1pt](b){} - - (0,2) node[draw=none,fill=none,font=\scriptsize,midway,left] {$t$};
\draw [->] (5,0) node[circle,fill,inner sep=1pt](b){} - - (5,2) node[draw=none,fill=none,font=\scriptsize,midway,right] {$t$};
\draw [->] (0,2) node[circle,fill,inner sep=1pt](a){} - - (5,2) node[draw=none,fill=none,font=\scriptsize,midway,above] {$a$}
node[circle,fill,inner sep=1pt](c){};
\end{tikzpicture}
\caption{The ``horobrick'' which tiles the 1-skeleton of $X_n$; its boundary is
labeled by the relator $tat^{-1}a^{-n}$.}
\label{fig:horobrick}
\end{figure}

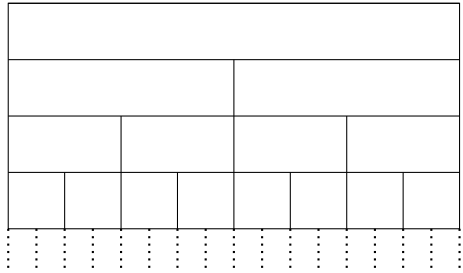
\begin{figure}[ht!]
\begin{tikzpicture}[scale=.75]
\draw (0,0) -- (0,4) -- (8,4) -- (8,0) -- (0,0);
\draw (0,1) -- (8,1);
\draw (0,2) -- (8,2);
\draw (0,3) -- (8,3);
\draw (0,4) -- (8,4); 
15

\draw (4,3) -- (4,0);
\draw (2,2) -- (2,0);
\draw (6,2) -- (6,0);
\draw(1,1) -- (1,0);
\draw(3,1) -- (3,0);
\draw(5,1) -- (5,0);
\draw(7,1) -- (7,0);
\draw [thick, dotted] (0,0) --(0,-.75);
\draw [thick, dotted] (1,0) --(1,-.75);
\draw [thick, dotted] (2,0) --(2,-.75);
\draw [thick, dotted] (3,0) --(3,-.75);
\draw [thick, dotted] (4,0) --(4,-.75);
\draw [thick, dotted] (5,0) --(5,-.75);
\draw [thick, dotted] (6,0) --(6,-.75);
\draw [thick, dotted] (7,0) --(7,-.75);
\draw [thick, dotted] (8,0) --(8,-.75);
\draw [thick, dotted] (0.5,0) --(0.5,-.75);
\draw [thick, dotted] (1.5,0) --(1.5,-.75);
\draw [thick, dotted] (2.5,0) --(2.5,-.75);
\draw [thick, dotted] (3.5,0) --(3.5,-.75);
\draw [thick, dotted] (4.5,0) --(4.5,-.75);
\draw [thick, dotted] (5.5,0) --(5.5,-.75);
\draw [thick, dotted] (6.5,0) --(6.5,-.75);
\draw [thick, dotted] (7.5,0) --(7.5,-.75);
\end{tikzpicture}
\caption{Part of a plane in $X_2$.}
\label{fig:plane}
\end{figure}

Using the metric on the topological planes in $X_n$ inherited from $\R^2$, where a single horizontal segment with label $a$ at height $0$ is defined to have length $1$, it follows that a single horizontal
segment at height $i$ in $X_n$ has length $n^i$.
Thus, if $g=t^{-u}a^vt^w$, the sum of the lengths of the horizontal edges in a path representing $g$ is $n^{-u}v$.  Since any two topological planes in $X_n$ agree beneath a horocycle, that is, in our model a line of the form $y = c$ for $c \in \Z$, each $g = t^{-u}a^vt^w$ has a well defined $x$-coordinate given by $n^{-u}v$, as does the terminal point of any path in the generators $\{a^{\pm 1},t^{\pm 1}\}$.
Let
\begin{equation}
\label{eqn:eta}
\eta=t^{e_0}a^{f_0} t^{e_1}a^{f_1}\cdots t^{e_k} a^{f_k}
\end{equation}
be a path in $X_n$ from the identity to $g = t^{-u}a^vt^w$.
Then each instance of the generator $a^{\pm 1}$ corresponds to a horizontal edge with length $n^h$ where $h$ is the height of the edge in $X_n$.  The endpoint of $\eta$ then has Euclidean $x$-coordinate given by
\[
x=n^{-u}v = \sum_{i=0}^k f_i n^{\sum_{j=0}^ie_j}.
\]

Our approach to finding geodesic words builds on \cite{EH},
where it is shown that any geodesic must have one of
several possible forms.  To each of these forms, we associate
a list of ``digits'' similar to the $f_i$ in Equation~\eqref{eqn:eta} and
related to the horizontal distance traversed by the path.
By carefully considering these
digits, we give conditions which certify that a path of this form
has minimal length.

\section{Representations of integers and geodesic paths}
\label{section:min_rep}

\subsection{The digit lattice}

In order to produce a geodesic representative of a given group
element \hbox{$g = t^{-u}a^vt^w$}, we must understand how to
efficiently represent the integer $v$ as a sequence of signed
digits with bounded absolute value and how to translate this sequence into
a geodesic word.  We formalize the concept of digit sequences
using the direct sum $\bigoplus_{i\in\N} \Z$, where we take the convention that
$0 \in \N$.  Given a vector
$\xx = (x_0, x_1, \dots) \in \bigoplus_{i\in\N}\Z$, 
define the function
$\Sigma:\bigoplus_{i\in\N}\Z \rightarrow \R$ by 
\[
\Sigma(\xx)=\sum_{i \in \N} x_in^{i}.
\]
For any $v \in \Z$, let $\Lv = \Sigma^{-1}(v)$ be the
set of vectors $\xx \in \bigoplus_{i\in\N}\Z$ with
$\Sigma(\xx) = v$.  For any vector $\xx$, we will always write
the coordinates with a matching non-bold letter, e.g. $x_i$, and
we will define
$\kx$ as the index of the final nonzero coordinate, so $\kx+1$ is the ``length'' of
$\x$.
We will refer to the coordinates of vectors as either ``coordinates''
or ``digits''depending on what makes the most intuitive sense
in context. 
For clarity, we will often write
vectors in $\bigoplus_{i\in\N}\Z$ or $\Lv$ as finite
sequences $\xx = (x_0, \dots, x_{\kx})$, assuming that $x_i = 0$
for $i > \kx$ and $x_{\kx} \neq 0$.

Define a set
of vectors $\{ \wi \}_{i \in \N}$ whose coordinates $\wwi_j$ are given by
\[
\wwi_j = \left\{ \begin{array}{ll} 1  & \textnormal{if $j=i+1$} \\
                                   -n & \textnormal{if $j=i$} \\
                                   0  & \textnormal{otherwise}
                 \end{array}\right.
\]
That is,
\[
\wi = (0, \ldots\, , 0\, ,\, \wwi_i, \wwi_{i+1}, 0, \ldots) =
(0, \underset{\ldots}{\ldots}, \underset{i-1}{0}, \underset{i}{-n}, \underset{i+1}{1}, \underset{i+2}{0},\, \underset{\ldots}{\ldots})
\]
where we indicate the index of each entry in the second expression.

\begin{lemma}\label{lemma:L_span}
The set $\mathcal{L}_0$ is a lattice spanned by $\{ \wi \}_{i \in \N}$.
\end{lemma}
\begin{proof}
As $\mathcal{L}_0$ is discrete and clearly closed under addition
and negation, it is a lattice, so we must show
that the $\wi$ span.
Given $\xx \in \mathcal{L}_0$, we will produce a finite linear
combination of the $\wi$ such that $\xx + \sum_{i \in I}\alpha_i\wi$ is
the vector consisting entirely of zeros, which
proves the claim.  

Suppose that $j>0$ is the
maximum index such that $x_j \ne 0$.  Then the
$j^{\textnormal{th}}$ coordinate of $\xx - x_j{\bf w}^{(j-1)}$ will
be $0$.  By induction, we conclude that there is some linear combination
of the $\wi$ such that ${\bf y} = \xx + \sum_i \alpha_i \wi$
has $y_i = 0$ for $i>0$.  However, as
\[
0 = \Sigma({\bf y}) = \Sigma(\xx) = \Sigma(\xx + \sum_i \alpha_i \wi) = y_0
\]
we conclude
that $y_0 = 0$ as well.
\end{proof}

\begin{lemma}\label{lemma:L_addition}
For any $\xx \in \mathcal{L}_v$, we have
$\mathcal{L}_{v+w} = \xx + \mathcal{L}_w$.  In particular,
$\mathcal{L}_v = \xx + \mathcal{L}_0$, so $\mathcal{L}_v$
is an affine lattice.
\end{lemma}
\begin{proof}
It is immediate from the linearity of the function $\Sigma$
that given $\xx \in \mathcal{L}_v$ and $\y \in \mathcal{L}_w$ we have
$\xx + \y \in \mathcal{L}_{v+w}$, so $\mathcal{L}_v + \mathcal{L}_w
\subseteq \mathcal{L}_{v+w}$. 
Given any $\x \in \mathcal{L}_v$ and $\z \in \mathcal{L}_{v+w}$, define $\y$ by
$y_i = z_i - x_i$.  By the linearity of $\Sigma$, we have
$\y \in \mathcal{L}_w$ and $\z = \xx + \y$.  We conclude that for any
$v,w$ we have $\mathcal{L}_{v+w} \subseteq \xx + \mathcal{L}_w$, completing the proof.
\end{proof}

\begin{ex}\label{example:L_v}
As an example, let $n=3$, $v=7$, and
\[
\xx = (\underset{x_0}{7}, \underset{x_1}{0},\, \underset{\ldots}{\ldots}),
\qquad
\y = (\underset{y_0}{1}, \underset{y_1}{2}, \underset{y_2}{0},\, \underset{\ldots}{\ldots}) ,
\qquad
\z = (\underset{z_0}{1}, \underset{z_1}{-1}, \underset{z_2}{1},
\underset{z_3}{0},\, \underset{\ldots}{\ldots}).
\]
Then $\xx,\y,\z \in \Lv$.  Note that $\y = \xx + 2{\bf w}^{(0)}$
and $\z = \y + {\bf w}^{(1)}$.
\end{ex}

As one might expect from the spanning set $\{\wi\}$, the digits of
the vectors in $\LL_0$ have an interesting ordinal relationship.  If the
most significant digit is large in absolute value, there must be a less significant digit with greater absolute value.  This is intuitive: to balance a high power of $n$ with a sum of smaller powers of $n$, we require many smaller powers of $n$.
We formalize this in Lemma~\ref{lemma:L0_big_digits}, which compares
the coefficients of the $\wi$ to the resulting digits that must be
present in their vector sum.

\begin{lemma}\label{lemma:L0_big_digits}
Let $\xx \in \LL_0$ with
$\xx = \sum_i \alpha_i \wi$.
For any $j$ such that $|\alpha_j| \ge m$,
there is $i\le j$ with $|x_i| > m(n-1)$.
\end{lemma}
\begin{proof}
The proofs are symmetric depending on the sign of $\alpha_j$,
so we assume without loss of generality that $\alpha_j > 0$
and hence $\alpha_j \geq m$.  It follows that
$x_j = \alpha_{j-1} - \alpha_j n  \leq \alpha_{j-1} - mn$,
where if $j=0$, we define $\alpha_{j-1}=0$.
It suffices to prove the lemma for the minimal $j$ such that $|\alpha_j| \ge m$,
so we will assume that that inequality holds.
Then $|\alpha_{j-1}|<m$,
so  $x_j +mn  \leq  \alpha_{j-1}<m$.  It follows that
$x_j < -m(n-1)$, as required.
\end{proof}

The following corollary is immediate.

\begin{corollary}\label{corollary:L0_at_least_n}
Let $\xx \in \LL_0$.  If there is an $i$ so that $x_i \ne 0$,
then there is a $j\le i$ with $x_j \ge n$.
\end{corollary}
\begin{proof}
Write $\xx = \sum_i \alpha_i \wi$.
Since $x_i \ne 0$, we must have $\alpha_i \ne 0$ or $\alpha_{i-1} \ne 0$.
Then apply Lemma~\ref{lemma:L0_big_digits} with $m=1$.
\end{proof}

Lemma~\ref{lemma:L0_nonunique_digit} shows that there is no vector in $\LL_0$
with a single nonzero digit.

\begin{lemma}\label{lemma:L0_nonunique_digit}
If $\xx \in \LL_0$ has some nonzero digit, then there must be at least
two nonzero digits in $\xx$.
\end{lemma}
\begin{proof}
Suppose there is $j$ so $x_j$ is the only nonzero digit of $\x$.
Then by the definition of $\LL_0$, we have
\[
0 = \Sigma(\x) = \sum_i x_in^i = x_jn^j.
\]
This contradicts the assumption that $x_j \ne 0$.
\end{proof}

\subsection{Geodesics from digit sequences}
Given a group element  $g = t^{-u}a^vt^w$, we define a map
$\eta_{u,v,w}:\Lv \to \{a^{\pm},t^{\pm}\}^*$ which takes
a vector $\xx = (x_0, \dots, x_\kx) \in \Lv$ to a word
$\eta_{u,v,w}(\xx)$
representing $g$ in the following way:
\[
\eta_{u,v,w}(\xx) = 
\left\{
\begin{array}{lll}
t^{-u} a^{x_0} t a^{x_{1}} \cdots t a^{x_{\kx}}t^{w-\kx}
& \textnormal{if $\kx \leq w$} & \textnormal{(shape 1)} \\
t^{\kx-u} a^{x_{\kx}} \ti a^{x_{\kx-1}} \cdots \ti a^{x_0} t^w
& \textnormal{if $w < \kx \leq u$} & \textnormal{(shape 2)} \\
t^{-u}a^{x_0}ta^{x_{1}} \cdots ta^{x_{\kx}} t^{w-\kx}
& \textnormal{if $u \leq w<\kx$} & \textnormal{(shape 3)} \\
t^{\kx-u}a^{x_{\kx}}\ti a^{x_{\kx-1}} \cdots  \ti a^{x_{0}} t^{w}
& \textnormal{if $w<u<\kx$} & \textnormal{(shape 4).} 
\end{array}\right .
\]
Shapes 1 and 3 and shapes 2 and 4 have identical expressions up to the signs of certain exponents.

\begin{remark}
We only consider triples $(u,v,w)$ which correspond to elements $g = t^{-u}a^vt^w$ in normal form.  In particular, we only allow $n|v$ if $uw=0$, which places restrictions on these triples.
This has implications about which vectors $\x \in \Lv$ we consider.
If the first digit of $\x$ is $0$, then $n|v$, and for any choice of $u,w \geq 0$ we require that either $u=0$ or $w=0$.
\end{remark}

We now show that the length of each path above is given by one
of two expressions.  Here, $| \cdot |$ denotes the actual length of the
given path, not the word length with respect to the generating set $\{a^{\pm 1},t^{\pm 1} \}$ in the group of the element it represents.

\begin{lemma}\label{lemma:length_formula}
For $\xx = (x_0, \dots, x_\kx) \in \Lv$ and $u,w \geq 0$ we have
\[
|\eta_{u,v,w}(\xx)| = 
\left\{
\begin{array}{lll}
\Vert \xx \Vert_1 + u + w & \textnormal{if $\kx \leq \max(u,w)$} & \textnormal{(shapes 1 and 2)} \\ 
\Vert \xx \Vert_1 + 2\kx - |u-w| & \textnormal{otherwise} & \textnormal{(shapes 3 and 4)}
\end{array}\right. .
\]
\end{lemma}
\begin{proof}
To prove the lemma, we add the absolute values of the exponents in the above expressions for $\eta_{u,v,w}(\xx)$.  Accounting for the signs of the expressions, the first formula follows immediately.

For the second case, we compute the length of a path of shape 3:
\[
|\eta_{u,v,w}(\xx)| = \Vert \xx \Vert_1 + u + \kx + \kx - w
\]
and for shape 4:
\[
|\eta_{u,v,w}(\xx)| = \Vert \xx \Vert_1 + (\kx-u) +  \kx + w.
\]
Considering the relative magnitudes of $u,w$ and $\kx$, we see that the two expressions combine into the second formula of the lemma.
\end{proof}

In Lemma~\ref{lemma:length_formula}, the two cases divide according
to whether $\kx \le \max(u,w)$ or $\kx > \max(u,w)$.  We remark that
the formulas agree in the boundary case when
$\kx = \max(u,w)$.  To see this, note that when $\kx=\max(u,w)$,
\begin{align*}
2\kx - |u-w| & = 2\max(u,w) - |u-w| \\
           & = 2\max(u,w) - \max(u,w) + \min(u,w) \\
           & = u + w
\end{align*}

If we have a vector $\x \in \Lv$ and add $\z \in \LL_0$, both the
$\ell^1$ norm and the vector length may change.  
Choose $u,w \geq 0$.  If $\kx$ and $k_{\x+\z}$ have different ordinal relationships to $\max(u,w)$,
computing $|\eta_{u,v,w}(\xx)|$ and $|\eta_{u,v,w}(\xx+\x)|$ may require different formulas from Lemma~\ref{lemma:length_formula}.
The next lemma explicitly computes this change.

\begin{lemma}\label{lemma:length_formula_change}
Let $\x \in \Lv$ and $\z \in \LL_0$ with $u,w \geq 0$. If $k_{\x + \z} > \kx$, then
\[
|\eta_{u,v,w}(\x+\z)| - |\eta_{u,v,w}(\x)| =
\Vert \x + \z \Vert_1 - \Vert \x \Vert_1 + 2\max(0, k_{\x+\z} - \max(\kx,u,w)).
\]
If $k_{\x + \z} < \kx$, then
\[
|\eta_{u,v,w}(\x+\z)| - |\eta_{u,v,w}(\x)| =
\Vert \x + \z \Vert_1 - \Vert \x \Vert_1 - 2\max(0, \kx - \max(k_{\x+\z},u,w)).
\]
\end{lemma}

\begin{proof}
If both $\kx \leq \max(u,w)$ and $k_{\x + \z} \leq \max(u,w)$, or both $\kx > \max(u,w)$ and $k_{\x + \z} > \max(u,w)$, then we use the same
length formula from Lemma~\ref{lemma:length_formula} to compute both $|\eta_{u,v,w}(\x+\z)|$ and $|\eta_{u,v,w}(\x)|$.
In the first case, the formulas in the lemma follow from the fact that $\max(0, \kx - \max(k_{\x+\z},u,w)) = 0$.
In the second case, the maximum is either $\kx -k_{\x+\z}$ or its negative, again giving rise to the two formulas in the lemma.

We consider the remaining cases.  
Suppose that $\kx \le  \max(u,w) \le k_{\x + \z}$ and one of the inequalities is strict.
Thus we use the first length formula from Lemma~\ref{lemma:length_formula} to compute $|\eta_{u,v,w}(\x)|$, and the second length formula from Lemma~\ref{lemma:length_formula} to compute $|\eta_{u,v,w}(\x+\z)|$.
It follows that
\begin{align*}
|\eta_{u,v,w}(\x+\z)| - |\eta_{u,v,w}(\x)| &= \Vert \x + \z \Vert_1 - \Vert \x \Vert_1 
+2k_{\x + \z} -|u-w| - u - w\\
&= \Vert \x + \z \Vert_1 - \Vert \x \Vert_1 
+2k_{\x + \z} -2\max(u,w). \\
\end{align*}
Since $\kx \le  \max(u,w)$, it follows that $\max(u,w) = \max(\kx,u,w)$, we have the desired formula in this case.

Suppose that $k_{\x + \z} \le  \max(u,w) \le \kx$ and one of the inequalities is strict. 
Using the appropriate length formula from Lemma~\ref{lemma:length_formula}, we then compute
\begin{align*}
|\eta_{u,v,w}(\x+\z)| - |\eta_{u,v,w}(\x)| &= \Vert \x + \z \Vert_1 - \Vert \x \Vert_1 
+u+w -2\kx + |u-w|\\
&= \Vert \x + \z \Vert_1 - \Vert \x \Vert_1 
-2\kx +2\max(u,w) \\
\end{align*}
and again we have the desired formula, completing the proof.
\end{proof}

\begin{lemma}\label{lemma:minimal_is_geodesic}
If $\xx \in \Lv$ is such that $|\eta_{u,v,w}(\xx)|$ is minimal,
then $\eta_{u,v,w}(\xx)$ is a geodesic representing the
group element $g = t^{-u}a^vt^w$.
\end{lemma}
\begin{proof}
This is a corollary of~\cite{EH}, Proposition~2.3, where it is shown that
there must be a geodesic representing $g$ which has one of shapes $1$--$4$.
For the geodesic $\xi$ guaranteed by~\cite{EH}, there is a vector ${\bf y} \in \Lv$ so that $\eta({\bf y}) = \xi$.

For a given $\xx \in \Lv$, one can form two possible induced paths which are potentially geodesic: one in shape $1$ or $3$, depending on whether $w<\kx$, and one in shape $2$ or $4$.  The word length formulas in Lemma~\ref{lemma:length_formula} allow us to choose the shorter of these paths as the output of $\eta_{u,v,w}$, and hence $\eta_{u,v,w}$ describes a geodesic path to $g$.
\end{proof}

If we are given $g = t^{-u}a^vt^w$ and want to find a geodesic
for $g$, then by Lemma~\ref{lemma:minimal_is_geodesic}, it suffices
to find a vector $\xx \in \Lv$ such that 
$|\eta_{u,v,w}(\xx)|$ is minimal; we will refer to such an
$\xx$ as a minimal vector.  By Lemma~\ref{lemma:L_addition},
this is equivalent to minimizing 
$|\eta_{u,v,w}(\xx+\z)|$, where $\xx \in \Lv$ is any
vector and $\z \in \mathcal{L}_0$.  Lemma~\ref{lemma:reduce_to_box} shows that some vectors $\xx$ are easily altered in this way to reduce $|\eta_{u,v,w}(\xx)|$.  We will refer to the change from $\x$ to $\x+\z$ in this way as reducing $\x$.

For $n \geq 3$, let $\Bvuw \subseteq \Lv$ be defined to be the set of 
$\xx = (x_0, \dots, x_\kx) \in \Lv$ satisfying the following conditions.

\begin{enumerate}
\item If $i<\kx$, then $|x_i| \le \fn$.
\item If $i=\kx < \max(u,w)$, then $|x_i| \le \fn$.
\item If $i=\kx \ge \max(u,w)$, then $|x_i| \le \fn + 1$.
\end{enumerate}
When $n=2$, we define $\Bvuw$ as above, replacing the third inequality with $|x_i| \le \fn + 2$.

The ``box'' $\Bvuw$ contains all digit sequences whose digits are
uniformly bounded as described above; most digits are bounded by $\fn$, but when $\kx \geq \max(u,v)$ we allow the most significant digit to be slightly larger.
For context, our plan is to find minimal vectors in $\Bvuw$, and the
modified bound on the final digit results from shortening $|\eta_{u,v,w}(\x)|$
in certain cases where it is more efficient to have one larger digit than
two smaller digits at the end of the vector.

In Lemma~\ref{lemma:reduce_to_box} below we show that given $\x \in 
\Lv$, we can find a vector $\y \in \Bvuw \subseteq \Lv$ so that
$|\eta_{u,v,w}(\y)| \le |\eta_{u,v,w}(\xx)|$.
Since $\eta_{u,v,w}(\xx)$ and $\eta_{u,v,w}(\y)$
represent the same group element, this implies that finding a geodesic for a
group element is equivalent to searching for a minimal vector within $\Bvuw$.  For such $\x$ and $\y$, we will
write $\y \ord \x$ to mean
that $|\eta_{u,v,w}(\y)| \le |\eta_{u,v,w}(\xx)|$.  Note that
although $\ord$ is transitive, it is not a partial order because it is
not antisymmetric.  However, it still makes sense to refer to vectors
as being minimal with respect to the relation.

\begin{lemma}\label{lemma:reduce_to_box}
If $\x \in \Lv$ but $\xx  \notin \BB_v^{u,w}$, then there exists
$\z \in \mathcal{L}_0$ so that $\xx + \z \in \BB_v^{u,w}$
and 
\[
\xx + \z \ord \xx.
\]
Consequently, if $\xx \in \Bvuw$ is minimal, then $\eta_{u,v,w}(\xx)$
is geodesic.
\end{lemma}

\begin{proof}
Let $\x \in \Lv$.  There are two conditions which might be violated to imply
$\xx \notin \Bvuw$.  In both cases, we will examine the minimal $i$
such that $|x_i|$ contradicts a defining condition of $\Bvuw$ and decrease this coordinate without affecting any $x_j$ with
$j < i$.  Combined with control over the length of
$\eta_{u,v,w}(\xx)$ as we do this, the lemma the follows
by induction on $i$.

First suppose conditions (1) or (2) of membership in $\Bvuw$
are violated for some minimal index $i$.  That is, $|x_i| > \fn$.
Without loss of generality, assume that $x_i>0$.
Let $\z = \wi$ and consider $\y = \xx+\z = \xx + \wi$.  Then $y_i = x_i -n$, so
\[
|y_i| \le |x_i| - 1,
\]
and $y_{i+1} = x_{i+1} + 1$, so
\[
|y_{i+1}| \le |x_{i+1}| + 1.
\]
Thus $\Vert \y \Vert_1 \le \Vert \x \Vert_1$.
The assumption that $\kx < \max(u,w)$ implies
that both $|\eta_{u,v,w}(\y)|$ and $|\eta_{u,v,w}(\xx)|$ are
computed using the first formula in Lemma~\ref{lemma:length_formula}, that is, for shapes 1 and 2, so
$|\eta_{u,v,w}(\y)| \le |\eta_{u,v,w}(\xx)|$, that is $\y \ord \x$.

Now suppose $n \geq 3$ and the third condition of membership in $\Bvuw$ is violated, so $|x_i| \ge \fn + 2$
where $i = \kx \ge \max(u,w)$.  Assume without loss of
generality that $x_i > 0$.
Let $\z = \wi$ and consider $\y = \xx+\z=\xx + \wi$.  Since $i=\kx$,
we have $x_{\kx+1}=0$ and $y_{\kx+1} = 1$, so $k_\y = k_\xx + 1$, and
we use the length formula from Lemma~\ref{lemma:length_formula} for shapes 3 and 4 to compute
\begin{align*}
|\eta_{u,v,w}(\y)| &= |\eta_{u,v,w}(\xx)| +
                     (\Vert \z \Vert_1 - \Vert \xx \Vert_1) + 2(\kx-{\bf k}_y) \\
                     &=|\eta_{u,v,w}(\xx)| +
                     |y_{\kx}| - |x_{\kx}|+ |y_{\kx +1}| + 2(\kx-{\bf k}_y) \\
                   &= |\eta_{u,v,w}(\xx)| +
                     |y_{\kx}| - |x_{\kx}| + 3.
\end{align*}
We know that $y_{\kx} = x_{\kx} -n$.  If $x_{\kx} \ge n$, then
$|y_{\kx}| = |x_{\kx}| - n$. Otherwise, 
$|y_{\kx}| \le |x_{\kx}| - 4$ if $n$ is even and
$|y_{\kx}| \le |x_{\kx}| - 3$ if $n$ is odd.  In both cases,
$|\eta_{u,v,w}(\y)| \le |\eta_{u,v,w}(\xx)|$, that is $\y \ord \x$.
Note that $k_{\y} = \kx+1$, but as $\kx$ is the maximal index in $\x$ 
we know that all digits of $\y$ now satisfy the bounds for $\Bvuw$, and the induction stops.  

There remains the special case of violating the third condition of membership in $\Bvuw$
when $n=2$.  If $|x_{\kx}| \ge \fn + 3 = 4$, let $\z =  2 {\bf w}^{(\kx)}$ consider
$\y = \xx+\z = \xx + 2 {\bf w}^{(\kx)}$.  
Again, $\y \in \Bvuw$ and $k_\y = \kx+1$, so the induction stops.
An analogous calculation  shows that
$|\eta_{u,v,w}(\y)| \le |\eta_{u,v,w}(\xx)|$, that is $\y \ord \x$.
\end{proof}

\begin{ex}\label{example:BS12_digit_bound}
The definition of $\Bvuw$ has a special case for $n=2$.  Here we give
an example to show that it is necessary.  Let $u=w=0$ and $v=7$.  Define
\[
\xx = (\underset{x_0}{1}, \underset{x_1}{3}),
\qquad \textnormal{ and } \qquad
\y = (\underset{y_0}{1}, \underset{y_1}{1}, \underset{y_2}{1}) = \xx + {\bf w}^{(1)}.
\]
Note that $\xx,\y \in \LL_7$, and we can compute
$|\eta_{0,7,0}(\xx)| = 6$ and
$|\eta_{0,7,0}(\y)| = 7$.  That is, $\x \ord \y$.  Using the
digit bound of $\fn+1$ in the definition of $\BB_7^{0,0}$ for
$n=2$ therefore cannot be correct.  By enumerating all the vectors
in $\BB_7^{0,0}$, we could check that $\x$ is actually minimal in
$\BB_7^{0,0}$ and it would follow from 
Lemma~\ref{lemma:reduce_to_box} that $\eta_{0,7,0}(\x)$ is geodesic.
In Section~\ref{section:geodesics_n_even_2}, we will give
some simple conditions to certify that $\x$ is minimal without requiring
this enumeration.
\end{ex}

If we consider two vectors $\x,\y \in \Bvuw$ which differ by a sum of $\LL_0$ basis vectors, we retain some control over the lengths of $\x$ and $\y$.

\begin{lemma}\label{lemma:no_L0_parts}
Let $\x,\y \in \Bvuw$ with $\ky \ge \kx$ and $\y-\x = \sum_{i=j}^\ell \alpha_i\wi$,
where $\alpha_j \ne 0$ and $\alpha_\ell \ne 0$.
Then $\kx \ge j$ and $\ky > \ell$.
\end{lemma}
\begin{proof}
We first prove that $\ky > \ell$.  Because $y_{\ell+1} - x_{\ell+1} = \alpha_{\ell}$,
at least one of $y_{\ell+1}$, $x_{\ell+1}$ must be nonzero.  Combined with the
hypothesis that $\ky \ge \kx$, we have $\ky > \ell$.

To prove that $\kx \ge j$, observe that because $\ky > \ell \ge j$, we know $|y_j| \le \fn$.
Since $x_j - y_j = -\alpha_j n$, we see that $x_j \ne 0$.  Therefore $\kx \ge j$.
\end{proof}

For the remainder of this paper, when we write $\y = \x+ \sum_{i=j}^{\ell} \alpha_i \wi$ we will always assume that $\alpha_j \neq 0$ and $\alpha_\ell \neq 0$.

\subsection{Minimal vectors for $n$ odd}
\label{section:geodesics_n_odd}

Let $g = t^{-u}a^vt^w \in BS(1,n)$ for $n$ odd.
Lemma~\ref{lemma:reduce_to_box} shows that $\Bvuw$ is nonempty
and that if $\xx \in \Lv$ is a
minimal vector in $\Bvuw$, then $\eta_{u,v,w}(\xx)$ is geodesic for $g$.
We now show that when $n$ is odd, the set $\Bvuw$ contains at most two vectors.

\begin{lemma}\label{lemma:odd_box}
Let $n \geq 3$ be odd and $\xx \in \Bvuw$.
\begin{enumerate}[itemsep=5pt]
\item If $k_\xx < \max(u,w)$, then $|\Bvuw| = 1$.
\item If $k_\xx \ge \max(u,w)$, then $|\Bvuw| \le 2$.
If $|\Bvuw| = 2$, then $\Bvuw$ has the form
\[
\Bvuw = \{\xx, \xx + \epsilon{\bf w}^{(k_\xx)}\},
\]
where $\epsilon \in \{-1,1\}$.  Moreover, $\y\in\Bvuw$ is not minimal
if and only if $\ky > \max(u,w)$ and the final digits of $\y$ are
$(\delta\fn, -\delta)$, 
where $\delta \in \{\pm 1\}$.
\end{enumerate}
\end{lemma}

\begin{proof}
Suppose that $|\Bvuw| \ge 2$. 
Let $\xx \in \Bvuw$ be of minimal length, 
and let $\y \in \Bvuw$ be any other vector.
In particular, $\kx \le \ky$.  Set
$\z = \y - \xx = \sum_i \alpha_i \wi \in {\mathcal L}_0$.
It follows from Lemma~\ref{lemma:L0_big_digits} that there is some minimal $j$ with $|z_j| \geq n$, and that $0 \leq j < \ky$.

If $j < \kx \le \ky$, then $|x_j|,|y_j|\le \fn = \frac{n-1}{2}$, and the maximum difference between $x_j$ and $y_j$ is $2 \fn < n$ because $n$ is odd.
Thus we
cannot have $|z_j| = |x_j - y_j| \ge n$, which contradicts Corollary~\ref{corollary:L0_at_least_n}.
We conclude that $j \ge \kx$.

If $j > \kx$, we have $z_{j'} = y_{j'}$ for $j' \geq j$.  It follows that
$|z_{j}| = |y_{j}| \geq n$; this bound on $|y_j|$ violates the
definition of $\Bvuw$, hence this does not occur.

It remains to consider $j = \kx$.  Note that if $\kx < \max(u,w)$ then
$|x_j|,|y_j| \le \fn$ and it is impossible to have $|x_\kx - y_\kx| \ge n$.
So we must have $\kx \ge \max(u,w)$.  In this case, we find exactly
one additional vector in $\Bvuw$.  As $\xx$ and $\y$ lie in $\Bvuw$,
we know that $|x_{\kx} - y_{\kx}| \leq n$, and without loss
of generality we assume that $x_{\kx} > 0$. The only way that the
inequality $|x_\kx - y_\kx| \ge n$ can be satisfied is if 
$x_{\kx} = \fn+1$ or $y_\kx = -(\fn + 1)$.
\begin{itemize}
    \item If $x_{\kx} = \fn+1$, then $y_{k_x}=-\fn$ and $\alpha_j = 1$.
    Thus, $y_{\kx+1} = 1-\alpha_{\kx+1}n$, so $\alpha_{\kx+1} = 0$.  We similarly
    conclude that $\alpha_{j'} = 0$ for all $j'>\kx$.  So $\y = \x + \wf{\kx}$.
    \item If $y_\kx=-(\fn+1)$, then $\ky = \kx$ and $x_\kx = \fn$.  As in
    the previous bullet, $y_{\kx+1} = 1$.  This is impossible, though, because
    $\ky=\kx < \kx + 1$.
\end{itemize}
Thus we have shown that if $|\Bvuw| \ge 2$, then $\kx \ge \max(u,w)$
and $\Bvuw = \{\x, \x + \wf{\kx}\}$, so $|\Bvuw| = 2$.
We have also shown
that this case only occurs in the first bullet above; in this case the we have $\ky = \kx+1$ and it follows from the second length formula in Lemma~\ref{lemma:length_formula} that
$|\eta_{u,v,w}(\y)| = |\eta_{u,v,w}(\x)| + 2$.
This proves the final statement of the lemma.
\end{proof}

Lemma~\ref{lemma:odd_box} shows that $\Bvuw$ can contain at most
two vectors, and contains two vectors only if $\kx \ge \max(u,w)$
for some $\x \in \Bvuw$.  The statement in Lemma~\ref{lemma:odd_box}
is not an ``if and only if''; it is possible that
$\kx \ge \max(u,w)$ and $|\Bvuw|=1$.

Lemma~\ref{lemma:odd_box} implies that finding a geodesic representative
of $g = t^{-u}a^vt^w$ in $B(1,n)$ when $n$ is odd is actually quite
straightforward: find any vector in $\Lv$, reduce its digits so that it lies in $\Bvuw$, and check its most significant digits to ascertain minimality.

An immediate consequence of Lemma~\ref{lemma:odd_box} is that when $n$ is odd, $\ord$ is a total order on $\Bvuw$.

\subsection{Minimal vectors for $n$ even}
\label{section:geodesics_n_even}

Let $g = t^{-u} a^v t^w \in BS(1,n)$ for $n$ even. In this case, the
set $\Bvuw$ can contain many more vectors, as well as multiple minimal vectors.
In order to choose a unique minimal vector in $\Bvuw$, we redefine $\xx \le_{u,w} \y$
for $n$ even so that it is a total order on $\Bvuw$.

For $n$ even, let $|\xx|$ and $|\y|$ denote
the vectors of the absolute values of the coordinates of $\xx$ and $\y$,
respectively, and define 
$\xx \ord \y$ if and only if
\begin{itemize}[itemsep=5pt]
    \item $|\eta_{u,v,w}(\y)| < |\eta_{u,v,w}(\xx)|$, or
    \item  $|\eta_{u,v,w}(\y)| = |\eta_{u,v,w}(\xx)|$ and
$|\xx| \leq |\y|$ in the lexicographic order that ranks lower
indexed coordinates as more significant.
\end{itemize}
The relation is strict, that is, $\xx \ords\y$ if and only if
$|\eta_{u,v,w}(\y)| < |\eta_{u,v,w}(\xx)|$ or $|\xx| < |\y|$.
Since the relation may depend on the absolute values of the coordinates of $\x$ and $\y$, it is not {\em a priori} the case
that a minimal vector is unique, or even that the relation given
is an order, which necessitates Lemma~\ref{lemma:abs_lex_unique}.
As $n$ is even, we will write
$\fne$ in place of $\fn$ for simplicity for the duration
of Section~\ref{section:geodesics_n_even}.

\begin{lemma}\label{lemma:abs_lex_unique}
Let $n$ be even.  For any $u,w \in \N$ and $v \in \Z$, the relation $\ord$ is a
total order on $\Bvuw$.
\end{lemma}
\begin{proof}
Let $\xx,\y \in \Bvuw$ be given.  If
$|\eta_{u,v,w}(\y)| \ne |\eta_{u,v,w}(\xx)|$ then $\xx \ords \y$ or
$\y \ords \xx$ and we are done.  Otherwise, the relation is determined
by the lexicographic order of $|\xx|$ and $|\y|$.  Since
the lexicographic order is a total order, the only way for
$\xx \ord \y$ and $\y \ord \xx$ to both hold is if $|\xx| = |\y|$.

Suppose this is the case, so $|\xx| = |\y|$ and thus 
$|x_i| = |y_i|$ for all $i$.
In particular,
$x_i - y_i$ must be even.
Since $\xx - \y \in \LL_0$,
we can write $\xx - \y = \sum_i \alpha_i\wi$.  Let $j$ be the maximal index
such that $\alpha_j \ne 0$.  Then $x_{j+1} - y_{j+1} = \alpha_j$,
so $\alpha_j$ must be even.  By assumption, $\alpha_j \ne 0$, 
so $|\alpha_j| \ge 2$.  It follows from Lemma~\ref{lemma:L0_big_digits}, 
there is some $\ell \le j$ with $|x_\ell - y_\ell| > 2(n-1) = 2n-2$,
and as $x_\ell - y_\ell$ is even, we have $|x_\ell - y_\ell| \ge 2n$.
The largest possible digits in $|\xx|$ and $|\y|$ are $\fne + 1$ (or
$\fne +2$ if $n=2$), which can only occur at index, respectively, $\kx$ or $\ky$.  
However, $\ell \neq \kx$ and $\ell \neq \ky$ because $x_{j+1} = -y_{j+1} \neq 0$, and $\ell < j+1$.
It follows that $|x_\ell|,|y_\ell| \le \fne$,
so $|x_\ell - y_\ell| \le n$, contradicting our earlier inequality $|x_\ell - y_\ell| > 2n-2$.
We conclude that there is no coordinate
in which $\xx$ and $\y$ differ, so $\xx = \y$.
\end{proof}

\begin{ex}\label{example:absolute_order}
As an example, let $n=4$, $v=26$,
and $u,w \ge 3$
and
\[
\xx = (\underset{x_0}{2}, \underset{x_1}{2}, \underset{x_2}{1}, \underset{x_3}{0},
 \, \underset{\ldots}{\ldots}),
\qquad
\y = (\underset{y_0}{-2}, \underset{y_1}{-1}, \underset{y_2}{2}, \underset{y_3}{0},
\, \underset{\ldots}{\ldots}).
\]
Then
\[
|\eta_{u,v,w}(\xx)| \; = \; \Vert \xx \Vert_1 + u+ w \; = \; 5 + u+ w
\; = \; \Vert \y \Vert_1 + u+ w \; = \; |\eta_{u,v,w}(\y)|,
\]
so the word lengths $|\eta_{u,v,w}(\xx)|$ and $|\eta_{u,v,w}(\y)|$ are equal,
but $\y \ords \xx$ in the absolute lexicographic order.  We will see
in Example~\ref{example:absolute_minimal} that both $\xx$ and $\y$ are minimal,
but $\y$ is the unique lexicographically minimal vector in $\Bvuw$.
\end{ex}

Although the question of minimality is more complicated for even $n$,
there are relatively simple conditions which allow us to determine whether $\x \in \Bvuw$ is minimal, and if not, to find $\y \in \Bvuw$ with $\y \ord \x$.
We now give a brief overview of this strategy, with precise details included in the lemmas below.
Recall that for any vector $\xx \in \Bvuw$, there is $\z \in \LL_0$
such that $\xx + \z \in \Bvuw$ is minimal, and we can write $\z$ as a 
linear combination of $\wi$.  
Disregarding the most significant digit of $\x$, we must have $|x_j| \le \fne$.  
If $\wi$ is the lowest indexed basis vector in $\z$, so $|z_i|  \geq  n$, we must have $|x_i| = \fne$ and $|z_i| = n$ to ensure that $\xx + \z \in \Bvuw$.
That is, potential reductions can only occur
when if $\x$ contains the digit $\pm\fne$.  By examining the digits
of $\x$ which follow this initial $\pm\fne$, we can determine whether the original
vector is minimal.

For the remainder of this section, we consider only $n \geq 4$ and prove analogous results for $n=2$ in  Section~\ref{section:geodesics_n_even_2}.

Let $\x \in \Bvuw$ and define a \emph{run} $\r \subseteq \x$  to be a sequence
of consecutive digits $\r = (x_j, \dots, x_\ell)$ such that
$|x_j| = \fne$ and $|x_i| \in \{\fne-1, \fne, \fne +1\}$
for all $j < i \le \ell$, and the sign $\sgn(x_i)$ is constant
for all $j \le i \le \ell$.  We denote this sign by
$\epsilon_\r =\sgn(x_j)$.  
We retain the indexing of the coordinates of $\r$ from $\x$, that is, the ``first" coordinate of $\r$ is $x_j$ rather than $r_0$ for clarity.
We remark that the digit bounds defining $\Bvuw$ imply that if
$|x_i| = \fne + 1$, then in fact $i=\ell=\kx$. 
The \emph{length} of the run is $\ell - j +1$.
We focus below on understanding possible runs contained in a vector $\x \in \Bvuw$; adding an appropriate linear combination of basis vectors $\wi$ to a run yields a vector $\y$ which may satisfy $\y \ords \x$.

Define the \emph{weight} of a run $\r$ to be
\[
\wt(\r) = 3\#\left\{\fne+1\right\} + \left(\#\left\{\fne\right\}-1\right) 
- \#\left\{\fne -1\right\}.
\]
That is, three times the number of occurrences
of the digit $\epsilon_\r(\fne+1)$ in the run, which is either $0$ or $1$, plus one less than the number
of occurrences of the digit $\epsilon_\r\fne$, minus the number of occurrences of the digit $\epsilon_\r(\fne-1)$ in the run.
This
rather strange formula will capture the change in
$\Vert\x\Vert_1$ which arises from adding a linear combination of $\wi$ to
the run $\r$.

\begin{lemma}\label{lemma:change_formula}
Let $n\ge 4$ be even.
Let $\x \in \Bvuw$ contain a run $\r = (x_j, \dots, x_\ell)$.
Let $\y = \x + \epsilon_\r\sum_{i=j}^\ell\wi$.  Then
$\Vert\y\Vert_1 = \Vert \x \Vert_1 - \wt(\r)
+ |x_{\ell+1} + \epsilon_\r| - |x_{\ell+1}|$.
\end{lemma}
\begin{proof}
The proof is just the computation of the change in absolute
value for each digit $x_j, \dots, x_{\ell+1}$.  
Note that the largest indexed basis vector in the above sum is ${\bf w}^{(\ell)}$, and the thus the digits of $\x$ affected by this sum are $x_j, \cdots ,x_{\ell+1}$.
We have
$|y_j| = |x_j|$, and for $j < i \le \ell$,
$|y_i| = |x_i - \epsilon_\r(n + 1)|$, so
\begin{itemize}
\item If $|x_i| = \fne -1$, then $|y_j| = \fne$.
\item If $|x_i| = \fne$, then $|y_j| = \fne - 1$.
\item If $|x_i| = \fne + 1$, then $|y_j| = \fne -2$.
\end{itemize}
In all cases, this change is accounted for by $\wt(\r)$; counting
one fewer instance of $\fne$ is necessary to account for the fact that
the absolute value of $x_j$ does not change.  All that remains
is to account for the difference between $|y_{\ell+1}|$ and
$|x_{\ell+1}|$, which is the final part of the expression.
\end{proof}

The search for a minimal vector is simplified if we are able to consider adding only linear combinations of basis vectors with no non-zero coefficients to $\x \in \Bvuw$.  The following lemma proves that this is sufficient, and relies on the fact that  our lexicographic order treats lower-index digits as more significant.  
Lemma~\ref{lemma:single_run_from_multiple_runs} shows that if $\y$ is a minimal vector in $\Bvuw$ obtained from $\x$ by adding two sums of $\LL_0$ basis vectors whose index sets are separated from each other, then adding only the sum with smaller indices will produce a vector which precedes $\x$ in the order $\ords$.

\begin{lemma}\label{lemma:single_run_from_multiple_runs}
Let $n\ge 2$ be even.  Let $\x,\y \in \Bvuw$, and let $\y$ be minimal.
Suppose that we have
\[
\y = \x + \sum_{i=j}^\ell \alpha_i\wi + \sum_{i>\ell+1} \alpha_i \wi.
\]
That is, $\y$ is obtained from $\x$ by adding a linear combination of $\wi$,
where $\wf{\ell+1}$ is omitted from the sum.  Then
\[
\x +  \sum_{i=j}^\ell \alpha_i\wi \ords \x.
\]
Furthermore, if $\kx \le \ell+1$, then $\alpha_i = 0$ for $i > \ell$.
\end{lemma}
\begin{proof}
We first prove the last claim of the lemma.  Suppose $\kx \le \ell+1$, and 
$\alpha_j \ne 0$ for some $j > \ell+1$, that is, suppose the right summand is nonzero.
Then $y_j = -\alpha_jn$ and $\ky > j$.  For any value of $n$, this is not possible in $\Bvuw$.

The idea of the proof is that because $\wf{\ell+1}$ does not appear in the linear
combination of vectors, the effects of the two summands
on both the $\ell^1$ vector norm and the lexicographic order are independent.

Let $\ba = \sum_{i=j}^\ell \alpha_i \wi$ and
$\bb = \sum_{i>\ell+1} \alpha_i \wi$, so $\y = \x + \ba + \bb$.
The lemma follows immediately if $\bb=0$, that is, $\bb$ is an empty sum.
Otherwise, $\bb$ is nonzero and it follows from Lemma~\ref{lemma:no_L0_parts} that $\kx,k_{\x+\bb},\ky > \ell+1$.  
Therefore, adding $\ba$ does not affect
the length of $\x$ or $\x + \bb$, although it might be the case that $\kx \ne k_{\x + \bb}$.
As $\ky = k_{\x+\bb}$, we use the same word length formula from Lemma~\ref{lemma:length_formula} to compute the lengths of each pair of geodesics which are compared below.
Therefore we have
\begin{align*}
|\eta_{u,v,w}(\y)| - |\eta_{u,v,w}(\x+\bb)| & = \Vert \y \Vert_1 - \Vert \x + \bb \Vert_1 = \sum_{i=j}^{\ell+1} |x_i + a_i| - |x_i| \\
|\eta_{u,v,w}(\x+\ba)| - |\eta_{u,v,w}(\x)| & = \Vert \x+\ba \Vert_1 - \Vert \x \Vert_1 = \sum_{i=j}^{\ell+1} |x_i + a_i| - |x_i|. \\
\end{align*}
Because $\y$ is minimal, the first difference is at most zero.  
As the rightmost terms in each set of equations above are equal, we conclude that $|\eta_{u,v,w}(\x+\ba)| \le |\eta_{u,v,w}(\x)|$ as well.
If the inequality if strict, it follows immediately that $\x+\ba \ords \x$.  If there is equality,
as Lemma~\ref{lemma:abs_lex_unique} proves that $\ords$ is a total order, there must be
some lexicographic difference between $\x$ and $\x + \ba$.  
An increase in lexicographic order would contradict the minimality of $\y$, as it would follow that
$\x + \bb \ords \x + \ba + \bb = \y$.  
We conclude that
$\x + \ba \ords \x$, completing the proof.
\end{proof}

When the weight of a run is positive, we have additional control over the digits of $\r \subseteq \xx$.
\begin{lemma}\label{lemma:adjacent_digits_from_weight}
Let $n\ge 4$ be even.  If $\r$ is a run in $\x \in \Bvuw$ and
$\wt(\r) > 0$, and
$\r$ does not contain a digit with absolute value $\fne + 1$, then
$\r$ contains a pair of adjacent digits $\fne$.
\end{lemma}
\begin{proof}
By the definition of weight, if $\wt(\r) > 0$, then we must have at least
two more digits with absolute value $\fne$ than $\fne-1$.
If we arrange a set of digits of the form $\fne$ and $\fne-1$ with a surplus of at least two $\fne$ digits,
we are forced to place two digits $\fne$ adjacent to each other.
\end{proof}

If $\x \in \Bvuw$ and $\r= (x_j, \dots, x_\ell)$ is a run in
$\x$, then we say that
$\x$ can be {\em reduced at} $\r$ if 
\[
\y \ords \x \in \Bvuw
\]
where $\y = \x + \epsilon_\r \sum_{i=j}^\ell\wi$.
In order to determine whether a vector $\x$ can be reduced at
$\r$, we use Lemma~\ref{lemma:change_formula} combined with
conditions on
$\wt(\r)$,  the change in absolute value of the
digit $x_{\ell+1}$, the lexicographic change, and, if $\kx < \max(u,w)$, the change in
the length of $\x$.  
Generally, if a vector can be reduced,
it can be reduced at a run.  There is a special case which does not follow this rule, given in the following lemma.

\begin{lemma}\label{lemma:even_reduction_end}
Let $n \ge 4$ be even, and $\xx \in \Bvuw$. If
$k_\xx > \max(u,w)$ and
the final digits of $\x$ are $(\delta (\fne-1), -\delta)$ 
or $(\delta \fne, -\delta)$, where $\delta \in \{-1,1\}$,
then $\x$ is not minimal.
\end{lemma}
\begin{proof}
The cases for $\delta$ are symmetric, so we assume without loss of generality that $\delta = 1$. For either sequence of digits, consider
$\y = \xx + \wf{k_\xx-1}$.
We have $\Vert \y \Vert_1 \le \Vert \xx \Vert_1 + 1$.
Note that $\ky = \kx -1$,
and by the assumption on $\kx$, we use the length formula in Lemma~\ref{lemma:length_formula} for paths of shape 3 and 4 to
compute both $|\eta_{u,v,w}(\x)|$ and $|\eta_{u,v,w}(\y)|$.
Therefore, 
\[
|\eta_{u,v,w}(\y)| \le \Vert \x \Vert_1 +1 +2(\kx-1) - |u-w| = |\eta_{u,v,w}(\x)|-1,
\]
so $\x$ is not minimal.
\end{proof}

In certain cases, sequences of digits with absolute value $\fne$ and the same sign form runs at which $\x$ can be reduced.
\begin{lemma}\label{lemma:even_reduction_adjacent_digits}
Let $n\ge 4$ be even and $\x \in \Bvuw$.  Suppose there is a maximal
sequence $x_j = x_{j+1} = \dots = x_\ell = \pm \fne$ of length at
least $2$ with $\kx > \ell$ and $|x_{\ell+1}| < \fne$.
Then $\x$ can be reduced at the run $\r=(x_j, \dots, x_\ell) \subseteq \x$.
\end{lemma}
\begin{proof}
Let
$\y = \x + \epsilon_\r\sum_{i=j}^\ell\wi$.
The assumptions on $\kx$ and $|x_{\ell+1}|$ guarantee that
$\y \in \Bvuw$ and $\ky \le \kx$.  
The digits in $\r$ ensure that $\wt(\r) \geq 1$.  
To compare $\Vert \x \Vert_1$ and $\Vert \y \Vert_1$ we use the equation given in Lemma~\ref{lemma:change_formula}, namely
\[\Vert\y\Vert_1 = \Vert \x \Vert_1 - \wt(\r)
+ |x_{\ell+1} + \epsilon_\r| - |x_{\ell+1}|.
\]
As it is always true that $|x_{\ell+1} + \epsilon_\r| - |x_{\ell+1}| \in \{\pm 1\}$, we see that $\wt(\r) - (|x_{\ell+1} + \epsilon_\r| - |x_{\ell+1}|) \geq 0$, and thus $\Vert \y \Vert_1 \leq \Vert \x \Vert_1$.

As $\ky \in \{\kx,\kx-1\}$, we use the same length formula from Lemma~\ref{lemma:length_formula} to compute both $|\eta_{u,v,w}(\x)|$ and $|\eta_{u,v,w}(\x)|$.  
If $\Vert \y \Vert_1 < \Vert \x \Vert_1$, if we use the first length formula, the result follows immediately.  
If we use the second length formula, we are also relying on the fact that $\ky \leq \kx$ to conclude that $\y \ords \x$.
If  $\Vert \y \Vert_1 = \Vert \x \Vert_1$, then
$\r = (\epsilon_\r \fne,\epsilon_\r \fne)$ and without loss of generality we assume that $\epsilon_\r=1$.  
Then 
\[
(y_j,y_{j+1},y_{j+2}) = \left(-\fne,-\left(\fne-1\right),x_{j+2}+1\right)
\]
and thus $\y$ precedes $\x$ in the lexicographic order, so $\y \ords \x$ in this case as well.
\end{proof}

We are interested in conditions which are both necessary and sufficient to conclude that $\x \in \Bvuw$ is not minimal.
This stronger statement is contained in  Proposition~\ref{lemma:even_minimal_characterization}.

\begin{proposition}
\label{lemma:even_minimal_characterization}
Let $n\ge 4$ be even, and let $\xx \in \Bvuw$.
Then $\x$ is not minimal if and only if 
\begin{itemize}
    \item there is a run in $\x$ at which $\x$ can be reduced, or
    \item Lemma~\ref{lemma:even_reduction_end} applies to $\x$.
\end{itemize}
\end{proposition}
\begin{proof}
If one of the two conditions in the lemma is satisfied, then $\x$ is
not minimal, so the proof reduces to showing the converse.
Let $\xx, \y \in \Bvuw$ and with $\y$ a minimal vector, and $\z = \y - \xx$.  Let $j$ be the minimal index
such that $x_j \ne y_j$ and write $\z = \sum_{i=j}^\ell \alpha_i \wi$.  It follows from Lemma~\ref{lemma:L0_big_digits} that $|z_j| \ge n$.  To satisfy
the digit bounds on $\Bvuw$, we must then have $|z_j| = n$.  Assume without
loss of generality that $x_j \ge 0$, so $z_j = -n$ and $y_j = x_j - n$.
The proof now reduces to cases corresponding to the possible values of $x_j$.
\begin{enumerate}[itemsep=5pt]
\item Case 1: $x_j = \fne + 1$. This digit can only occur if $j = k_\xx$ and $\kx \geq \max(u,w)$; it follows that $\alpha_j = 1$. 
This means that $y_\kx = -(\fne - 1)$ and $y_{\kx+1} = 1- \alpha_{\kx+1}n$, so we must have
$\alpha_{\kx+1}=0$ and $k_\y = k_\xx + 1 > \max(u,w)$.
We use the second formula from Lemma~\ref{lemma:length_formula}
to compute both $|\eta_{u,v,w}(\x)|$ and $|\eta_{u,v,w}(\y)|$.  We see that
\begin{align*}
|\eta_{u,v,w}(\y)| & = \Vert \xx \Vert_1+ |y_{\kx}| - |x_{\kx}| + |y_{\kx+1}| + 2(\kx +1) - |u-w| \\
                   & = |\eta_{u,v,w}(\xx)| + |y_{\kx}| - |x_{\kx}| + |y_{\kx+1}| + 2 \\
                   & = |\eta_{u,v,w}(\xx)| + 1,
\end{align*}
contradicting our assumption that $\y$ is minimal.  Thus this case does not occur.

\item Case 2: $x_j = \fne$. This is the involved case, which is proven in Section~\ref{sec:technical_lemmas} as Lemma~\ref{lemma:lastcase_lemma3.20}.

\item Case 3: $x_j = \fne -1$. In this case, $y_j = -\fne -1$, so $k_\y = j$. 
Since $|y_j| = \fne+1$, it follows from the definition of $\Bvuw$ that $\ky \geq \max(u,w)$.  
As the length of $\y$ is determined, we must have $x_{j+1} = -1$, and $k_\xx = j+1$.  Then $\kx > \ky \geq \max(u,w)$, and we see that $\x$ satisfies the conditions of Lemma~\ref{lemma:even_reduction_end}.

\item Case 4: $x_j < \fne-1$. Here $y_j < -(\fne +1)$, contradicting the fact that $\y \in \Bvuw$.
\end{enumerate}
These four cases complete the proof of the proposition.
\end{proof}

It can be computationally difficult to check whether $\x$ contains a run at which it can be reduced.
When $\kx < w$, Proposition~\ref{lemma:adjacent_digits} presents straightforward observable conditions which guarantee that $\x$ contains a run at which it can be reduced.
This prompts the following definition; if $g = t^{-u}a^vt^w$ and $\x \in \Bvuw$ with $\kx < \max(u,w)$,
we say that $\eta_{u,v,w}(\x)$ has \emph{strict shape 1}.

We rely on Proposition~\ref{lemma:adjacent_digits} when computing the growth rate of $BS(1,n)$ in \cite{TW_growth}, namely we show in \cite{TW_growth} that the set of geodesics of strict shape 1 forms a regular language whose growth rate is the same as the growth rate of $BS(1,n)$.  We use a corollary of this result below to show that the sets of elements positive, negative and zero conjugation curvature which we exhibit in Sections~\ref{sec:curvature} and ~\ref{sec:n=2_curvature} have positive density in $BS(1,n)$.

\begin{proposition}\label{lemma:adjacent_digits}
Let $n>2$ be even and $\x \in \Bvuw$ with $\kx < \max(u,w)$.
Then $\x$ is not minimal if and only if one of the following holds, for $\delta \in \{\pm 1\}$.
\begin{itemize}[itemsep=5pt]
\item There are two adjacent digits in $\x$
of the form $(\delta\fne,\delta\fne)$.
\item There are two adjacent digits in $\x$
of the form $(\delta\fne,x_i)$ with $\sgn(x_i) = -\sgn(\delta)$.
\end{itemize}
\end{proposition}
\begin{proof}
By applying Proposition~\ref{lemma:even_minimal_characterization} and
observing that Lemma~\ref{lemma:even_reduction_end} does not apply to $\x$,
the proof reduces to showing that there is a run at which $\x$ can
be reduced if and only if one of the above conditions holds.

First observe that in each of the two cases in the lemma there is a run at which $\x$ can be reduced.  
Consider the run which is the maximal sequence of digits $\delta\fne$ containing the digit(s) in the statement of the lemma, that is, $\r = (\delta\fne,\delta\fne, \cdots ,\delta\fne) = (x_j, \cdots ,x_\ell)$ where $j \leq l$. Let 
\[
\y = \x + \sum_{i=j}^\ell \delta \wi.
\]
 and compute $\wt(\r) = l-j \geq 0$.

As $\kx < \max(u,w)$ and $\ky \leq \kx+1$, we have $\ky \leq \max(u,w)$ and thus we use the first length formula in Lemma~\ref{lemma:length_formula} to compute both $|\eta_{u,v,w}(\x)|$ and $|\eta_{u,v,w}(\y)|$.
Note that the two formulas agree when $\ky = \max(u,w)$.  
As this formula does not take into account the length of the vectors, any change in word length results from a change in $\ell^1$ norm between $\x$ and $\y$.

It follows from Lemma~\ref{lemma:change_formula} that 
\[
\Vert \x \Vert_1 - \Vert \y \Vert_1 = \wt(\r) +  |x_{\ell+1}|-|x_{\ell+1}+\delta| .
\]
Suppose that $\wt(\r) = l-j \geq 1$, so there are at least two digits of the form $\delta \fne$.
We know that $|x_{\ell+1}|-|x_{\ell+1}+\delta| \in \{\pm 1\}$
 and hence $\Vert \x \Vert_1 - \Vert \y \Vert_1 \geq 0$.
If the inequality is strict, it follows that $\x$ can be reduced at $\r$, that is, $\x$ is not minimal.
If there is equality, notice that the change from $x_{j+1}$ to $y_{j+1}$ is a lexicographic reduction, as $|x_{j+1}| = \fne$ and $|y_{j+1}| = \fne -1$.
Thus $\y \ords \x$, that is, $\x$ is not minimal.

Suppose that $\wt(\r)=l-j=0$, so we are in the second case of the lemma.  In this case, $ |x_{\ell+1}|-|x_{\ell+1}+\delta|  =1$, so $\Vert \x \Vert_1 - \Vert \y \Vert_1 = 1>0$.  Thus $\y \ords \x$ and we conclude that $\x$ is not minimal.

Now we must show the converse.  That is, if 
$\x$ can be reduced at a run $\r$ then 
one of the conditions in the
statement of the lemma holds. 
Let $\r = (x_j, \dots, x_\ell)$ be such a run,
and $\y = \x + \epsilon_\r\sum_{i=j}^\ell\wi$.

Again note that $\ky \leq \kx+1$, so the assumption that $\kx < \max(u,w)$ means that, as above, we use the first length formula in Lemma~\ref{lemma:length_formula} to compute both $|\eta_{u,v,w}(\x)|$ and $|\eta_{u,v,w}(\y)|$.
Thus any change in word length results from a change in $\ell^1$ norm between $\x$ and $\y$.

As $\y \ords \x$, we know that $|\eta_{u,v,w}(\y)| \le |\eta_{u,v,w}(\x)|$,
so
\begin{align*}
|\eta_{u,v,w}(\x)| - |\eta_{u,v,w}(\y)|
& = \Vert\x\Vert_1 - \Vert \y \Vert_1 \\
& = \wt(\r)  + |x_{\ell+1}| - |x_{\ell+1} + \epsilon_\r| \\
& \ge 0
\end{align*}
Consider $x_{\ell+1}$.
We know that $|x_{\ell+1}| - |x_{\ell+1} + \epsilon_\r| \in  \{\pm 1\}$
\begin{enumerate}[itemsep=5pt]
\item[(a)] If $|x_{\ell+1}| - |x_{\ell+1} + \epsilon_\r| = 1$, then
$\wt(\r) \ge -1$.

\smallskip

\begin{itemize}[itemsep=5pt]
    \item If $\wt(\r)=-1$, then $|\eta_{u,v,w}(\x)| = |\eta_{u,v,w}(\y)|$,
    so in order to have $\y \ords \x$, we must have a lexicographic reduction from $\x$ to $\y$, that is, $\y$ precedes $\x$ in the lexicographic order,
    meaning there must be a decrease in absolute value from $|x_{j+1}|$ to $|y_{j+1}|$.
    There are two ways this can occur: $x_{j+1} = \delta\fne$ and the run
    has length at least 2, or $\sgn(x_{j+1}) = -\sgn(\delta)$ and the
    run has length 1.  In either case, one of the conditions of the lemma
    is satisfied.
    \item If $\wt(\r) = 0$, then $\r$ contains exactly one more digit
    $\delta\fne$ than it does $\delta(\fne-1)$.
    Either the first condition of the lemma is satisfied, or $\r$ ends with $\delta\fne$,
    and because $|x_{\ell+1}| > |x_{\ell+1} + \epsilon_\r|$, we must have $\sgn(x_{\ell+1}) = -\sgn(\delta)$,
    so the second condition of the lemma is satisfied.
    \item If $\wt(\r) > 0$, then $\r$ contains at least two more
    digits $\delta\fne$ than it does digits $\delta(\fne-1)$, so
    the first condition of the lemma is satisfied.
\end{itemize}

\item[(b)] If $|x_{\ell+1}| - |x_{\ell+1} + \epsilon_\r|  = -1$, then
$\wt(\r) \ge  1$, which is the third case above.
\end{enumerate}
In all cases, we have shown that one of the two conditions of the
lemma is satisfied.
\end{proof}

\begin{ex}\label{example:absolute_minimal}
Proposition~\ref{lemma:adjacent_digits} provides a straightforward
way to ensure that a vector corresponding to a geodesic of strict shape 1 is minimal.  We revisit
Example~\ref{example:absolute_order};  recall
$n=4$, $v=26$, and $u,w \ge 3$.
Consider
\[
\xx = (\underset{x_0}{2}, \underset{x_1}{2}, \underset{x_2}{1}, \underset{x_3}{0},
 \, \underset{\ldots}{\ldots})
\qquad
\text{ and }
\qquad
\y = (\underset{y_0}{-2}, \underset{y_1}{-1}, \underset{y_2}{2}, \underset{y_3}{0},
\, \underset{\ldots}{\ldots}).
\]
Note that $\x$ satisfies the first condition of
Proposition~\ref{lemma:adjacent_digits}, so is not
minimal.  Indeed, $\y \ords \x$. However, no
condition of Proposition~\ref{lemma:adjacent_digits} applies to $\y$, so $\y$ is minimal.
\end{ex}

\subsection{Minimal vectors for $n=2$}
\label{section:geodesics_n_even_2}
In this section, we provide statements analogous
to Lemmas~\ref{lemma:even_minimal_characterization}
and~\ref{lemma:adjacent_digits} for
the special case of $n=2$.
The reason the statements and proofs of
Section~\ref{section:geodesics_n_even} do not apply directly is the
fact that when $\kx \geq \max(u,w)$ the absolute value of the most significant digit
of a vector $\x \in \Bvuw$
for $n=2$ is bounded by $\fne + 2$, rather than $\fne+1$.
For the remainder
of this section we assume that $n=2$.

We defer the proofs of the main propositions in this section to Section~\ref{sec:technical_n=2}, as they are similar in structure to the proofs in Section~\ref{section:geodesics_n_even}.  However, we clarify below the slight differences between a run when $n>2$ and $n=2$, as well as a difference which may arise when a vector $\x \in \Bvuw$ can be reduced at a run $\r$.

When $n=2$, define a \emph{run} $\r$ in $\xx \in \Bvuw$ to be a sequence
of consecutive digits $\r = (x_j, \dots, x_\ell)$ of $\x$ such that
$|x_j| = \fne=1$ and for all $i$ with $j \le i \le \ell$ we have
$\sgn(x_i) = \sgn(x_j)$ or $\sgn(x_i) = 0$.   We denote this sign by
$\epsilon_\r =\sgn(x_j)$.  We remark that the digit bounds
defining $\Bvuw$ imply that if
$|x_i| \in \{\fne + 1, \fne + 2\} = \{2,3\}$, then in fact $i=\ell=\kx$ and $\kx \geq \max(u,w)$. 
The \emph{length} of the run is $\ell - j +1$.
Thus we define a run to either
\begin{itemize}
    \item begin with $1$, consist of a word in $\{0,1\}^*$ and possibly conclude with the digit 2 or 3, or
    \item begin with $-1$, consist of a word in $\{0,-1\}^*$ and possibly conclude with the digit -2 or -3.
\end{itemize}
If $\r =  (x_j, \dots, x_\ell) \subseteq \x \in \Bvuw$ is a run, we say that $\x$ can be reduced
at $\r$ if
\[
\x + \epsilon_\r \left[\sum_{i=j}^{\ell-1} \wf{i} + \alpha_{\ell}\wf{\ell}\right] \ords \x,
\]
where $\alpha_{\ell} \in \{1,2\}$.  
The possibility that $\alpha_\ell = 2$ does
not occur when $n>2$.
Thus our conditions for determining the minimality of $\x \in \Bvuw$ are slightly different when $n=2$.
Because the digit bounds in $\Bvuw$ when $n=2$ allow a final digit with absolute
value as large as $3$, one might suppose that we need to consider linear
combinations of $\wi$ with final coefficient as large as $3$.  However, the following lemmas
give us more control over these coefficients.

\begin{lemma}\label{lemma:n=2_unequal_lengths}
Let $n=2$ and $\x,\y \in \Bvuw$.  Let $\y-\x = \sum_{i=j}^\ell \alpha_i \wi$.
If $\alpha_\kx \ne 0$, then $\ky > \kx$.
\end{lemma}

\begin{proof}
Let $m$ be maximal so that $\alpha_m \ne 0$.  Since $\alpha_\kx \ne 0$, we have
$m \ge \kx$.  Then $y_{m+1} = \alpha_m$, so $\ky = m+1 > \kx$.
\end{proof}

\begin{lemma}
\label{lemma:less_than_6}
Let $n=2$ and $\x,\y \in \Bvuw$ with $\kx \le \ky$.  Let $\y-\x = \sum_{i=j}^\ell \alpha_i \wi$.
Then $|\alpha_i| \le 2$, with $|\alpha_i|=2$
only possible if $i = \kx < \ky$.
\end{lemma}

\begin{proof}
It follows from Lemma~\ref{lemma:no_L0_parts} that $\ell < \ky$;
we prove the lemma by induction on $i$.  First suppose $i<\kx \leq \ky$, so $|x_i|,|y_i| \leq 1$.
Writing $y_i - x_i = \alpha_{i-1}-2\alpha_i$ and applying the induction assumption that $|\alpha_{i-1}| \le 1$, it follows that $2 |\alpha_i| \leq 3$ and thus $|\alpha_i| \leq 1$.

If $i=\kx$ it follows from Lemma~\ref{lemma:n=2_unequal_lengths} that $\kx < \ky$.  Now we have the bounds $|x_i| \leq 3$ and $|y_i| \leq 1$.
Writing $y_i - x_i = \alpha_{i-1}-2\alpha_i$ and applying the induction assumption that $|\alpha_{i-1}| \le 1$, it follows that $2 |\alpha_i| \leq 5$ and thus $|\alpha_i| \leq 2$.

If $i=\kx+1$, the same analysis with $|\alpha_{i-1}| \leq 2$ shows that $|\alpha_i| \leq 1$.  It then follows from previous arguments that for $\kx < i < \ky$ we have $|\alpha_i| \leq 1.$ 
\end{proof}

The next lemma, analogous to Lemma~\ref{lemma:even_reduction_end}, describes a situation where $\x \in \Bvuw$ is not minimal but does not
necessarily contain a run at which it can be reduced.

\begin{lemma}
\label{lemma:n=2_notminimal}
Let $n=2$ and $\x \in \Bvuw$.  If $\kx > \max(u,w)$ and $\x$ ends in
the digits $(0, \delta)$ for $\delta \in \{\pm 1\}$ then $x$ is not minimal.
\end{lemma}
\begin{proof}
It is easily checked that adding $- \delta {\bf w}^{(\kx-1)}$ to $\x$
increases $\Vert \x \Vert_1$ by $1$ and reduces the length of the vector by $1$.
Since $\kx > \max(u,w)$, we use the second length formula in
Lemma~\ref{lemma:length_formula} to compute $|\eta_{u,v,w}(\x)|$ and
$|\eta_{u,v,w}(\x- \delta {\bf w}^{(\kx-1)})|$, so
\[
|\eta_{u,v,w}(\x)| - |\eta_{u,v,w}(\x- \delta {\bf w}^{(\kx-1)})| = -1 + 2 = 1,
\]
and hence $\x$ is not minimal.
\end{proof}

In Lemma~\ref{lemma:n=2_110} we identify several digit patterns which imply that a vector $\x \in \Bvuw$ is not minimal.
Moreover, the existence of one of these patterns guarantees that $\x$ contains a run at which it can be reduced.
We defer the proof of Lemma~\ref{lemma:n=2_110} to Section~\ref{sec:technical_lemmas}.

\begin{lemma}
\label{lemma:n=2_110}
Let $n=2$ and suppose that $x \in \Bvuw$ and $\delta \in \{\pm 1\}$.
If any of the following occur, then there is a run at which
$\x$ can be reduced, and hence $\x$ is not minimal.
\begin{enumerate}[itemsep=5pt]
    \item $\x$ contains the digits $(\delta, -\delta\alpha)$ for $\alpha > 0$.
    \item $\kx \ne \max(u,w)$ and $\x$ ends in the digits $(\delta, \delta)$.
    \item $\x$ contains the digits $(\delta, \delta, \alpha)$ for any $\alpha$.
\end{enumerate}
\end{lemma}

It follows from Lemma~\ref{lemma:n=2_110} that the
only way $\x$ can be minimal and contain the digit sequence $(1,1)$
is if $\kx = \max(u,w)$ and these digits occur at the end of $\x$.

\begin{remark}
\label{remark:digit_sequences}
Let $\x \in \Bvuw$ and $\r \in \{0,1\}^*$ be a run in $\x$.  If $\r$ contains at least
two more occurrences of the digit 1 than the digit 0,
then either $\r$ contains the sequence $(1, 1, 0)$ or $\r$ ends in $(1, 1)$.
In the first situation, $\x$ can be reduced at the run $(1, 1, 0)$.
In this second case, if $x_{\kx} \in \r$ and $\kx \neq \max(u,w)$, then by Lemma~\ref{lemma:n=2_110},
$\x$ can be reduced at the run $(1,1)$.
\end{remark}

The following lemma is the analog of Proposition~\ref{lemma:even_minimal_characterization} for the case $n=2$.
One direction of the proof is clear, and we defer the remainder of the proof to Section~\ref{sec:technical_lemmas}.
\begin{proposition}
\label{lemma:lemma318_n=2}
Let $n=2$ and $\x \in \Bvuw$.  Then $\x$ is not minimal if and only if one of the following occurs.
\begin{itemize}[itemsep=5pt]
    \item There is a run at which $\x$ can be reduced.
    \item Lemma~\ref{lemma:n=2_notminimal} applies to $\x$.
\end{itemize}
\end{proposition}

Note that the conclusion of Lemma~\ref{lemma:n=2_adjacent_digits} is
identical to that of Proposition~\ref{lemma:adjacent_digits} when $n=2$.
The different analysis of the case $n=2$ leads us to separate the propositions.
It follows from Proposition~\ref{lemma:n=2_adjacent_digits} that to determine whether $\x \in \Bvuw$ is minimal, where $\x$ corresponds to a geodesic of strict shape 1, it is sufficient to consider adjacent pairs of coordinates, and rule out two specific patterns.

\begin{proposition}
\label{lemma:n=2_adjacent_digits}
Let $n=2$ and $\x \in \Bvuw$ and $\kx < \max(u,w)$.
Then $\x$ is not minimal if and only if
$\x$ contains a digit sequence of the form  $(\delta, \delta)$
or $(\delta, -\delta)$, for $\delta \in \{\pm 1\}$. 
\end{proposition}

\section{Growth and Regular Languages}
\label{sec:growth}
Given $u,w$, and $\x$ with $\kx <\max(u,w)$, 
Lemmas~\ref{lemma:odd_box},~\ref{lemma:adjacent_digits} and~\ref{lemma:n=2_adjacent_digits} provide a straightforward
way to determine whether $\x \in \Bvuw$ is minimal,
that is, whether $\eta_{u,v,w}(\x)$ is a geodesic, by examining the
digits of $\x$.  
Recall that if $\kx < \max(u,w)$,
we say that $\eta_{u,v,w}(\x)$ has \emph{strict shape 1}.

In \cite{TW_growth} we prove that the set of vectors
$\x$ for which there are $u,w$ so that $\eta_{u,v,w}(\x)$
is geodesic and has strict shape 1 forms a regular language, denoted ${\mathcal D}_n$.
This language is not a language of geodesic paths, merely of vectors which yield geodesic paths of strict shape $1$ with a choice of $u$ and $w$. 
Let $\OO_n$ denote the corresponding language of geodesic paths of strict shape $1$.
In \cite{TW_growth} we show that $\OO_n$ is also a regular language, exhibiting finite state automata which accept these two languages.

The finite state automaton accepting ${\mathcal D}_n$ has a finite number of states, regardless of the value of $n$.
We use it to produce a finite state automaton accepting $\OO_n$ by performing a ``digit expansion" procedure which produces an automaton where the number of states does depend on $n$.  
The salient piece of information about this machine is that it has one strongly connected component which determines its growth rate.  
We refer the reader to \cite{TW_growth} for additional details of this procedure.
Figure~\ref{fig:fsa_n=2} depicts the finite state automaton accepting ${\mathcal O}_2$.
While the analogous automaton for $n>2$ is more complex, it shares the feature that there is one strongly connected component, and one additional component containing the state $s_{\ti}$, which accounts for the initial string of the letter $\ti$ at the start of an accepted word.
This fact will be referred to below in the proof of Lemma~\ref{lemma:middle_vector}.

\begin{figure}[ht!]
\tikzset{every state/.style={minimum size=2em}} 
\begin{center}
\resizebox {\columnwidth} {!} {
\begin{tikzpicture}[->,>=stealth',shorten >=1pt,auto, node distance=5cm,
semithick]

\node[state,accepting] (A) {$s_{0,0}$};
\node[state] (A1) [above right of =A, yshift=-0.5cm] {$s_{0,1}$};
\node[state] (A_1) [above left of =A, yshift=-0.5cm] {$s_{0,-1}$};

\node[state] (T) [above of =A] {$s_{t^{-1}}$};
\node[state] (start) [above of =T,yshift=-1.5cm] {$\texttt{start}$};

\node[state,accepting] (B) [right of =A] {$s_{1,0}$};
\node[state, draw=none] (BB) [right of=B] {};

\node[state,accepting] (C) [left of = A] {$s_{2,0}$};
\node[state, draw=none] (CC) [left of=C] {};

\path(start)
edge node[left] {$t^{-1}$} (T)
edge [bend left] node[right, pos=0.25] {$t$} (A)
edge [bend left] node[left] {$a$} (A1)
edge [bend right] node[right] {$a^{-1}$} (A_1);

\path(T)
edge [loop below] node[below] {$t^{-1}$} (T)
edge node[above] {$a$} (A1)
edge node[above] {$a^{-1}$} (A_1);

\path(A)
edge [loop below] node[below] {$t$} (A)
edge node[right]{$a$} (A1)
edge node[right]{$a^{-1}$} (A_1);

\path(A1)
edge node[above] {$t$} (B);

\path(A_1)
edge node[above] {$t$} (C);

\path(B)
edge node[above]{$t$} (A)
edge [shorten >=75pt, dashed] node[above] {$a$} (BB)
edge [shorten >=75pt, dashed] node[below] {$a^{-1}$} (BB);

\path(C)
edge node[above] {$t$} (A)
edge [shorten >=75pt, dashed] node[above] {$a$} (CC)
edge [shorten >=75pt,dashed] node[below] {$a^{-1}$} (CC);

\end{tikzpicture}
}
\end{center}
\caption{The finite state automaton $\OO_2$
accepting the language $\OO_2$ of geodesics of strict shape 1 in $BS(1,2)$.
Accept states are indicated with a double circle.}
\label{fig:fsa_n=2}
\end{figure}
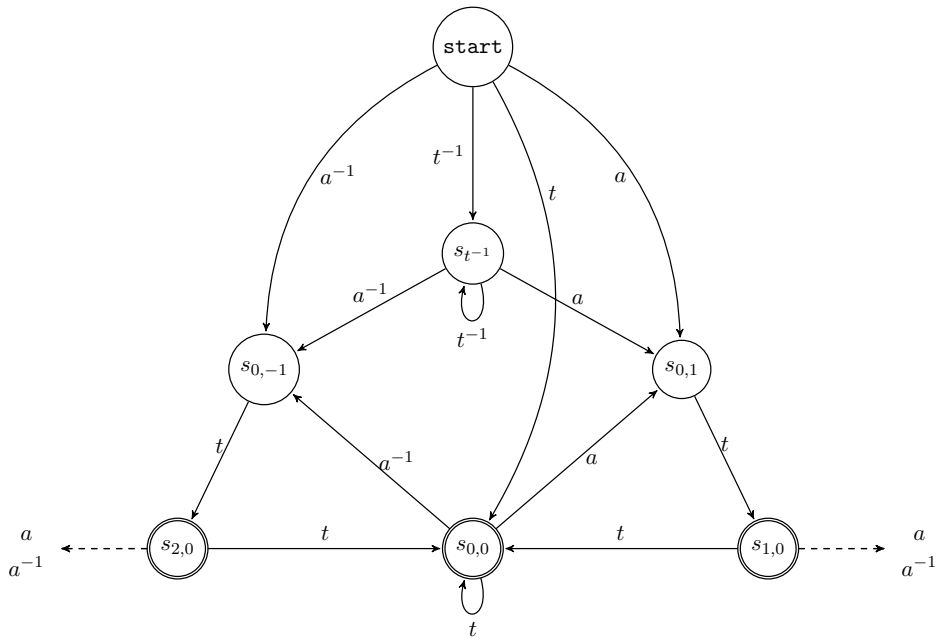

Recall that the growth rate of a sequence $\{f(N)\}_{N=1}^\infty$ is $\lambda$ if
\[
\lim_{N\to\infty} \frac{\log f(N)}{N\log\lambda} = 1.
\]
Equivalently, we write $f(N) = \Theta(\lambda^N)$; that is,
there are constants $A,B>0$ such that $$A \lambda^N \le f(N) \le B \lambda^N$$
for sufficiently large $N$.

For any set ${\mathcal A} \subset BS(1,n)$, we use the notation ${\mathcal A}(N)$ to denote all elements of the set with word length $N$ with respect to the generating set $\{a,t\}$.
A main result of \cite{TW_growth} is the following theorem, which shows that to understand the growth rate of $BS(1,n)$ it is sufficient to understand the growth rate of the sequence $\OOn$.
Let $S_n(N)$ denote the sphere
of radius $N$ in $BS(1,n)$.  The growth rate of $BS(1,n)$, or any finitely generated group,
is defined to be the growth rate of the sequence $\{|S_n(N)|\}_{n \in \N}$.

We say that ${\mathcal A}$ has positive density in $BS(1,n)$ if there is some $\epsilon >0$ so that for all sufficiently large $N$,
\[
\epsilon < \frac{|{\mathcal A}(N)|}{|S_n(N)|} < 1-\epsilon.
\]
If ${\mathcal A}$ has the same growth rate as $BS(1,n)$ it follows immediately that ${\mathcal A}$ has positive density in $BS(1,n)$.

\begin{theorem}[\cite{TW_growth}, Corollary 5.4]
\label{corollary:shapes_injection}
In the notation above, we have
\[
|\OO_n(N)| \le |S_n(N)| \le 20|\OO_n(N+3)|.
\]
Consequently,
the growth rates of the sequences $\OOn$ and $\{|S_n(N)|\}_{n \in \N}$ are identical.
\end{theorem}
This growth rate is computed explicitly in \cite{TW_growth}.

The following lemma allows us to effectively compute the growth rate of the
function which counts the number of accepted paths of a given length
in a finite state automaton.

\begin{lemma}[\cite{TW_growth}, Lemma 4.2]
\label{lemma:fsa_growth}
Let $F$ be a finite state automaton with state set $S$.
Let $f(N)$ denote the number of accepted paths in $F$ of length $N$,
and for each $s \in S$, let $f_s(N)$ denote the number of accepted
paths in $F$ beginning at state $s$.
Let $S_1, \dots, S_c$ be the strongly connected components in $F$.
\begin{enumerate}
    \item For each $i$, the growth rate of $f_s$ is constant over all $s \in S_i$.
    \item The growth rate of $f$ is the maximum of the growth rates of the $S_i$.
\end{enumerate}
\end{lemma}

To compute the growth rate of $\OOn$, we must account for the fact that the number of states in the finite state automaton accepting $\OO_n$ depends on $n$, while the number of states in the finite state automaton accepting ${\mathcal D}_n$ is constant.
We do this via a matrix equation of fixed size, where the entries are growth series for paths beginning, respectively, in each state of the automaton accepting $\OO_n$.  
That is, we trade a computation
with arbitrarily large matrices over the integers (computing
an eigenvalue) for a computation with fixed size matrices with
entries which are infinite series.

The following basic fact about exponential growth will be referred to frequently in Sections~\ref{sec:curvature} and ~\ref{sec:n=2_curvature}.  A proof is included in \cite{TW_growth}.

\begin{lemma}[\cite{TW_growth},Lemma 4.1]
\label{lemma:exp_growth}
Suppose that $f(N) = \Theta(\lambda^N)$ with $\lambda > 1$.
\begin{enumerate}[itemsep=5pt]
\item Both $f(N+k)$ and $\sum_{i=1}^Nf(i)$ are $\Theta(\lambda^N)$.
\item If $f(N)$ and $g(N)$ are $ \Theta(\lambda^N)$, there are $N_0,d>0$ so that $f(N)/g(N) > d$ for $N > N_0$.
\end{enumerate}
\end{lemma}

\section{Conjugation curvature in $BS(1,n)$}
\label{sec:curvature}

We begin this section with several results about minimal vectors which follow
from the technology developed in
Section~\ref{section:min_rep}.  We combine this
with our understanding of growth rates from Section~\ref{sec:growth} to study the density of elements whose conjugation curvature is, respectively, positive, negative and zero.
Recall that the conjugation curvature $\kappa_r(h)$ is  defined to be
\[
\kappa_r(h) = \frac{l(h) - \frac{1}{|S_n(r)|} \sum_{w \in S_n(r)} l(h^w)}{l(h)}
\]
that is, the difference between the word length of $h$ and the average word length of the conjugates of $h$ by all $w$ in the sphere $S_n(r)$ of radius $r$ in the Cayley graph $\Gamma(G,S)$,
scaled by the word length of $h$.
We show that $BS(1,n)$ has a positive density of elements with
$\kappa_r(g)<0$ and $\kappa_r(g)=0$,  where $r$ is allowed to
assume a finite range of values.  Additionally, when $r=1$ we show that
$BS(1,n)$ has a positive density of elements with $\kappa_1(g)>0$.

\subsection{Conjugation curvature when $r=1$.}
\label{sec:r=1}

When computing $\kappa_r(g)$ for $g = t^{-u}a^vt^w$ we must be able to evaluate
$l(g^p)$ where $p=s_1s_2 \cdots s_r$, and each
$s_i \in \{a^{\pm 1},t^{\pm 1}\}$.  
We begin by understanding how $u$, $v$, and $w$ change under conjugation
by a single generator of $BS(1,n)$.  This enables us to characterize when $\x$ is a minimal vector for both $g$ and $g^s$, for $s \in \{a^{\pm 1},t^{\pm 1}\}$, which in turn allows us to compute
the change in word length and thus $\kappa_1(g)$.

We outline this idea with the simplifying assumption that 
$n \nmid v$ and $uw>0$.
Let $g= t^{-u}a^vt^w$. With these assumptions, the four conjugates
of $g$ by the generators are as follows.
\begin{enumerate}[itemsep=5pt]
\item $g^t=t(t^{-u}a^vt^w)\ti = t^{-(u-1)}a^vt^{w-1}$
\item $g^{\ti}=\ti(t^{-u}a^vt^w)t = t^{-(u+1)}a^vt^{w+1}$
\item $g^{a}=a( t^{-u}a^vt^w)a^{-1} = t^{-u}a^{n^u+v-n^w}t^w$
\item $g^{\ai}=a^{-1}( t^{-u}a^vt^w)a = t^{-u}a^{-n^u+v+n^w}t^w$
\end{enumerate}
When $uw=0$, we obtain $g^t=t^{-u}a^{nv}t^w$, and the remaining
conjugates are unchanged.
If $n|v$ then $uw=0$ and we obtain $g^{\ti} = t^{-u} a^{\frac{v}{n}} t^w$;
the remaining conjugates are unchanged. Observe that the
formulas above demonstrate how $u,v,w$ change under conjugation.

In order to determine the geodesic lengths of $g^{a^{\pm 1}}$ and
$g^{t^{\pm 1}}$, we begin with a minimal vector
$\x \in \Bvuw$, so $\eta_{u,v,w}(\x)$ is a geodesic representing $g$.
We must then find a minimal vector in one of
$\BB_v^{u-1,w-1}$,
$\BB_v^{u+1,w+1}$, $\BB_{n^u+v-n^w}^{u,w}$,
and $\BB_{-n^u+v+n^w}^{u,w}$ in order to calculate the word length of the appropriate conjugate of $g$.
Sometimes this is straightforward, for example, in 
Lemma~\ref{lemma:zero_curv} with the assumption that $w=u$.  
We begin with some convenient corollaries
of the results in Sections~\ref{section:geodesics_n_odd} 
and~\ref{section:geodesics_n_even} which will allow us to recognize minimal vectors under specific conditions.

\begin{lemma}\label{lemma:still_minimal}
If $\x$ is a minimal vector in $\Bvuw$ then $\x$ is a minimal vector in $\BB_v^{u',w'}$ for any pair $u',w'$
so that $\max(u,v)$ and $\max(u',w')$ have the same ordinal relationship to $\kx$.
\end{lemma}

\begin{proof}
The hypotheses for Lemma~\ref{lemma:odd_box} when $n$ is odd, Proposition~\ref{lemma:even_minimal_characterization} when $n>2$ is even, and Proposition~\ref{lemma:lemma318_n=2} when $n=2$ depend only on the ordinal relationship between $\max(u,v), \ \max(u',w')$ and $\kx$.  
Thus it follows from the appropriate lemma that $\x$ is minimal in $\BB_v^{u',w'}$.
\end{proof}

Given $\kx > \max(u,w)$, Lemma~\ref{lemma:still_minimal_2} extends the conclusion of Lemma~\ref{lemma:still_minimal} by relaxing the condition that $\kx > \max(u',w')$ to allow $\kx \geq \max(u',w')$.

\begin{lemma}\label{lemma:still_minimal_2}
If $\x$ is a minimal vector in $\Bvuw$ and $\kx > \max(u,w)$,
then $\x$ is a minimal vector in $\BB_v^{u',w'}$ for any pair
$u',w'$ with $\kx \ge \max(u',w')$.
\end{lemma}

Before proving Lemma~\ref{lemma:still_minimal_2}, we prove the following two lemmas which describe the change in word length as $u$ and $w$ are, respectively, decremented and incremented while all other parameters are unchanged.

\begin{lemma}\label{lemma:still_minimal_2_claim_1}
For any $\x \in \Lv$
\[
|\eta_{u,v,w+1}(\x)| = 
|\eta_{u,v,w}(\x)| + 
\left\{\begin{array}{ll}
\phantom{-}1  & \textnormal{if $\max(u,w) \ge \kx$} \\
\phantom{-}1  & \textnormal{if $\max(u,w) < \kx$ and $u > w$} \\
-1 & \textnormal{if $\max(u,w) < \kx$ and $u \le w$}
\end{array}\right.
\]
An identical equality holds if we exchange the roles of $u$ and $w$.
\end{lemma}
\begin{proof}
The situation is symmetric in $u$ and $w$, so it suffices to
consider only $w$.  If $\max(u,w) \ge \kx$, then we can use the first
length formula in Lemma~\ref{lemma:length_formula} to compute both $|\eta_{u,v,w}(\x)|$ and $|\eta_{u,v,w+1}(\x)|$,
and the lemma is immediate.  If $\max(u,w) < \kx$, then we can use
the second length formula in Lemma~\ref{lemma:length_formula} to compute the lengths of both paths, so the sign of the change in length
depends on the order of $u$ and $w$ as given.  Note that we are using the fact that if
$\max(u,w) = \kx$ then the formulas in Lemma~\ref{lemma:length_formula} agree.
\end{proof}

\begin{lemma}\label{lemma:still_minimal_2_claim_2}
For any $\x \in \Lv$
\[
|\eta_{u-1,v,w}(\x)| = 
|\eta_{u,v,w}(\x)| + 
\left\{\begin{array}{ll}
-1  & \textnormal{if $\max(u,w) \ge \kx$} \\
\phantom{-}1  & \textnormal{if $\max(u,w) < \kx$ and $u > w$} \\
-1 & \textnormal{if $\max(u,w) < \kx$ and $u \le w$}
\end{array}\right.
\]
An identical equality holds if we exchange the roles of $u$ and $w$.
\end{lemma}
\begin{proof}
The proof is analogous to Lemma~\ref{lemma:still_minimal_2_claim_1};
we observe the effect of subtracting $1$ from $u$ in both length
formulas in Lemma~\ref{lemma:length_formula}.
\end{proof}

We now prove Lemma~\ref{lemma:still_minimal_2}

\begin{proof}[Proof of Lemma~\ref{lemma:still_minimal_2}]
If $\kx > \max(u',w')$, the conclusion follows directly from Lemma~\ref{lemma:still_minimal}.  
We address the case when $\kx = \max(u',w')$.  Let
$\y \in \BB_v^{u',w'}$.  First note that because $\max(u,w) < \max(u',w')$, we also have $\y \in \Bvuw$.
As $\x \in \Bvuw$ is minimal, we know that $\x \ords \y$.  We will show that
$\x <_{u',w'} \y$, which proves the lemma.

If $\ky \ge \kx$, then we use
the same length formula to compute all the lengths in the next equation, whether we consider $\x$ and $\y$ in $\Bvuw$ or $\BB_v^{u',w'}$.  It follows that
\[
|\eta_{u,v,w}(\x)| - |\eta_{u',v,w'}(\x)| = |\eta_{u,v,w}(\y)| - |\eta_{u',v,w'}(\y)|.
\]
That is, the effect on the path length by changing $u$ and $w$ to $u'$ and $w'$ is the
same for $\x$ and $\y$. Thus $\x <_{u,w} \y$ if and only if
$\x <_{u',w'} \y$.

For the remainder of the proof, we assume that $\ky < \kx$; in this case we may need different length formulas from Lemma~\ref{lemma:length_formula} to compute $|\eta_{u,v,w}(\x)|$ and $|\eta_{u,v,w}(\y)|$, as well as $|\eta_{u',v,w'}(\x)|$ and $|\eta_{u',v,w'}(\y)|$.
In this situation, we assume without loss of generality
that $\max(u',w') = w' \ge u'$.  
Since $\kx = \max(u',w')>\max(u,w)$,it follows that $w'>w$.
View this increase in value as repeated additions of the number $1$,
and apply
Lemma~\ref{lemma:still_minimal_2_claim_1}  to conclude that
\[
|\eta_{u,v,w'}(\x)| - |\eta_{u,v,w}(\x)| = \Delta,
\]
where $|\Delta| \le w'-w$.  

To compute the analogous difference for $\y$, let $w'' = w+\epsilon$ for some $\epsilon \leq w'-w$.
If $\max(u,w'') \geq \ky$ then $|\eta_{u,v,w''+1}(\y)| = |\eta_{u,v,w''}(\y)| + 1$.
If $\max(u,w'') < \ky$ then the change in path length depends on the ordinal relationship between $u$ and $w''$, in which case we have
\[
|\eta_{u,v,w''+1}(\y)| - |\eta_{u,v,w''}(\y)| = |\eta_{u,v,w''+1}(\x)| - |\eta_{u,v,w''}(\x)|.
\]
Combining these two possibilities yields
\[
|\eta_{u,v,w'}(\y)| - |\eta_{u,v,w}(\y)| \ge \Delta.  
\]

To analyze the analogous change in path length as the $u$ coordinate is varied, we are hampered by the fact that we do not know the ordinal relationship between $u$ and $u'$.
However, we do know that $\ky < \kx = \max(u',w')$, and hence we use the first length formula in Lemma~\ref{lemma:length_formula} to compute both
 $|\eta_{u,v,w}(\x)|$ and $|\eta_{u,v,w}(\y)|$, as well as $|\eta_{u',v,w'}(\x)|$ and $|\eta_{u',v,w'}(\y)|$.
Thus
\[
|\eta_{u',v,w'}(\x)| - |\eta_{u,v,w'}(\x)| = 
|\eta_{u',v,w'}(\y)| - |\eta_{u,v,w'}(\y)| = u'-u.
\]

Combining our analysis, we have
\[
|\eta_{u',v,w'}(\x)| - |\eta_{u,v,w}(\x)| \le
|\eta_{u',v,w'}(\y)| - |\eta_{u,v,w}(\y)|,
\]
or equivalently,
\[
|\eta_{u',v,w'}(\x)| - |\eta_{u',v,w'}(\y)| \le
|\eta_{u,v,w}(\x)| - |\eta_{u,v,w}(\y)| \le 0,
\]
where the right inequality follows from the fact that $\x \ords \y$.
If the inequality is strict, it follows that $\x <_{u',w'} \y$.
If there is equality, then there must be a lexicographic reduction from $\y$ to $\x$.  
Since the digits of $\x$ and $\y$ do not change whether we consider them in $\Bvuw$ or $\BB_{v}^{u',w'}$, the same lexicographic reduction allows us to conclude that $\x <_{u',w'} \y$
Thus in either case, $\x$ is a minimal vector in $\BB_v^{u',w'}$.
\end{proof}

The next lemma covers the special case when $n|v$, and thus if $\x \in \Bvuw$ is a minimal vector, we know that $x_0 = 0$.  

\begin{lemma}
\label{lemma:change_in_v}
Let $g = t^{-u}a^vt^w$ where $n|v$, and $\x \in \Bvuw$ is a minimal vector with $\kx > \max(u,w)$.  Then
$\y$ is a minimal vector in $\BB_{v'}^{u,w}$ where $y_i = x_{i+1}$ for $0 \leq i \leq \kx-1$ and $v' = \frac{v}{n}$.
\end{lemma}
\begin{proof}
Suppose that $\y \in {\mathcal B}_{v'}^{u,w}$ is not minimal.  Then there is $\z \in \LL_0$
such that $\y + \z \ords \y$.  Define $\z' \in \LL_0$ by prepending a digit $0$ to $\z$,
so $z'_0 = 0$ and $z'_i = z_{i-1}$ for $i>0$.  
As the digits of $\x$ and $\y$ are the identical but
simply shifted by one index, we have $\Vert \x \Vert_1 = \Vert \y \Vert_1$
and $\Vert \x + \z' \Vert_1 = \Vert \y + \z \Vert_1$.
It follows as well that $\ky = \kx-1 \geq \max(u,w)$.
Additionally, the change in vector
lengths is the same, so $\kx - k_{\x + \z'} = \ky - k_{\y + \z}$, and any relevant lexicographic
change occurs in both pairs of vectors.
If $k_{\y + \z} \ge \ky\geq \max(u,w)$, then we use the second length formula in Lemma~\ref{lemma:length_formula} to conclude that
\[
\eta_{u,v,w}(\x + \z') - \eta_{u,v,w}(\x) = \eta_{u,v,w}(\y + \z) - \eta_{u,v,w}(\y).
\]
As any relevant lexicographic change occurs in both pairs of vectors, and we know that $\y + \z \ords \y$, it follows that $\x + \z' \ords \x$, a contradiction.

If $k_{\y + \z} < \ky$ we do not know the ordinal relationship between $k_{\x + \z'}$, respectively
$k_{\y + \z}$, and $\max(u,w)$.  However, we can apply Lemma~\ref{lemma:length_formula_change} to compute
\begin{align*}
|\eta_{u,v,w}(\y + \z)| - |\eta_{u,v,w}(\y)| &= 
\Vert \y + \z \Vert_1 - \Vert \y \Vert_1 - 2\max(0, \ky - \max(k_{\y+\z},u,w)) \\
|\eta_{u,v,w}(\x + \z')| - |\eta_{u,v,w}(\x)| &= 
\Vert \x + \z' \Vert_1 - \Vert \x \Vert_1 - 2\max(0, \kx - \max(k_{\x+\z'},u,w)). \\
\end{align*}
Recall that $\Vert \x \Vert_1 = \Vert \y \Vert_1$
and $\Vert \x + \z' \Vert_1 = \Vert \y + \z \Vert_1$, 
and $y_i = x_{i+1}$ for $0 \leq i \leq \ky - \kx-1$.
It follows that $\ky - \max(k_{\y+\z},u,w) \le  \kx - \max(k_{\x+\z'},u,w)$.
Thus
\[
|\eta_{u,v,w}(\y + \z)| - |\eta_{u,v,w}(\y)| \ge |\eta_{u,v,w}(\x + \z')| - |\eta_{u,v,w}(\x)|.
\]
As $\y + \z \ords \y$, it follows that $\x + \x' \ords \x$, a contradiction.
\end{proof}

Let $g = t^{-u}a^vt^w$ and $\x \in \Bvuw$.
When computing $\kappa_1(g)$ it is often more straightforward to determine $l(g^{t^{\pm 1}})$ than $l(g^{a^{\pm 1}})$.  
We introduce a restriction which will allow us to easily determine when the vectors corresponding to $l(g^a)$ and $l(g^{\ai})$ are minimal in the appropriate $\Bvuw$.
This restriction is not meant
to be exhaustive; rather it gives us control over a broad range
of vectors $\x$ for which $\eta_{u,v,w}(\x)$ has shape $3$ or $4$.

Let $g=t^{-u}a^vt^w$ with 
$
v_+=n^u+v-n^w \text{ and } v_-= -n^u+v+n^w.
$
Recall that $g^a = t^{-u}a^{v_+}t^w$ and $g^{\ai} = t^{-u}a^{v_-}t^w$.
Beginning with a vector $\x \in \LL_v$,
\begin{itemize}[itemsep=5pt]
    \item to obtain a vector in $\LL_{v_+}$, one can add the digit $1$ to $x_u$ and subtract $1$ from $x_w$.  Denote the resulting vector by $\rho_{u,-w}(\x)$, and
    \item to obtain a vector in $\LL_{v_-}$, one can subtract the digit $1$ from $x_u$ and add $1$ to $x_w$.  Denote the resulting vector by $\rho_{-u,w}(\x)$.
\end{itemize}
Note that $\eta_{u,v_+,w}(\rho_{u,-w}) = g^{a}$ and $\eta_{u,v_-,w}(\rho_{-u,w}) = g^{\ai}$.
It is possible that the length of $\rho_{u,-w}(\x)$ or $\rho_{-u,w}(\x)$ will differ from the length of $\x$.
If $u > \kx$ or $w > \kx$, forming $\rho_{u,-w}(\x)$ and $\rho_{-u,w}(\x)$ will change digits with indices greater than $\kx$, creating a longer vector.
If, on the other hand, $w<u = \kx$ and $x_u = 1$, the length of $\rho_{-u,w}(\x)$ will be less than the length of $\x$.

We say that $\x \in \Bvuw$ is {\em strongly minimal} if
both $\rho_{u,-w}(\x) \in \BB_{v_+}^{u,w}$ and
$\rho_{-u,w}(\x) \in \BB_{v_-}^{u,w}$ are minimal.  
When $n$ is odd, $\x$ will be strongly minimal if we restrict $1<|x_{u}|,|x_{w}| < \fn$.  When $n$ is even, $\x$ will be strongly minimal if we restrict $1<|x_{i}|,|x_{w}| < \frac{n}{2}-3$.    While these are not the only conditions which guarantee that an element is strongly minimal, they are easily met for odd $n$ and even $n>10$.  The notion of a strongly minimal element allows for a clean statement of later theorems.

The following lemma gives a simple example of of a family of elements $g \in BS(1,n)$ whose conjugation curvature satisfies $\kappa_1(g) = 0$.

\begin{lemma}\label{lemma:zero_curv}
Let $g= t^{-u}a^vt^u$ for $u \in \N$ and $v \in \Z$ and
let $\x \in \Bvuw$ minimal. If $u \notin \{\kx-1,\kx,
\kx+1\}$ then
$\kappa_1(g) =0$.
\end{lemma}

\begin{proof}
As $u \notin \{\kx-1,\kx,\kx+1\}$, it follows from 
Lemma~\ref{lemma:still_minimal} that $\x$ is also minimal in both
$\BB_v^{u-1,u-1}$ and $\BB_v^{u+1,u+1}$.   
The assumption that $w=u$ immediately implies that $g = g^a = g^{a^{-1}}$. 
The restriction on the values of $u$ allows us to use the same length formula from  Lemma~\ref{lemma:length_formula}  to compute both $l(g), \ l(g^t)$ and $l(g^{\ti})$, and thus
\begin{align*}
\kappa_1(g)l(g)
&= l(g) - \frac{1}{4}\left(l(g^t) + l(g^\ti) + l(g^a) + l(g^\ai)\right) \\
& = l(g) - \frac{1}{4}\left(|\eta_{u-1,v,u-1}(\x)| + |\eta_{u+1,v,u+1}(\x)| + 2l(g)\right) \\
& = l(g) - \frac{1}{4}\left( l(g) + l(g) + 2l(g)\right),
\end{align*}
where the last equality follows from Lemmas~\ref{lemma:still_minimal_2_claim_1} and~\ref{lemma:still_minimal_2_claim_2}.
 As $l(g) \geq 1$, this simplifies
to $\kappa_1(g) = 0$.
\end{proof}

When $g = t^{-u}a^vt^w$ is represented by a geodesic of shape 3 or 4, Theorems~\ref{thm:type1curv} and~\ref{thm:type2curv} provide broad conditions on when $\kappa_1(g)$ is negative or zero.  
Later theorems in Section~\ref{sec:curvature} allow for analogous conclusions about $\kappa_r(g)$ for a range of values of $r$.
In order to express a variety of conditions in a concise way,
it will be helpful to introduce the notation $\delta_C$, where $C$ is a logical
expression and $\delta_C = 1$ if $C$ is satisfied and $0$ otherwise.

\begin{theorem}\label{thm:type1curv}
Let $g = t^{-u} a^vt^w \neq e$ and let $\x \in \Bvuw$ be strongly
minimal with $\kx> \max(u,w)$.  The conjugation curvature
$\kappa_1(g)$ then satisfies:
\begin{enumerate}
    \item $\kappa_1(g) = 0$ iff
    $\delta_{u\ne w}(\delta_{x_u=0} + \delta_{x_w=0}) = 0$ and either 
    $uw > 0$ or $n|v$.
    \item $\kappa_1(g) < 0$ otherwise.
\end{enumerate}
\end{theorem}

Any $g \in BS(1,n)$ to which Theorem~\ref{thm:type1curv}
applies can be represented by a geodesic path of shape $3$ or $4$. 
When $g$ is represented by a geodesic path of shape $1$ or $2$
we obtain an analogous theorem,  stated below for completeness. 
Its proof involves checking many cases nearly identical to those
in the proof of Theorem~\ref{thm:type1curv}.  As we do not need
Theorem~\ref{thm:type2curv} in later results, we include its
statement but leave its proof to the interested reader.

\begin{theorem}\label{thm:type2curv}
Let $g = t^{-u} a^vt^w \neq e$  and let $\x \in \Bvuw$ be strongly
minimal with $\kx \le \max(u,w)$.
The conjugation curvature $\kappa_1(g)$ then satisfies:
\begin{enumerate}[itemsep=5pt]
\item $\kappa_1(g)=0$ iff $u,w>\kx$ with $u = w$.
\item $\kappa_1(g)<0$ otherwise.\qed
\end{enumerate}
\end{theorem}

\begin{proof}[Proof of Theorem~\ref{thm:type1curv}]
We compute the effect on word length of conjugation by each
of the generators of $BS(1,n)$.  Since $\kx>\max(u,w)$, the second length formula from Lemma~\ref{lemma:length_formula} is always used to compute $l(g)$.

{\bf Case 1: conjugation by $t$.} If  $uw>0$, then as computed above
    $g^t = t^{-(u-1)}a^vt^{w-1}$ and $\kx  > \max(u-1,w-1)$,
    so it follows from Lemma~\ref{lemma:still_minimal} that $\x$ is minimal in $\BB_v^{u-1,w-1}$ and we use the second length formula from Lemma~\ref{lemma:length_formula} to compute $l(g^t)$.
    It follows that
    $l(g^t) = l(g)$.  
    
    If $uw=0$, then $g^t = t^{-u}a^{vn}t^{w}$.
    Let $\y = (y_0,y_1, \cdots ,y_{\ky})$ be obtained from $\x$ by defining $y_0=0$ and $y_i = x_{i-1}$ for $1 \leq i \leq \ky$. 
    Then $\y$ is the vector in ${\mathcal B}_{vn}^{u,w}$  such that $\eta_{u,vn,w}(\y) = g^t$.
    Observe that $\ky = \kx + 1$, so $\max(u,w) < \kx < \ky$.  
    If $\y$ were not minimal, it would follow from Lemma~\ref{lemma:change_in_v} that $\x$ was not minimal, a contradiction.
    As $\max(u,w) < \kx < \ky$ we use the second length formula in Lemma~\ref{lemma:length_formula} to compute $l(g^t)$, and thus $l(g^t) = |\eta_{u,nv,w}(\y)| = |\eta_{u,v,w}(\x)|+2 = l(g)+2.$

    In summary:
    \[
    l(g^t) = \left\{\begin{array}{ll} l(g) & \textnormal{if $uw > 0$} \\
                                      l(g)+2 & \textnormal{if $uw=0$.}
                                      \end{array}\right.
    \]
    
{\bf Case 2: conjugation by $\ti$.} If $n \nmid v$, then
    $g^\ti = t^{-(u+1)}a^vt^{w+1}$, 
    and $ \kx \ge \max(u+1,w+1)$.   
    Lemma~\ref{lemma:still_minimal_2} guarantees that
    $\x \in \BB_v^{u+1,w+1}$ is minimal.
    As $\kx \ge \max(u+1,w+1)$, we use the
    second length formula from Lemma~\ref{lemma:length_formula} to compute $l(g^{\ti})$, noting that the two formulas agree when 
    $\kx = \max(u,w)$.  It follows that
    $l(g^\ti) = l(g)$.
    
    If $n|v$, then $uw=0$ and $g^{\ti}=t^{-u}a^{\frac{v}{n}}t^w$.
    Note that since $n|v$, the least significant digit of $\x \in \Lv$ is $0$.
    Let $\y = (y_0,y_1, \cdots ,y_{\ky})$ be obtained from $\x$ by defining  $y_i = x_{i+1}$ for $0 \leq i \leq \ky = \kx-1$, that is, each entry of $\x$ is
    shifted left by one position to create $\y$.
    Then $\y$ is the vector in ${\mathcal B}_{\frac{v}{n}}^{u,w}$ corresponding to $g^t$.
    It follows from Lemma~\ref{lemma:change_in_v} that $\y \in \BB_{\frac{v}{n}}^{u,w}$ is minimal.
    As $\ky \geq \max(u,w)$, we use the second length formula in Lemma~\ref{lemma:length_formula} to compute $l(g^{\ti})$ and see that
    $l(g^{\ti}) = |\eta_{u,\frac{v}{n},w}(\y)| = |\eta_{u,v,w}(\x)|-2 = l(g)-2$
    In summary,
    \[
    l(g^t) = \left\{\begin{array}{ll} 
             l(g)   & \textnormal{if $n\nmid v$} \\
             l(g)-2 & \textnormal{if $n|v$.}
        \end{array}\right.
    \]
    
{\bf Case 3: conjugation by $a^{\pm 1}$.} Let $v_+ = n^u+v-n^w$ and $\x_+ = \rho_{u,-w}(\x)$
    and $v_- = -n^u+v+n^w$ and $\x_- = \rho_{-u,w}(\x)$.
    Since $\x$ is strongly minimal, we have that $\x_+ \in \BB_{v_+}^{u,w}$
    and $\x_- \in \BB_{v_-}^{u,w}$ are minimal,
    so computing $l(g^a)$ and $l(g^\ai)$ reduces to applying the length
    formula to compute $|\eta_{u,v_+,w}(\x_+)|$ and $|\eta_{u,v_-,w}(\x_-)|$.
    
    The assumptions that $\x$ is strongly minimal and $\kx > \max(u,w)$ ensure that 
    changing the digits at indices $u$ and $w$ does not alter the length of $\eta_{u,v,w}(\x)$;
    hence $k_{\x_+} = k_{\x_-} = \kx$, and the change in length between $\eta_{u,v,w}(\x)$ and $\eta_{u,v_+,w}(\x_+)$, respectively $\eta_{u,v,w}(\x)$ and $\eta_{u,v_-,w}(\x_-)$,
    reduces to the change in absolute value between the coordinates with indices $u$ and $w$.
    
    If $u=w$, then the changes to $x_u=x_w$ sum to zero, and we have $l(g^a) = l(g^\ai) = l(g)$.
    Otherwise,
    \begin{align*}
    l(g^a) &= |\eta_{u,v_+,w}(\x_+)|
            = l(g) + |x_u + 1| - |x_u| + |x_w - 1| - |x_w| \\
    l(g^\ai) &= |\eta_{u,v_-,w}(\x_-)| 
              = l(g) + |x_u - 1| - |x_u| + |x_w + 1| - |x_w|
    \end{align*}
    Observe that if $x_u \ne 0$, then $|x_u + 1| + |x_u - 1| - 2|x_u| = 0$, 
    and if $x_u = 0$, then $|x_u + 1| + |x_u - 1| - 2|x_u| = 2$; analogous statements hold for $x_w$.
    Thus we have shown that
    \[
    l(g^a) + l(g^\ai) = 2l(g) + 2\delta_{u \neq w}(\delta_{x_u=0} + \delta_{x_w=0}).
    \]
Combining the above computations with those for $l(g^{t^{\pm 1}})$,  note that
if $n|v$, then $uw=0$, and hence if $uw>0$, then $n\nmid v$, and hence
\[
l(g^t) + l(g^\ti) + l(g^a) + l(g^\ai)
=
\left\{\begin{array}{ll} 
4l(g) + 2\delta_{u\ne w}(\delta_{x_u=0} + \delta_{x_w=0})
& \textnormal{if $uw>0$ or $n|v$} \\
4l(g) + 2 + 2\delta_{u\ne w}(\delta_{x_u=0} + \delta_{x_w=0})
& \textnormal{if $uw=0$ and $n \nmid v$}
\end{array}\right.
\]
The theorem follows immediately from this formula.
\end{proof}

The following two facts arise in the proof of
Theorem~\ref{thm:type1curv}, and we state them below in Lemma~\ref{lemma:type1_facts} for easy
reference, noting that the second is true in greater generality
than the context of Theorem~\ref{thm:type1curv}.
\begin{lemma}\label{lemma:type1_facts}
Let $g = t^{-u}a^vt^w \in BS(1,n)$ where $\x \in \Bvuw$ is minimal and $0 < \max(u,w) \leq \kx-1$.
\begin{enumerate}[itemsep=5pt]
\item If $uw>0$ and $n \nmid v$ then $l(g) = l(g^t) = l(g^{t^{-1}})$.
\item If $\x$ is strongly minimal and $x_ux_w>0$  then 
$l(g) = l(g^a) = l(g^\ai)$.\qed
\end{enumerate}
\end{lemma}

\subsection{Sets of positive density in $BS(1,n)$}
\label{sec:sets_of_pos_density}

Theorem~\ref{corollary:shapes_injection}
states that the growth rate of the sequence $\OOn$
is the same as the growth rate of $BS(1,n)$.  In order to show
that a subset of elements of $BS(1,n)$ has positive density, we subdivide this subset according to word length, which is always computed with respect to the generating set $\{a,t\}$ for $BS(1,n)$.
Let $f(N)$ be the function which counts the number of elements in this subset of a given word length $N$.
We will show that the growth rate of $\{f(N)\}_{N \in \N}$ is identical to that of $\OOn$.

Let ${\mathcal Q}_n \subset \OO_n$ be the subset of geodesic words which do not begin with $\ti$ and end with a single $t$, omitting the word $t$, and let ${\mathcal Q}_n(N)$ be the set of such words with word length $N$; we denote the size of ${\mathcal Q}_n(N)$ by $q_n(N)$.
\begin{lemma}
\label{lemma:middle_vector}
The growth rate of $\{q_n(N)\}_{N \in \N}$ is the same as the growth rate of $\OO_n$, that is, $q_n(N) = \Theta(\lambda_n^N)$.
\end{lemma}
\begin{proof}
This follows from Lemma~\ref{lemma:fsa_growth} and the
structure of the finite automata $\OO_n$.  There are two strongly
connected components of $\OO_n$: the one containing only the state
$s_{\ti}$ and the one containing the digit expansions
of the states $s_i$ of $\DD_n$.  This latter component determines
the growth rate of $\OO_n$, so as long as we do not affect this
strongly connected component, we leave the growth rate unchanged.
Modify the finite state automata by:
\begin{itemize}
    \item Removing $s_\ti$
    \item Removing the state which accepts a string of the form $t^n$ from the set of accept states; this state is labeled as $s_{0,0}$ in Figure~\ref{fig:fsa_n=2}.
\end{itemize}
The resulting finite state automata accepts exactly those
paths in $\OO_n$ which do not begin with any power of $t^{-1}$
and which end with exactly one $t$.  That is, it accepts
exactly the geodesic words in $\mathcal{Q}_n$.  Although we have changed the accept
states, the set of edges in the main strongly connected component
is unchanged.  Thus
$\{q_n(N)\}_{N \in \N}$ and $\OOn$ have the same growth rate.  That is,
$q_n(N) = \Theta(\lambda_n^N)$.
\end{proof}

\subsection{Detecting minimal vectors}
To show that $BS(1,n)$ has a positive density of elements of positive, zero and negative conjugation curvature, we construct in each case a family of words by concatenating a prefix, ``middle'' and suffix.
The middle segment of each word is always chosen to be an element of ${\mathcal Q}_n$, so it is accepted by the finite state automaton adapted from $\OO_n$ described in Section~\ref{sec:sets_of_pos_density} and corresponds to a minimal vector. 
In the next sections, we will vary the prefix and suffix in order to construct
examples of elements with the desired conjugation curvature.
The following lemma will be useful to certify that
the growth rate of these special geodesics is comparable to
that of the whole group.
Recall that for any set ${\mathcal A}$ of elements of $BS(1,n)$, we will always use the notation $\mathcal{A}(N)$ for $N \in \N$ to denote the elements of ${\mathcal A}$ whose word length with respect to the generating set $\{a,t\}$ of $BS(1,n)$ is $N$.
\begin{lemma}\label{lemma:prefix_suffix_growth_rate}
Let $\mathcal{A} \subseteq BS(1,n)$ be a set of geodesic words of the form
$p\xi s$, where $p,s$ are constant and $\xi$ may be any word in $\mathcal{Q}_n$.
Then the growth rate of $\{|\mathcal{A}(N)|\}_{N \in \N}$ is the same as
the growth rate of $BS(1,n)$ and thus $\mathcal{A}$ has positive density in
$BS(1,n)$.
\end{lemma}
\begin{proof}
It follows from Theorem~\ref{corollary:shapes_injection} and Lemma~\ref{lemma:middle_vector} that
the growth rate of $\{q_n(N)\}_{N \in \N} = \{|\mathcal{Q}_n(N)|\}_{N \in \N}$
is the same as that of $BS(1,n)$.  By construction, we have
$|\mathcal{A}(N+|p|+|s|)| = |\mathcal{Q}_n(N)|$, so by
Lemma~\ref{lemma:exp_growth} the growth rates of
$\{|\mathcal{A}(N)|\}_{N \in \N}$ and $|\mathcal{Q}_n(N)|$ are the same,
and the lemma follows.
\end{proof}

The next series of lemmas show that words constructed in this way are geodesic, that is, when we look at the corresponding vector of consecutive exponents of the generator $a$ in each word, this vector is minimal.

Suppose $g = t^{-u}a^vt^w \in BS(1,n)$ is constructed as in Lemma~\ref{lemma:prefix_suffix_growth_rate}, for some choice of nonempty prefix, suffix and middle word $\xi \in {\mathcal Q}_n$.  Let $\x \in \Bvuw$ denote the associated vector of consecutive exponents of the generator $a$ in $g$, and let $\x'$ be the vector of consecutive exponents of the generator $a$ in $\xi$.
Note that $\x'$ is minimal
because $\xi \in {\mathcal Q}_n$.
The following series of lemmas presents simple criteria which allow us to conclude that $\x$ is minimal by relating the minimality of $\x \in \Bvuw$ to the minimality of $\x' \in {\mathcal B}_{v'}^{u',w'}$, for a choice of $u',v',w'$ specified below.

We make the following convention with regard to indexing $\x$ and $\x'$.  Let $\x = (x_0, \cdots ,x_\kx)$ and $\x' = (x_m, \cdots ,x_\ell)$ where $0 < m \leq \ell < \kx$.
When we consider $\x'$ as an independent vector, we will continue to write it as $\x' = (x_m, \cdots ,x_\ell)$ rather than shifting the indices so that they begin at $0$. 
We make this choice to retain the context of $\x' \subseteq \x$.
When we want to add a linear combination of $\LL_0$ basis vectors to $\x'$, we must index them accordingly and write $\x' + \sum_{i=m}^{\ell-1} \alpha_i \wi$ to obtain the vector $(x'_m-\alpha_mn, \cdots ,x'_\ell+\alpha_{\ell-1})$.
When we compute $\Sigma(\x')$, we evaluate the sum $\Sigma(\x') = \sum_{i=m}^\ell x'_in^{i-m}$.

Lemma~\ref{lemma:run_in_the_middle} shows that if $\x$ constructed in this way can be reduced at a run $\r \subseteq \x' \subseteq \x$ which does not contain the final digit of $\x'$, then $\x'$ can also be reduced at $\r$.  By ``$\subseteq$'' here we mean a subsequence of consecutive digits.  In what follows, we will use $v'$ to denote $\Sigma(\x')$, where the function $\Sigma $ is defined in Section~\ref{section:min_rep}.

\begin{lemma}
\label{lemma:run_in_the_middle}
Let $\x \in \Bvuw$ be constructed as above with $\kx > \max(u,w)$ and $\x' = (x_m, \cdots ,x_s) \subseteq \x$, where $\x' \in {\mathcal B}_{v'}^{0,s-m+2}$.
If $\r = (x_j, \cdots ,x_{\ell}) \subset \x'$ with $m \leq j \leq l<s < \kx$ is a run at which $\x$ can be reduced, then $\x'$ can be reduced at $\r$.
\end{lemma}

\begin{proof}
Let $\x,\x',u,v,w$ and $\r$ be as in the statement of the lemma, with $u'=0$ and $w' = s-m+2$.
As $\x$ can be reduced at $\r$, when $n \geq 3$ we have
$$\x + \delta \sum_{i=j}^{\ell} \wi \ords \x$$ for a choice of $\delta \in \{\pm 1\}$.
When $n=2$ the same inequality holds, where we can choose the coefficients of the $\wi$ to be identically $\delta$ because the run $r$ does not contain the final digit of $\x$, as discussed in Section~\ref{section:geodesics_n_even_2}.

Note that $\eta_{u',v',w'}(\x')$ is a geodesic of strict shape $1$ and $\eta_{u,v,w}(\x)$ is a geodesic of shape $3$, so we use different formulas to compute their length.  
As the suffix is nonempty, and $\r \subset \x'$ does not contain the final digit of $\x'$, the vectors the vectors $\x$ and $\x+\delta \sum_{i=j}^{\ell} \wi$ are the same length.
Thus we use the second length formula in Lemma~\ref{lemma:length_formula} to compute both $|\eta_{u,v,w}(\x)|$ and $|\eta_{u,v,w}(\x+\delta \sum_{i=j}^{\ell} \wi)|$.

If $\x'$ and $\x'+\delta \sum_{i=j}^{\ell} \wi$ have the same length, then both  $|\eta_{u',v',w'}(\x')|$ and $|\eta_{u',v',w'}(\x'+\delta \sum_{i=j}^{\ell} \wi)|$
are computed using the first length  formula in Lemma~\ref{lemma:length_formula}.
It might be the case that the length of $\x'+\delta \sum_{i=j}^{\ell} \wi$ is one less than then length of $\x'$.  
As the condition for the first length formula is that $\K{x'} \leq \max(u',w')$, we see that if the length of the vector decreases but $u'$ and $w'$ are unchanged, we use the same length formula from Lemma~\ref{lemma:length_formula} to compute $|\eta_{u',v',w'}(\x'+\delta \sum_{i=j}^{\ell} \wi)|$.
This ensures that the change in word length in either case reflects only the change in $\ell^1$ norm between the vectors.
Thus
\begin{align*}
|\eta_{u,v,w}(\x+\delta \sum_{i=j}^{\ell} \wi)| - |\eta_{u,v,w}(\x)|
&=
|\eta_{u',v',w'}(\x'+\delta \sum_{i=j}^{\ell} \wi)| - |\eta_{u',v',w'}(\x')| \\
&=
\Vert \x'+\delta \sum_{i=j}^{\ell} \wi \Vert_1 - \Vert \x' \Vert_1 \leq 0.
\end{align*}
If the final inequality is strict, it is clear that $\x'$ can be reduced at $\r$.  If there is equality, then the lexicographic reduction which occurs between $\x$ and $\x+\delta \sum_{i=j}^{\ell} \wi$  will also occur between $\x'$ and $\x'+\delta \sum_{i=j}^{\ell} \wi$ and hence $\x'$ can be reduced at $\r$, that is,
\[
\x'+ \delta \sum_{i=j}^{\ell} \wi <_{u',w'} \x'.
\]
\end{proof}

Lemma~\ref{lemma:next_digit_0} extends Lemma~\ref{lemma:run_in_the_middle} when $n$ is even to conclude that if $\x$ can be reduced at a run $\r \subseteq \x'\subset \x$ which contains the final digit $x_{\ell}$ of $\x'$ and $x_{\ell+1} = 0$, then $\x'$ can also be reduced at $\r$.

\begin{lemma}
\label{lemma:next_digit_0}
Let $n$ be even and $\x \in \Bvuw$ be constructed as in Lemma~\ref{lemma:run_in_the_middle} with $\kx > \max(u,w)$ and $\x' = (x_m, \cdots ,x_{\ell}) \subseteq \x$ with $\x' \in {\mathcal B}_{v'}^{0,l-m+2}$ and $v' = \Sigma(\x')$.
If $\r = (x_j, \cdots ,x_{\ell}) \subseteq \x'$ is a run at which $\x$ can be reduced and $x_{\ell+1}=0$, then $\x'$ can be reduced at $\r$.
\end{lemma}

\begin{proof}
As $\x$ can be reduced at $\r$, when $n \geq 4$ it follows immediately that
\[
\x+\epsilon_{\r} \sum_{i=j}^{\ell} \wi \ords \x.
\]
When $n=2$ the same inequality holds, where we can choose the coefficients of the $\wi$ to be identically $\epsilon_{\r}$ because the run $r$ does not contain the final digit of $\x$, as discussed in Section~\ref{section:geodesics_n_even_2}.

By construction, $\s$ must have at least two digits, as $x_{\ell+1} = 0$ and $\x$ cannot end with the digit $0$.
This ensures that $\r$ does not contain the final two digits of $\x$ and hence the vectors $\x$ and $\x+\epsilon_{\r} \sum_{i=j}^{\ell} \wi$ have the same length.
As $\kx > \max(u,w)$, the  the second length formula in Lemma~\ref{lemma:length_formula} is used to compute both  $|\eta_{u,v,w}(\x)|$ and $|\eta_{u,v,w}(\x + \epsilon_{\r}\sum_{i=j}^{\ell} \wi)|$.  

When considering $\x'$, take $u'=0$ and $w'=l-m+2$.  
We now show that the lengths of $\eta_{u',v',w'}(\x')$ and $\eta_{u,v,w}(\x'+\epsilon_{\r}\sum_{i=j}^{\ell} \wi)$  are computed using the same word length formula by noting the ordinal relationship between $w'$ and the length of the vector in each case.
\begin{itemize}[itemsep=5pt]
    \item As $\x' \in {\mathcal B}_{v'}^{0,l-m+2}$, we have $w'=l-m+2 > l-m+1 = \K{\x'}$.
    \item If $\y = \x'+\epsilon_{\r} \sum_{i=j}^{\ell} \wi$, then $\K{\y} = \K{\x'}+1 = l-m+2$, so $w'=l-m+2 = \K{\y}$.
\end{itemize}
As the two length formulas in Lemma~\ref{lemma:length_formula} agree when $\max(u',w') = \ky$, we see that the lengths of both $\eta_{u',v',w'}(\x')$ and $\eta_{u,v,w}(\x'+\epsilon_{\r} \sum_{i=j}^{\ell} \wi)$ are computed using the first length formula in Lemma~\ref{lemma:length_formula}, which does not rely on the length of $\x'$ or $\y$.
Thus in both cases the difference in word length between the geodesics arising from each pair of vectors  is exactly the difference in $\ell^1$ norm between the vectors.

We then compute
\begin{align*}
0 \leq  |\eta_{u,v,w}(\x)|-|\eta_{u,v,w}(\x+\epsilon_{\r} \sum_{i=j}^{\ell} \wi)| 
&=
\wt(\r) + |x_{\ell+1}| - |x_{\ell+1} + \epsilon_{\r}| \\
&= \wt(\r) + |\epsilon_{\r}| \\
&= |\eta_{u',v',w'}(\x')| - |\eta_{u',v',w'}(\x'+\epsilon_{\r} \sum_{i=j}^{\ell} \wi)|
\end{align*}
where $\wt(\r)$ is defined in Section~\ref{section:geodesics_n_even}. The equality in the first line is proven in Lemma~\ref{lemma:change_formula} for $n \geq 4$ and is easily checked for $n=2$.  The transition from the first line to the middle line relies on the fact that $x_{\ell+1} =0$. The transition from the last line to the middle line relies of the fact that $\x'+\epsilon_{\r} \sum_{i=j}^{\ell} \wi$ has a leading coefficient of $1$ not present in $\x'$.

If the initial inequality is strict, it is clear that $\x'$ can be reduced at $\r$.  If there is equality, then the lexicographic reduction which occurs between $\x$ and $\x+\epsilon_{\r} \sum_{i=j}^{\ell} \wi$  will also occur between $\x'$ and $\x'+\epsilon_{\r}\sum_{i=j}^{\ell} \wi$, as $\x' \subseteq \x$ and has highest index equal to $\ell$.  Thus $\x'$ can be reduced at $\r$, that is,
\[
\x'+ \epsilon_{\r} \sum_{i=j}^{\ell} \wi <_{u',w'} \x'.
\]
\end{proof}

Combining the previous two lemmas allows us to show that if if $n$ is even and $g = t^{-u}a^vt^w$ is constructed from a prefix, suffix and $\xi \in {\mathcal Q}_n$ with a bound on the absolute value of the exponents in $p$ and $s$, then the resulting vector $\x$ of consecutive exponents of the generator $a$ is minimal in $\Bvuw$.  We prove Lemma~\ref{lemma:prefix-suffix} only for even $n \geq 4$. If $n$ is odd and $\x \in \Bvuw$ is as above, the minimality of $\x$ is determined solely by inspection of the final two digits of $\x$, as described in Lemma~\ref{lemma:odd_box}.

\begin{lemma}
\label{lemma:prefix-suffix}
Let $n \geq 4$ be even and suppose $g = p \xi s$ where $\xi \in {\mathcal Q}_n$ and $p$ and $s$ are as follows:
\begin{itemize}[itemsep=5pt]
    \item $p = t^{-u} a^{p_0}ta^{p_1} \cdots ta^{p_m}ta^0t$ for $m \geq 0$ , where $\p = (p_0,p_1, \cdots ,p_m,0)$ is the vector of consecutive exponents of the generator $a$ in $p$ and $u \geq 0$,
    \item $s = a^0ta^{s_0}ta^{s_1}t \cdots ta^{s_{m'}}t^{-h}$ for $m' \geq 0$, where $\s = (0,s_0,s_1, \cdots  s_{m'})$ is the vector of consecutive exponents of the generator $a$ in $s$ and $0 \leq h \leq m'$, and
    \item $(s_{m'-1},s_{m'})$ is not equal to either $(\delta(\fne-1),-\delta)$ or $(\delta \fne,-\delta)$, for $\delta\in \{\pm 1\}$.
\end{itemize}
Further assume that for all $i$, $|p_i| \leq \fne-1$ and $|s_i| < \fne-1$.  

Let $\x'$ denote the vector of consecutive exponents of the generator $a$ in $\xi$, and $\x = \p \x' \s$ the analogous vector for $g$.  Then $\x \in \Bvuw$, where $v = \Sigma(\x) = \sum_{i=0}^{\kx} x_i n^i$, $w = \kx-h$ and $\x$ is minimal.
\end{lemma}

Using the notation in the statement of Lemma~\ref{lemma:prefix-suffix}, if $\r \subset \x$ denotes a run, we write $\r \cap \p$ to denote any common digits of $\x$ contained in both $\r$ and $\p$, with the analogous definition for $\s \cap \r$.

\begin{proof}
Note that by construction,  $\x \in \Bvuw$, where $v = \Sigma(\x) = \sum_{i=0}^{\kx} x_i n^i$ and we choose $w = \kx -h$ so that $g$ is a geodesic of shape $3$. We must show that $\x$ is minimal.
The definition of ${\mathcal Q}_n$ ensures that $\x' \in {\mathcal B}_{v'}^{0,\K{\x'}+1}$, where $v' = \Sigma(\x')$, and $\x'$ is minimal.

Suppose $\x$ is not minimal.  Inspecting the final two digits of $\x$, that is, the final two digits of $\s$, we see that Lemma~\ref{lemma:even_reduction_end} does not apply, and it follows from Proposition~\ref{lemma:even_minimal_characterization} that $\x$ contains a run $\r$ at which it can be reduced.  

The final exponent of $t$ in the definition of $\s$ implies that $\max(u,w) < \kx$.
The digit restrictions on $\p$ force $\r \cap \p = \emptyset$, and hence when $n \geq 4$, no run has its first digit in $\p$.  The fact that the first digit of $\s$ is $0$ implies that $\r \cap \s = \emptyset$, and hence no run with first digit in $\x'$ can be continued into $\s$.  Thus we conclude that $\r \subset \x'$.

If $\r$ does not contain the last digit of $\x'$, it follows from Lemma~\ref{lemma:run_in_the_middle} that $\x'$ can be reduced at $\r$, that is, $\x'$ is not minimal, a contradiction.  
Assuming that $\r$ contains the last digit of $\x'$, as the initial digit of $\s$ is $0$, it follows from Lemma~\ref{lemma:next_digit_0} that  $\x'$ can be reduced at $\r$, that is, $\x'$ is not minimal, a contradiction.  
Thus we conclude that $\x$ is minimal.
\end{proof}

The following remark codifies the changes we consider when $g \in BS(1,n)$ is conjugated by a single generator as well as a string of generators.  It will be referenced repeatedly throughout the following sections. 

\begin{remark}
\label{remark:conjugation}
Let $g = t^{-u}a^vt^w$ with $\x \in \Bvuw$ a minimal vector. 
For any $q \in BS(1,n)$ define $u(q),v(q)$ and $w(q)$ so that $g^q = t^{-u(q)}a^{v(q)}t^{w(q)}$ in normal form.
We make the following observations about $g^c$ for $c \in \{t^{\pm 1},a^{\pm 1}\}$.
\begin{itemize}[itemsep=5pt]
\item When $c = t^{\pm 1}$, we have $|u-w| = |u(c)-w(c)|$ and $v(c) = v$.  Moreover, $\x$ can be viewed as an element of ${\mathcal B}^{u\mp1,w\mp1}_{v}$, and it is again minimal.
\item When $c = a^{\pm 1}$ we have $u(c) = u$ and $w(c) = w$, so it is again true that $|u-w| = |u(c)-w(c)|$.  As discussed earlier, $v(c) = v_+ = n^u+v-n^w$ when $c = a$ and  $v(c) = v_-= -n^u+v+n^w$ when $c = \ai$.  
Observe that $\rho_{-u,w}(\x)$ is a vector in ${\mathcal B}^{u(c),w(c)}_{v_+}$ in the former case, and $\rho_{u,-w}(\x)$ is a vector in ${\mathcal B}^{u(c),w(c)}_{v_-}$; in either case we denote this vector as $\x(c)$.
\end{itemize}
When conjugating $g$ by $q_1q_2 \cdots q_n$, following the above steps for each successive conjugation by $q_i$ creates a vector $\x(q) \in {\mathcal B}^{u(q),w(q)}_{v(q)}$.  
In the remainder of this paper, given $g,q$ and $\x$, the notation $\x(q)$ denotes the vector created in this way, which may or may not be minimal.
Note that $\eta_{u(q),v(q),w(q)}(\x(q)) = g^q$.

Moreover, the above two conditions imply that $|u-w| = |u(q)-w(q)|$.  This fact will be useful when $\x(q)$ is minimal and $|\eta_{u,v,w}(\x)|$ and $|\eta_{u(q),v(q),w(q)}(\x(q))|$ are both computed using the second length formula in Lemma~\ref{lemma:length_formula}.
\end{remark}

\subsection{A positive density set of elements with positive
conjugation curvature}
When considering positive conjugation curvature, we require
$r=1$ and provide slightly different examples depending on
the parity of $n$.

\begin{theorem}
\label{thm:positive_kappa}
The group $BS(1,n)$ has a positive density of elements $g$ with $\kappa_1(g) >0$, for $n \geq 3$.
\end{theorem}

\begin{proof}
For $n \geq 3$, let ${\mathcal P}_n$ denote the set of words of the form $ g=p \xi s$ where $\xi \in {\mathcal Q}_n$ and 
\begin{itemize}[itemsep=5pt]
    \item $p = t^{-1}ata^{\fn}ta^{(-1)^n}ta^0t$ with vector of consecutive exponents of the generator $a$ given by  $\p = (1, \fn,(-1)^n,0)$, and 
    \item $s = a^0tata^0tat^{-2}$ with vector of consecutive exponents of the generator $a$ given by $\s=(0,1,0,1)$.
\end{itemize}
Note that by construction, $u =1$ and $w = \kx-2$.

Let $\x'$ denote the vector of consecutive exponents of the generator $a$ in $\xi$.  As $p$ ends with the generator $t$ and $\xi$ ends with the generator $t$, we can write $\x = \p \x' \s$ as the vector of consecutive exponents of the generator $a$ in $g$.  Let $\x = (x_0, \cdots ,x_{\kx})$ and $\x' = (x_m, \cdots ,x_s)$ for $0 < m \leq s < \kx$.

We first show that $\x \in {\mathcal B}_v^{1,\kx-2}$ is minimal, where $v = \sum_{i=0}^{\kx} x_in^i$.  The definition of ${\mathcal Q}_n$ ensures that $\x' \in {\mathcal B}_{v'}^{0,\K{\x'}+1}$ is minimal, with $u'=0$ and $w' = \K{\x'}+1$ and $v' = \Sigma(\x')$.  By construction, $\x \in {\mathcal B}_v^{1,\kx-2}$.  If $n$ is odd, it follows from Lemma~\ref{lemma:odd_box} that $\x$ is minimal.  

If $n \geq 4$ is even, it follows from Proposition~\ref{lemma:even_minimal_characterization} and inspection of $(x_{\kx-1},x_{\kx}) = (0,1)$ that $\x$ contains a run $\r = (x_j, \cdots ,x_{\ell})$ at which it can be reduced.  
If $\r \cap \p$ is nonempty and $n>4$, then $r = (\fne)$ and it is clear by inspection that $\x$ cannot be reduced at $\r$.  
If $n=4$, then $\p$ contains runs of the form $(2)$ and $(2,1)$, neither of which is a run at which $\x$ can be reduced.
Since the first digit in $\r$ must be $\pm \fne$, and any remaining digits either $\pm \fne$ or $\pm(\fne -1)$, we see that $r \subseteq \x'$.

If $\r$ does not contain the final digit of $\x'$, it follows from Lemma~\ref{lemma:run_in_the_middle} that $\x'$ can be reduced at $\r$, contradicting the fact that $\x' \in {\mathcal B}_{v'}^{0,\K{\x'}+1}$ is minimal.   
Thus it must be the case that $l=s$.  
As the first digit of $s$ is $x_{\ell+1}$ which equals $0$, 
it follows immediately from Lemma~\ref{lemma:next_digit_0} that $\x'$ can be reduced at $\r$, a contradiction.  
Thus we conclude that $\x \in {\mathcal B}_v^{1,\kx-2}$ is minimal.  
That is, every word in $\mathcal{P}_n$ is geodesic,
and it follows from Lemma~\ref{lemma:prefix_suffix_growth_rate} that $\mathcal{P}_n$
has positive density in $BS(1,n)$.  
It remains to show that every $g \in \mathcal{P}_n$ has $\kappa_1(g) > 0$, which requires us to compute $l(g^{t^{\pm 1}})$ and $l(g^{a^{\pm 1}})$.

First consider $l(g^t)$ and $l(g^{\ti})$.  
Given the normal form for $g^t$ and $g^{\ti}$, we see that $\x \in \Lv$ is a vector so that $\eta_{u \pm 1,v,w\pm 1}(\x) = g^{t^{\mp 1}}$ and $\kx > \max(u+1,w+1)$, it follows from Lemma~\ref{lemma:type1_facts} that $l(g^t) = l(g^{\ti}) = l(g)$. 

We next compute $l(g^a)$ and $l(g^{\ai})$. 
Recall that $g^a = t^{-u} a^{v_+}t^w$ where $u=1, \ w - \kx-2$
and $v_+ = n^u+v-n^w$.  
Consider $\rho_{u,-w}(\x) = \p' \x \s' \in \LL_{v_+}$, where $\rho_{-u,w}(\x)$ and $\rho_{u,-w}(\x)$ are defined in Section~\ref{sec:r=1}.
Here we have $\p' = (1, \fn+1,(-1)^n,0)$ and $\s' = (0,0,0,1)$
As written, $\rho_{u,-w}(\x)$ is not minimal.  However, adding the basis vector ${\bf w}^{(1)}$ 
and reusing the notation yields $\rho_{u,-w}(\x) = \p' \x \s'$, where
\[
\p' = 
\left\{\begin{array}{ll}
(1, -\fne+1, 2, 0) & \textnormal{if $n$ is even}\\
(1, -\fn, 0, 0) & \textnormal{if $n$ is odd}
\end{array}
\right.
\]
and $\s' = (0,0,0,1)$ is unchanged. We assess below whether this new form of $\rho_{u,-w}(\x)$ is minimal.

Recall that  $g^{\ai} = t^{-u}a^{v_-}t^w$ where $u=1, \ w = \kx-2$ and $v_- = -n^u+v+n^w$.  
Consider $\rho_{-u,w}(\x) = \p'' \x \s'' \in \LL_{v_-}$ where $\p'' = (1, \fn-1,(-1)^n,0)$ and $\s'' = (0,2,0,1)$.

If $n=3$ then $s''$ contains $2=\fn+1$ before the final digit.  However, adding the basis vector ${\bf w}^{(\kx-2)}$ yields the new suffix vector $\s'' = (0,-1,1,1)$.

If $n$ is odd it follows from Lemma~\ref{lemma:odd_box} that $\rho_{-u,w}(\x)$ and $\rho_{u,-w}(\x)$ are minimal.  If $n \geq 4$ is even, we assess whether $\rho_{-u,w}(\x)$ and $\rho_{u,-w}(\x)$ are minimal in two cases.
\begin{itemize}[itemsep=5pt]
    \item If $n=4$ it is possible that either the prefix or suffix vector in $\rho_{-u,w}(\x)$ or $\rho_{u,-w}(\x)$ contains a run of the form $(2)$.  It is easily checked that neither vector can be reduced at such a run.
    \item If $\rho_{-u,w}(\x)$  or $\rho_{u,-w}(\x)$ is not minimal, as the final two digits of $\s$ are unchanged from those of $\x$, Lemma~\ref{lemma:even_reduction_end} does not apply, and it follows from Proposition~\ref{lemma:even_minimal_characterization} that $\x$ contains a run $\r$ at which it can be reduced.  
   By inspection of the digits in the prefix and suffix, we see that $\r \subseteq \x'$.  It follows from Lemma~\ref{lemma:run_in_the_middle} or~\ref{lemma:next_digit_0} that $\x'$ is not minimal, a contradiction.
\end{itemize}
We conclude that $\rho_{u,-w}(\x)$ and $\rho_{-u,w}(\x)$ are minimal in ${\mathcal B}_{v_+}^{u,w}$ and ${\mathcal B}_{v_-}^{u,w}$, respectively.  

In all cases, we see that vectors $\x$, $\rho_{u,-w}(\x)$ and $\rho_{-u,w}(\x)$ all have the same length, and additionally 
the values of $u$ and $w$ are not altered when $g$ is conjugated by $a$ or $a^{-1}$.
Thus any changes between the word length of the corresponding elements arises from a change in the $\ell^1$ norm of the respective vectors.  
As $\Vert \rho_{-u,w}(\x) \Vert_1 = \Vert \x \Vert_1$ it follows that $l(g^{\ai}) = l(g)$.

Comparing the $\ell^1$ norms of $\x$ and $\rho_{u,-w}(\x)$, it follows that $l(g^a) = l(g) - 1$ if $n$ is even and $l(g^a) = l(g)-2$ is $n$ is odd. 
In all cases, these computations show that $\kappa_1(g) > 0$.
\end{proof}

\subsection{A positive density set of elements with zero conjugation curvature}
\label{sec:zero}

In this section we construct a family of elements $g \in BS(1,n)$ for $n \geq 3$ with $\kappa_r(g) =0$, where $r$ is allowed to assume any value between 1 and $\max(1,\lfloor \frac{n}{4} \rfloor-1)$. 
For any choice of $r$ in this range, the set of elements constructed has positive density in $BS(1,n)$.

\begin{theorem}\label{prop:zero_kappa_r}
The group $BS(1,n)$ for  $n \geq 3$ has a positive density of elements $g$ with
$\kappa_r(g) =0$ for $r \leq \max(1,\lfloor \frac{n}{4} \rfloor-1)$.
\end{theorem}

\begin{proof}
First suppose that $n \geq 8$, so that $\lfloor \frac{n}{4} \rfloor>1$.
Let $b=\lfloor \frac{n}{4} \rfloor$ and take $1 \leq r\leq b-1$.
Construct a set of elements ${\mathcal Z}_n$ so that $g \in {\mathcal Z}_n$ is the concatenation $p \xi s$, where $p$ and $s$ are words specified below and $\xi \in {\mathcal Q}_n$.  Namely,
\begin{itemize}[itemsep=5pt]
    \item $p = t^{-(r+1)}a^{b}ta^{b}t \cdots a^{b}ta^0t$ with $2r+2$ repetitions of $a^{b}$, and corresponding vector of consecutive exponents of $a$ given by $\p = (b,b, \cdots ,b,0)$ of length $2r+3$. 
    \item $s=a^0ta^{b}ta^{b} \cdots ta^{b}t^{-(r+4)}$ with $2r+5$ repetitions of $a^{b}$, and corresponding vector of consecutive exponents of $a$ given by $\s = (0,b,b, \cdots ,b)$ of length $2r+6$.
\end{itemize}
Observe that $u = r+1$ and $w = \kx - (r+4)$.
Let $\x'$ be the vector of consecutive exponents of the generator $a$ in $\xi$. According to the definition of ${\mathcal Q}_n$, we have $\x' \in {\mathcal B}_{v'}^{0,\K{\x'}+1}$, where $v' = \Sigma(\x')$, and by assumption $\x'$ is minimal.
Inspection of the initial and final letters of $p,\xi$ and $s$ allows us to write $\x = \p \x' \s$ as the vector of consecutive exponents of the generator $a$ in $g=p \xi s$.  By construction, $\x \in {\mathcal B}_{v}^{r+1,\kx-(r+4)}$.  We now show that $\x$ is minimal.

As $n \geq 8$, we know that $\fn-\lfloor \frac{n}{4} \rfloor > 1$, so $1 \leq r  < \lfloor \frac{n}{4}\rfloor -1$.  
Note that the final two digits of $\x$ are $(b,b)$; if $n$ is odd, it follows immediately from Lemma~\ref{lemma:odd_box} that $\x$ is minimal.  
When $n \geq 4$ is even, as $|p_i| \leq \fn-1$ and $|s_i| < \fn-1$ it follows from Lemma~\ref{lemma:prefix-suffix} that $\x \in \Bvuw$ is minimal.
This shows that when $n \geq 8$, every word in $\mathcal{Z}_n$ is geodesic
and it follows from Lemma~\ref{lemma:prefix_suffix_growth_rate} that $\mathcal{Z}_n$
has positive density in $BS(1,n)$.
It remains to show that for every $g \in \mathcal{Z}_n$ we have $\kappa_r(g) = 0$.
For later reference, as $\eta_{u,v,w}(\x)$ has shape $3$, we compute $l(g) = \Vert \x \Vert_1 + 2\kx - |u-w|$.

Let $q = q_1\dots q_r$.  In order to show $\kappa_r(g) = 0$, we will first
show that $\x(q)$ is minimal for all such $q$, where $\x(q)$ is defined in Remark~\ref{remark:conjugation}.  Since $\x$ is also minimal, we can then
use the appropriate length formula to compute first $l(g^q)$ and then $\kappa_r(g)$.

To prove the minimality of $\x(q)$, let $q'=q_1\dots q_j$ for $0 \le j \le r-1$.
For any generator $c \in \{a^{\pm 1}, t^{\pm 1}\}$, we will iteratively construct $\x(q'c)$
from $\x(q')$. 
Following the notation in Remark~\ref{remark:conjugation}, recall that
\begin{enumerate}[itemsep=5pt]
\item when $c=t^{\pm 1}$ we have 
\begin{enumerate}[itemsep=5pt]
    \item $u(q'c) = u(q')\pm 1$ and $w(q'c) = w(q') \pm 1$,
    \item $|u(q')-w(q')| = |u(q'c)-w(q'c)|$, and
    \item $\x(q'c) = \x(q')$.
\end{enumerate}
\item when $c = a^{\pm 1}$, we have $u(q'c) = u(q')$ and $w(q'c) = w(q')$, and
we construct $\x(q'c)$ from $\x(q')$ by altering two coordinates.
Namely, when $c = a$ we have $$x(q'c)_{u(q')} =x(q'c)_{u(q'c)}= x(q')_{u(q')}+1$$ and
$$x(q'c)_{w(q')}=x(q'c)_{w(q'c)} = x(q')_{w(q')}-1,$$ with the signs reversed when $c = a^{-1}$.
\end{enumerate}
In other words, to obtain $u(q'c), w(q'c)$ and  $\x(q'c)$ from $u(q'), w(q')$ and  $\x(q')$, we either alter
$u$ and $w$ by $1$, or we alter the digits $x_{u(q')}$ and $x_{w(q')}$ by $1$.

The vector $\x$ has the following form:
\[
\x = 
(\overset{\p}{\overbrace{\underset{0}{b},\dots, \underset{u=r+1}{b}, \dots, b,\underset{2r+2}{0}}}, \x',
\overset{\s}{\overbrace{\underset{\kx-(2r+5)}{0}, b, \dots, \underset{w=\kx-(r+4)}{b}, \dots, \underset{\kx}{b}}}),
\]
where we have indicated relevant indices below the vector.  As $\x(q)$ is obtained from $\x$ by
applying steps (1) or (2) above a total of $r$ times, once for each generator in $q$, we see that $\x(q)$ must have two properties:
\begin{itemize}[itemsep=5pt]
    \item $\x'$ is unchanged between $\x$ and $\x(q)$, 
    \item the final two digits of $\x$ and $\x(q)$ are identical, and
    \item there are upper and lower bounds on the digits in $\p(q)$ and $\s(q)$.
Since $b = \lfloor \frac{n}{4} \rfloor$ and $r \leq \max(1,\lfloor \frac{n}{4} \rfloor-1)$,
we have $1 \leq x(q)_i \le 2\lfloor\frac{n}{4}\rfloor -1$ for any digit $x(q)_i$ in $\p(q)$ or $\s(q)$.
\end{itemize}

If $n$ is odd, it follows from Lemma~\ref{lemma:odd_box} and inspection of the final two digits of
$\s(q)$ that $\x(q)$ is minimal.  
If $n \geq 4$ is even, then all digits in $\p(q)$ and $\s(q)$ satisfy $|x(q)_i| \le \fne -1$.

If  $|x(q)_i| < \fne -1$ for all digits in  $\s(q)$,
the minimality of $\x(q)$ follows from Lemma~\ref{lemma:prefix-suffix}.
If there is some $i \geq \kx - (2r+5)$ with
$|x(q)_i| = \fne -1$,  then we must have $q = a^{\pm r}$ and $i = w$.
In this case, since $\x(q)$ does not satisfy Lemma~\ref{lemma:even_reduction_end}, 
it follows from Lemma~\ref{lemma:even_minimal_characterization} that there is a run $\r$ in $\x(q)$
at which it can be reduced.  As $\r$ begins with a digit $x_j$ satisfying $|x_j| = \fne$, $\r \cap \p(q) = \emptyset$.  As the first digit of $\s(q)$ is $0$,and when $n \geq 4$ is even the digit $0$ cannot be part of a run, we have $\r \cap \s(q) = \emptyset$ as well, so 
$\r \subseteq \x'$.

If $\r$ does not contain the final digit of $\x'$, it follows from Lemma~\ref{lemma:run_in_the_middle} that $\x'$ can be reduced at $\r$.  Otherwise, as the first digit of $\s(q)$ is $0$, it follows from Lemma~\ref{lemma:next_digit_0} that $\x'$ can be reduced at $\r$.  In either case, this contradicts the fact that $\x'$ is minimal.  Thus $\x(q)$ is minimal for all $q$ as above with $|q| = r$.

It remains to compute $l(g^q)$.
As $k_{\x(q)} = \kx < \max(u,w)$,
we use the second length formula in Lemma~\ref{lemma:length_formula} to 
compute $l(g^q)$.  
We have
\begin{align*}
l(g^q) &= \Vert \x(q) \Vert_1 + 2k_{\x(q)} - |u(q) - w(q)| \\
       &= \Vert \x(q) \Vert_1 + 2\kx - |u - w|\\
       &= l(g) + \Vert \x(q) \Vert_1 - \Vert \x \Vert_1
\end{align*}
where the second line follows from the first by the properties in Remark~\ref{remark:conjugation}, namely that $\kx = k_{\x(q)}$ and $|u(q) - w(q)| = |u-w|$.
Inspecting the final equality above, we note that with each successive conjugation by a letter in $q$, we add and subtract $1$ in different coordinates in our vector.  
As $b = \lfloor \frac{n}{4} \rfloor$ and $r<b$, any digit which differs between $\x$ and $\x(q)$is positive.  
Thus the net change to the $\ell^1$ norm of the vector after each successive conjugation is $0$, so $\Vert \x(q) \Vert_1 = \Vert \x \Vert_1$ and thus $l(g^q) = l(g)$.
As $q$ was any string of length $r$, it follows that $\kappa_r(g) = 0$ when $n \ge 8$.

We now verify the result when $3 \leq n \leq 7$, in which case $b=r=1$.  
Let $g \in {\mathcal Z}_n$ be constructed as above; we first show that the corresponding vector $\x$ of consecutive exponents of the generator $a$ is minimal.
\begin{itemize}[itemsep=5pt]
\item If $n$ is odd, it follows from inspection of the final two digits $(b,b)$ of $\x$ and Lemma~\ref{lemma:odd_box} that $\x$ is minimal.
\item When $n=6$ it follows from Lemma~\ref{lemma:prefix-suffix} that $\x$ is minimal, as $|s_i| < \fn-1 = 2$ for all $s_i \in \s$.
\item When $n=4$, Lemma~\ref{lemma:prefix-suffix} does not apply to $\x$.  
Suppose that $n=4$ and $\x$ is not minimal.
Inspection of the final two digits of $\x$ shows that Lemma~\ref{lemma:even_reduction_end} does not apply, and hence it follows from Proposition~\ref{lemma:even_minimal_characterization} that $\x$ contains a run $\r$ at which it can be reduced.   
However, when $n=4$ any run must begin with $\pm 2$, and hence $\r \subseteq \x'$.
It then follows from Lemma~\ref{lemma:run_in_the_middle} or~\ref{lemma:next_digit_0} that $\x'$ is not minimal, a contradiction.  Thus $\x$ is minimal.
\end{itemize}
As $\x$ is minimal, every word in $\mathcal{Z}_n$ for $3 \leq n \leq 7$ is geodesic.
It then follows from from Lemma~\ref{lemma:prefix_suffix_growth_rate} that $\mathcal{Z}_n$ has positive density in $BS(1,n)$ for $3 \leq n \leq 7$.  
It remains to show that any $g$ in one of these sets has $\kappa_1(g) = 0$.

For all $3 \leq n \leq 7$ it follows immediately from Lemma~\ref{lemma:type1_facts} that $l(g^t) = l(g^{\ti}) = l(g)$.  

Consider the prefix and suffix vectors for $g^a$ and $g^{a^{-1}}$.  That is, $g^a$ has
\begin{itemize}
    \item prefix vector $\p(a) = (1,1,2,1,0)$ and $u(a)=2$, and
    \item suffix vector $\s(a) = (0,1,0,1,1,1,1,1)$ where the second $0$ has index $w(a) = k-5$.
\end{itemize}
In the corresponding vectors for $g^{\ai}$ we have $x(\ai)_{u(\ai)} = 0$ and $x(\ai)_{w(\ai)}=2$.
As the argument is completely analogous, we consider only the case of $g^a$.

If $n \in \{5,7\}$, it follows directly from Lemma~\ref{lemma:odd_box} that $\x(a)$ is minimal.  If $n=6$ the same conclusion follows from Lemma~\ref{lemma:prefix-suffix}.  
If $n=4$, inspection of the final two digits of $\s(a)$ shows that Lemma~\ref{lemma:even_reduction_end} does not apply, and if $\x(a)$ was not minimal, it would follow from Proposition~\ref{lemma:even_minimal_characterization} that $\x(a)$ contained a run $\r$ at which it could be reduced. 
It is easily checked that if $\r \subset \p(a)$ then $\r = (2)$ or $\r = (2,1)$ and $\x(a)$ cannot be reduced at $\r$.
We conclude that $\r \subseteq \x'$, and then it follows from Lemma~\ref{lemma:run_in_the_middle}  or~\ref{lemma:next_digit_0} that $\x'$ is not minimal, a contradiction.  Thus $\x(a)$ is minimal when $n=4$.

When $n=3$ we must first replace $\x(a)$ with $\x(a)+{\bf w}^{(2)}+{\bf w}^{(3)}$.  Note that both vectors have the same $\ell^1$ norm, but $\x(a)+{\bf w}^{(2)}+{\bf w}^{(3)} \in {\mathcal B}_{v(a)}^{u(a),w(a)}$.  Keeping the same notation for this vector, it then follows from Lemma~\ref{lemma:odd_box} that $\x(a)$ is minimal.

For all $3 \leq n \leq 7$, as $k_{\x(a)} = \kx$ and $\max(u(a),w(a)) < k_{\x(a)}$, the same length formula from Lemma~\ref{lemma:length_formula} is used to compute both $l(g)$ and $l(g^a)$. 
It follows that  $l(g^a) = l(g) +1-1 = l(g)$.  The analogous argument applies to $l(g^{\ai})$ and we conclude that for $3 \leq n \leq 7$ every $g \in {\mathcal Z}_n$ satisfies $\kappa_1(g) = 0$.
\end{proof}

\subsection{A positive density set of elements with negative conjugation curvature}
\label{section:negative}

Our most robust result for conjugation curvature is Theorem~\ref{thm:neg_kappa_r}, which finds a positive density of elements $g \in BS(1,n)$ with $\kappa_r(g)<0$ for the widest range of $r$.

\begin{theorem}\label{thm:neg_kappa_r}
For $n \geq 3$, the group $BS(1,n)$ has a positive density of elements with $\kappa_r(g)<0$, for any  $1 \leq r \leq  \max(1,\lfloor \frac{n}{2} \rfloor-2)$.
\end{theorem}

\begin{proof}
Choose $r$ so that $1 \leq r \leq \max(1,\lfloor \frac{n}{2} \rfloor-2)$.  
Construct a set of words ${\mathcal N}_n$ so that $g \in {\mathcal N}_n$ is formed by $p \xi s$ where $p$ and $s$ are words specified below, and $\xi \in {\mathcal Q}_n$.  Namely,
\begin{itemize}[itemsep=5pt]
    \item $p=t^{-(r+1)}at^{2r+3}$, with corresponding vector $\p = (1,0,0,\cdots ,0)$ of consecutive exponents of the generator $a$, of length $2r+3$.
    \item $s=t^{2r+5}at^{-(r+4)}$, with corresponding vector $\s = (0,0,\cdots ,0,1)$ of consecutive exponents of the generator $a$, of length $2r+6$.
\end{itemize}

Let $\x'$ be the vector of consecutive exponents of the generator $a$ in $\xi$.  According to the definition of ${\mathcal Q}_n$, we know that $\x' \in {\mathcal B}_{v'}^{0,\K{\x'}+1}$ where $v' = \Sigma(\x')$, and $\x'$ is minimal.
Inspection of the initial and final letters of $p,\xi$ and $s$ allows us to write $\x = \p \x' \s$ as the vector of consecutive exponents of the generator $a$ in $g$.    By construction, $\x \in {\mathcal B}_v^{r+1,\kx-(r+4)}$.

When $n$ is odd, as the final two digits of $\x$ are $(0,1)$, it follows from Lemma~\ref{lemma:odd_box} that $\x$ is minimal.
When $n$ is even, it follows immediately from Lemma~\ref{lemma:prefix-suffix} that $\x$ is minimal.
The minimality of $\x$ ensures that 
every word in $\mathcal{N}_n$ is geodesic,
and it follows from Lemma~\ref{lemma:prefix_suffix_growth_rate} that $\mathcal{N}_n$
has positive density in $BS(1,n)$.
It remains to show that every $g \in \mathcal{N}_n$ satisfies $\kappa_r(g) < 0$.

Let $q = q_1\dots q_r$.  In order to show that $\kappa_r(g) < 0$, we will first
show that $\x(q)$ is minimal for all such $q$, where $\x(q)$ is defined in Remark~\ref{remark:conjugation}.  Since $\x$ is also minimal, we can then
use the appropriate length formula to compute first $l(g^q)$ and then $\kappa_r(g)$.

The proof now follows the outline of the proof of Theorem~\ref{prop:zero_kappa_r}; recall the rules specified there for obtaining $u(q),w(q)$ and $\x(q)$ from $u,w$ and $\x$, which follow from Remark~\ref{remark:conjugation}.
As  the initial prefix and suffix vectors $\p$ and $\s$ are different than those in Theorem~\ref{prop:zero_kappa_r}, we draw the following conclusions in this case from the application of those rules.  In particular, any digit $x_i$ which differs in $\x$ and $\x(q)$
\begin{itemize}[itemsep=5pt]
    \item has $1 \leq i \leq 2r+1$ or $\kx - (2r+4) \leq i \leq \kx -4$.  It follows that $k_{\x(q)} = \kx$, and that $\x$ and $\x(q)$ share the same final two digits.
     \item is $0$ in $\x$ and, as $1 \leq r \leq\max(1,\fn-2)$, satisfies $|x(q)_j| < \fne -1$.  
\end{itemize}
It follows from these two facts that $\x(q)$ satisfies the conditions of Lemma~\ref{lemma:prefix-suffix} and hence is minimal.

As $k_{\x(q)} = \kx$ and the upper bound on $r$ ensures that $u(q)<w(q) < k_{\x(q)}$, to compute both $l(g)$ and $l(g^q)$ we use the second formula in Lemma~\ref{lemma:length_formula}.  
It follows from Remark~\ref{remark:conjugation} that $|u-w| = |u(q)-w(q)|$.
Thus the difference between $l(g^q)$ and $l(g)$ is exactly the difference in the respective $\ell^1$ norms of $\x(q)$ and $\x$.
Moreover, any digit which differs between $\x$ and $\x(q)$ is $0$ in $\x$ and nonzero in $\x(q)$.  Thus,
\[
\Vert \x(q) \Vert_1 - \Vert \x \Vert_1 = \sum_{i \in I} |x(q)_i| \geq 0,
\]
so $l(g^q) \geq l(g)$.
When $q=a^{\pm r}$, the same analysis shows that $l(g^{a^r}) = l(g^{a^{-r}}) = l(g) +2r$ and hence the above inequality is sometimes strict.

Thus we conclude that
\[
\sum_{p \in S_n(r)}l(g^p) > |S_n(r)|l(g),
\]
and it follows that all $g \in {\mathcal N}_n$ have $\kappa_r(g)<0$ for $r \leq \max(1,\fn -2)$.
\end{proof}

\section{Conjugation curvature in $BS(1,2)$}
\label{sec:n=2_curvature}

When $n=2$ we again find that $BS(1,2)$ contains a positive density of elements $g$ with  positive, negative and zero conjugation curvature $\kappa_1(g)$.

When $n=2$ and $\x \in \Bvuw$, we note that a run is a string in $\{0,1\}^*$ or $\{0,-1\}^*$
which begins with $\pm 1$.  
We redefine the weight of a run $\r$ when $n=2$ to account for the slight differences between the cases $n=2$ and $n>2$.
Given $\x \in \Bvuw$, define the weight of a run $\r = (x_j, \dots, x_\ell)$  to be
\[
\wt(\r) = \left(\#\left\{1\right\}-1\right)-\#\left\{ 0\right\}
\]
where $\#\left\{1\right\}$ denotes the number of occurrences of the digit $\epsilon_\r=\sgn(x_j)$ in $\x$ and $\#\left\{0\right\}$ is defined analogously.
Note that unlike the case of the weight of a run for $n>2$, digits
with absolute value greater than $1$ are not considered
when computing weight.

If ${\bf s} = (s_i, \cdots ,s_{\ell}) \subseteq \x$ is a string of digits from either $\{0,1\}^*$ or $\{0,-1\}^*$, define $\beta({\bf s})$ to be the difference between the number of zeros in ${\bf s}$ and the number of ones in ${\bf s}$.

Lemma~\ref{lemma:n=2runs} demonstrates that if $\x$ can be reduced at a run $\r$ of the form
$(\lambda,0, \cdots ,0,\lambda, r_1, \cdots ,r_{\ell})$, where $\lambda \in \{ \pm 1\}$, then there is a shorter run at which $\x$ can be reduced.

\begin{lemma}
\label{lemma:n=2runs}
Let $n=2$ and $\x \in \Bvuw$ with 
\[
\r = (x_j, \cdots ,x_{\ell}) = (\lambda,0, \cdots ,0,\lambda, r_1, \cdots ,r_q) \subseteq \x
\]
a run where $x_j = x_m = \lambda$ and $\ell \leq \kx-2$, for $\lambda \in \{\pm 1\}$.
Then
\[
\x + \lambda \sum_{i=m}^{\ell} \wi \; \ords \; \x+ \lambda \sum_{i=j}^{\ell} \wi.
\]
\end{lemma}

In other words, if $\x$ can be reduced at $\r$, and $\ell \le \kx-2$ 
then Lemma~\ref{lemma:less_than_6} applies, so the reduction takes the form on the right.
It follows from Lemma~\ref{lemma:n=2runs} that with respect to $\ords$ it is better to reduce $\x$ at the shorter run $(\lambda,r_1,\dots,r_q)$.

\begin{proof}
With $\r = (x_j, \cdots ,x_{\ell}) = (\lambda,0, \cdots ,0,\lambda, r_1, \cdots ,r_q) \subseteq \x$ as in the statement of the lemma, so $x_j = x_m = \lambda$, we have
\[
\Vert \x+\lambda \sum_{i=j}^{\ell} \wi \Vert_1 = \Vert \x \Vert_1 + (m-1-j)-1+\beta(r_1, \cdots ,r_q) + |x_{\ell+1}+\lambda| - |x_{\ell+1}|
\]
where $m-1-j$ is the length of the first string of zeros, and we subtract $1$ to account for the change in $x_m = \lambda$.  
By construction, $m-j \geq 2$.  
In contrast,
\[
\Vert \x+\lambda\sum_{i=m}^{\ell} \wi \Vert_1 = \Vert \x \Vert_1 + \beta(r_1, \cdots ,r_q) + |x_{\ell+1}+\lambda| - |x_{\ell+1}|.
\]
As $\r$ does not contain the final two digits of $\s$, the vectors $ \x + \lambda\sum_{i=j}^{\ell} \wi$ and $ \x+\lambda\sum_{i=m}^{\ell} \wi$ have the same length.
It follows that the same word length formula from Lemma~\ref{lemma:length_formula} is used to compute both $|\eta_{u,v,w}(\x+\lambda\sum_{i=j}^{\ell} \wi)|$ and $|\eta_{u,v,w}(\x+\lambda\sum_{i=m}^{\ell} \wi)|$.
Thus the difference in word length between the two paths is exactly the difference in $\ell^1$ norm between the vectors.
If $m-j>2$, the lemma follows from comparing the expressions for the $\ell^1$ norms of the two vectors.
If $m-j=2$, then the two vectors have the same $\ell^1$ norm.  However, as the second digit of $\r$ is $0$, we see that $\x + \lambda\sum_{i=m}^{\ell} \wi$ precedes $ \x+\lambda\sum_{i=j}^{\ell} \wi$ in the lexicographic order, and the lemma follows.
\end{proof}

When constructing examples of elements of different curvatures when $n=2$, we again create families of geodesics as the concatenation $p \xi s$ where $p$ is a prefix and $s$ is a suffix, each of a specified form, and $\xi \in {\mathcal Q}_2$.
In an effort to simplify the discussion, we will phrase everything in terms of the vectors of exponents of the generator $a$ in the prefix, suffix, and $\xi$.
Let $Q_2$ be the
set of vectors satisfying either one of the following equivalent conditions
\[
Q_2 = \{\x \, | \, \x \textnormal{ is minimal in $\BB_{\Sigma(\x)}^{0,\kx+1}$} \}
    = \{\x \, | \, \eta_{0,\Sigma(\x),\kx+1}(\x) \in \QQ_2\}.
\]
We will be interested in filtering $Q_2$ by the lengths of the associated geodesics in $\QQ_2$.
Thus we define $Q_2(N) = \{\x \in Q_2 \, | \, |\eta_{0,\Sigma(\x),\kx+1}(\x)| = N\}$.

\begin{lemma}\label{lemma:n=2_Q2_growth_rate}
The growth rate of $\{|Q_2(N)|\}_{N \in \N}$ is the same as the growth rate of $BS(1,n)$.
\end{lemma}
\begin{proof}
The lemma follows immediately from the definition of $Q_2$ and Lemma~\ref{lemma:middle_vector}.
\end{proof}

Lemma~\ref{lemma:n=2_prefix_suffix_growth_rate} rephrases Lemma~\ref{lemma:prefix_suffix_growth_rate} using vectors instead of geodesics.

\begin{lemma}\label{lemma:n=2_prefix_suffix_growth_rate}
Let $\mathcal{A}$ be a set of triples $(u,w,\x)$, where $\x$ is minimal in $\BB_{\Sigma(\x)}^{u,w}$.  Let
\[\mathcal{A}(N) = \{(u,w,\x) \in \mathcal{A} \, |\, |\eta_{u,\Sigma(\x),w}(\x)| = N\}.\]  Suppose that $|\mathcal{A}(N)| = |Q_2(N+c)|$ for some constant $c \in \Z$.
Then the set of geodesics 
\[ \{\eta_{u,\Sigma(\x),w}(\x) \mid (u,w,\x) \in \mathcal{A}\} \] 
has positive density in $BS(1,n)$.
\end{lemma}
\begin{proof}
The lemma follows immediately from the definition of $\mathcal{A}$ and  Lemmas~\ref{lemma:exp_growth} and ~\ref{lemma:n=2_Q2_growth_rate}.
\end{proof}

To show that $BS(1,2)$ has a positive density of elements with positive, negative and zero conjugation curvature, 
our strategy is to construct a variety of sets $\mathcal{A}$
as in Lemma~\ref{lemma:n=2_prefix_suffix_growth_rate}, where the vector
$\x$ is the concatenation $\p\x'\s$ of a prefix, suffix, and vector
$\x' \in Q_2$.    We will always have $u=2$ and $w=\kx-\epsilon$
for a particular value of $\epsilon \in \{1,2\}$ chosen later.

Our prefixes have the following forms:
\begin{itemize}[itemsep=5pt]
    \item $\p_1(b) = (1,0,b,0,-1,0,1,0)$, for $b \in \{0,\pm 1\}$.
    \item $\p_2(b_1,b_2) = (1,0,b_1,b_2,0,-1,0)$ where $(b_1,b_2) \in \{(1,0),(0,1),(0,0)\}$.
\end{itemize}
Our suffixes have the following form 
\begin{itemize}
    \item $\s(c_1,c_2,c_3) = (0,0,-1,0,c_1,c_2,c_3)$, where 
    \[(c_1,c_2,c_3) \in \{(0,1,2),(0,0,2),(0,0,3),(-1,0,3),(1,0,2)\}.\]
\end{itemize}

In order to simplify subsequent proofs, we will combine a number
of prefix-suffix pairs into one large set $\VV_2$ of vectors, prove
that any vector in $\VV_2$ is minimal, and choose subsets of
$\VV_2$ whose elements have, respectively, positive, zero and negative conjugation curvature.

Define $\VV_2$ to be the union of all triples of the following forms; in all cases, $\x'$ is an arbitrary vector in $Q_2$.
\begin{itemize}
    \item $(2,w,\x)$, where $w = \kx-1$ and $\x = \p_2(b_1,b_2)\x'\s(c_1,c_2,c_3)$, for
    \[
    ((b_1,b_2),(c_1,c_2,c_3)) \in \{((1,0),(0,1,2)),\,((0,1),(0,0,2)),\,((0,0),(0,0,3))\}
    \]
    \item $(2,w,\x)$, where $w = \kx-2$ and $\x = \p_2(b_1,b_2)\x'\s(c_1,c_2,c_3)$, for
    \[
    ((b_1,b_2),(c_1,c_2,c_3)) \in \{((1,0),(0,1,2)),\,((0,1),(1,0,2)),\,((0,0),(-1,0,3))\}
    \]
    \item $(2,w,\x)$, where $w = \kx-2$ and $\x = \p_1(b)\x'\s(c_1,c_2,c_3)$, for
    \[
    (b,(c_1,c_2,c_3)) \in \{(0,(0,1,2)),\,(1,(1,0,2)),\,(-1,(-1,0,3))\}
    \]
\end{itemize}

We begin with a technical lemma required to prove 
that any vector in $\VV_2$ is minimal.
It concludes that if $\x$ is constructed as in $\VV_2$, and $\x$ can be reduced at a run $\r$, 
then there is also a run $\r'$ at which $\x$ can be reduced which does not overlap the suffix vector.
Moreover, with respect to the $\ords$ order, it is better to reduce $\x$ at $\r'$ than at $\r$.

\begin{lemma}
\label{lemma:big_suffix}
Let $(2,w,\x) \in \VV_2$, with $\x = \p \x' \s \in \BB_{\Sigma(\x)}^{u,w}$.
Let $\r = (x_j, \cdots ,x_{\ell}) \subseteq \x$ be a run at which
$\x$ can be reduced with $l \leq \kx-1$.
Suppose that $\r$ is the concatenation $\r' \r_{\s}$ where
$\r' = \r \cap \p\x' = (x_j, \cdots ,x_m) \neq \emptyset$ and
$\r_{\s} = \r \cap \s =(x_{m+1}, \cdots ,x_{\ell})\neq \emptyset$.  
Then 
\[
\x+ \epsilon_{\r} \sum_{i=j}^{m}\wi  \; \ords \; \x + \epsilon_{\r} \sum_{i=j}^{\ell}\wi  \; \ords \; \x.
\]
\end{lemma}

\begin{proof}
Considering the combinations of prefixes and suffixes for elements in ${\mathcal V}_2$, note that if $|\r_{\s}|>2$ then $\r$ contains the digit $-1$, and thus $\epsilon_\r = -1$.  
In particular, $\r$ cannot contain the final digit of $\s$, which in all cases is positive.

As $\r$ cannot contain the final digit of $\s$, it follows from  Lemma~\ref{lemma:less_than_6} that the result $\y$
of reducing $\x$ at $\r$ is 
\[
\y = \x + \epsilon_\r\sum_{i=j}^\ell \wi.
\]
That is, the coefficient $2\epsilon_\r$ does not appear in the above sum.
When $l \leq \kx -2$, it is clear that $\ky = \kx$.
When $l=\kx-1$, that is, $\r$ terminates at the penultimate digit of $\x$, the final digit of $\y$ is $x_{\kx} -1$ which, by inspection, in all cases is either $1$ or $2$.  
So it is again the case that $\ky = \kx$.

Let $\y' = \x + \epsilon_\r\sum_{i=j}^m \wi$.  
As the length of $\s$ is $7$, it is immediate that $\r'$ does not contain the final two digits of $\x$ and hence $k_{\y'} = \ky = \kx$.

The elements of ${\mathcal V}_2$ are constructed so that 
$\kx > \max(u,w)$; as $k_{\y'} = \ky = \kx$  we use the second word length formula in Lemma~\ref{lemma:length_formula} to compute  $|\eta_{u,\Sigma(\x),w}(\x)|$, $|\eta_{u,\Sigma(\y),w}(\y)|$ and $|\eta_{u,\Sigma(\y'),w}(\y')|$.
Thus any change in word length arises from a change in $\ell^1$ norm between the vectors.
Consider
\begin{align*}
\Vert \x \Vert_1 - \Vert \y \Vert_1 &= \wt(\r) + |x_{\ell+1}| - |x_{\ell+1} + \epsilon_\r|   \\
\Vert \x \Vert_1 - \Vert \y' \Vert_1 &= \wt(\r') + |x_{m+1}| - |x_{m+1} + \epsilon_\r|.
\end{align*}
As $\y \ords \x$ we know that the first difference is nonnegative.  To prove the lemma, we must show that 
\[
\left( \Vert \x \Vert_1 - \Vert \y' \Vert_1 \right) - \left( \Vert \x \Vert_1 - \Vert \y \Vert_1 \right) = \Vert \y \Vert_1 - \Vert \y' \Vert_1 \geq 0,
\]
and when there is equality, that there is a lexicographic reduction from $\y$ to $\y'$.
For a string of digits $(x_s, \cdots ,x_t)$ in either $\{0,1\}^*$ or $\{0,-1,\}^*$, let $\beta(x_s, \cdots ,x_t)$ denote the difference between the number of nonzero entries and the number of zero entries.
With this notation, we can write $\wt(\r) = \wt(\r') + \beta(\r_{\s})$.
It follows that
\begin{align*}
\Vert \y \Vert_1 - \Vert \y' \Vert_1  &= -\beta(\r_{\s}) + |x_{m+1}| - |x_{m+1} + \epsilon_\r| - |x_{\ell+1}| + |x_{\ell+1}+ \epsilon_\r| \\
&=  -\beta(\r_{\s})-1 - |x_{\ell+1}| + |x_{\ell+1}+ \epsilon_\r|,
\end{align*}
recalling that $x_{m+1} = 0$.

We know that $\r_{\s}$ is a prefix of $\s$.  When $|\r_{\s}| >2$ we must have $\epsilon_\r = -1$, for shorter words it is possible that $\epsilon_\r = 1$. 
As $|\s| =7$ it is straightforward to list the possible prefixes of $\s$ which could constitute a run and compute the quantity above.  
We do not consider any prefix of $\s$ which contains a $-1$ and a $+1$, as these digits cannot both occur in a run.
These values are compiled in Table~\ref{table:suffix}, and we see in all cases that they are non-negative.

\begin{table}[ht!]
\begin{center}
\begin{tabular}{|l|c|c|c|c|}
\hline
$\r_{\s}$ &  $x_{\ell+1}$ & $\epsilon_\r$ & $\beta(\r_{\s})$ & $-\beta(\r_{\s})-1 - |x_{\ell+1}| + |x_{\ell+1}+ \epsilon_\r|$\\
\hline
$(0)$ & 0 & 1 & -1 & 1 \\
$(0,0)$ & -1 & 1 & -2 & 0 \\
\hline
$(0)$ & 0 & -1 & -1 & 1 \\
$(0,0)$ & -1 & -1 & -2 & 2 \\
$(0,0,-1)$ & 0 & -1 & -1 & 1 \\
$(0,0,-1,0)$ & 0 & -1 & -2 & 2 \\
$(0,0,-1,0)$ & -1 & -1 & -2 & 2 \\
$(0,0,-1,0)$ & 1 & -1 & -2 & 0 \\
$(0,0,-1,0,0)$ & 1 & -1 & -3 & 1 \\
$(0,0,-1,0,0)$ & 0 & -1 & -3 & 3 \\
$(0,0,-1,0,-1)$ & 0 & -1 & -1 & 1 \\
$(0,0,-1,0,0,0)$ & 2 & -1 & -4 & 2 \\
$(0,0,-1,0,0,0)$ & 3 & -1 & -4 & 2 \\
$(0,0,-1,0,-1,0)$ & 3 & -1 & -2 & 0 \\
\hline
\end{tabular}
\label{table:suffix}
\caption{For each possible $\r_{\s} \subset \s$ and  value of $\epsilon_\r$, we compute the quantity  $\Vert \y \Vert_1 - \Vert \y' \Vert_1 = -\beta(\r_{\s})-1 - |x_{\ell+1}| + |x_{\ell+1}+ \epsilon_\r|$.}
\end{center}
\end{table}

When Table~\ref{table:suffix} shows that $-\beta(\r_{\s})-1 - |x_{\ell+1}| + |x_{\ell+1}+ \epsilon_\r| > 0$, the lemma follows immediately.  
When $-\beta(\r_{\s})-1 - |x_{\ell+1}| + |x_{\ell+1}+ \epsilon_\r| = 0$,
we show that there must be a lexicographic reduction from $\y$ to $\y'$.  
Consider the digit in each vector with index $m+2$.
Since $\s$ begins with $(0,0)$, we see that $|y'_{m+1}| = 1$ and $|y'_{m+2}| = 0$.
Comparing to $|y_{m+1}| = 1$ and $|y_{m+2}| = 1$ we note the lexicographic reduction and conclude that $\y' \ords \y$ in this case as well.
\end{proof}

We now show that every element of $\VV_2$ is minimal.

\begin{lemma}
\label{lemma:n=2_prefix_suffix}
For all $(2,w,\x) \in \VV_2$, we have that $\x \in \BB_{\Sigma(\x)}^{2,w}$ is minimal.
\end{lemma}

\begin{proof}
Write $\x = \p\x'\s$ as in the definition of $\VV_2$.
By construction, $\x' \in {\mathcal B}_{\Sigma(\x)}^{0,\K{\x'}+1}$ is minimal.
Suppose that $\x$ is not minimal.  
Inspection of the final two digits of $\s$ shows that Lemma~\ref{lemma:n=2_notminimal} does not apply to $\x$, and thus
it follows from Proposition~\ref{lemma:lemma318_n=2} that there is a run $\r = (x_j, \cdots ,x_{\ell}) \subseteq \x$ at which $\x$ can be reduced.

As the sign of all the elements in a run is either identical or 0, notice from 
the form of the possible suffixes that 
either $\r$ does not contain the final digit of the suffix, or $\r = (1,2)$ or $(1,0,2)$. 
In the latter cases, it is easily checked that $\x$ can not be reduced at $\r$.
In the former case, we may apply Lemma~\ref{lemma:big_suffix}.  Thus in all cases we may assume that $\r \cap \s = \emptyset$. 
It then follows from Lemma~\ref{lemma:less_than_6} that we can write
\[
\y = \x + \epsilon_\r\sum_{i=j}^\ell \wi
\]
to denote the result of reducing $\x$ at the run $\r$.  This is, no coefficient in the linear combination of basis vectors above is $2\epsilon_\r$.

We now show that $\r \cap \p = \emptyset$.  
It is easily checked that if $\r \subseteq \p$ then $\x$ cannot be reduced at $\r$.  
Suppose that $\r \cap \p$ is nonempty.  For prefix $\p_i$, set $\lambda = (-1)^{i+1}$. Then $\r \cap \p$
must begin $(\lambda,0)$.
If there is a subsequent digit of $\lambda$ in $\r \cap \x'$, consider the first occurrence of $\lambda$ in $\r \cap \x'$.
It follows from Lemma~\ref{lemma:n=2runs} that it is at least as effective to reduce $\x$ at the run $\r'\subset \x'$, where $\r' \subseteq \r$ is the run beginning with the first occurrence of $\lambda$ in $\r \cap \x'$.

If there is no digit $\lambda$ in $\r \cap \x'$, then $\r \cap \x'$
has no nonzero digits.
That is, $\r \cap \p = (\lambda,0)$  and $\r \cap \x' = (0,0, \cdots,0)$ consists of $d$ zeros, for some $d \geq 1$.  
Then we have
\[
\Vert \x + \epsilon_\r \sum_{i=j}^{\ell} \wi \Vert_1 = \Vert \x \Vert_1 + (d+1) + |x_{\ell+1}+\epsilon_\r| - |x_{\ell+1}|  > \Vert \x \Vert_1
\]
where the inequality follows from the facts that  $|x_{\ell+1}+\epsilon_\r| - |x_{\ell+1}| \in \{\pm 1\}$ and $d \geq 1$.
We conclude that $\x$ cannot be reduced at $\r$, a contradiction.  Thus we may assume that $\r \subseteq \x'$.

If $\r$ does not contain the final digit of $\x'$, it follows from Lemma~\ref{lemma:run_in_the_middle} that $\x'$ can be reduced at $\r$, a contradiction.  
If $\r$ does contain the final digit of $\x'$, as the first digit of each suffix is $0$, it follows from Lemma~\ref{lemma:next_digit_0} that $\x'$ can be reduced at $\r$, a contradiction.  Thus we conclude that $\x\in {\mathcal B}_{\Sigma(\x)}^{2,w}$ is minimal.
\end{proof}

We now prove that $BS(1,2)$ contains a positive density of elements with positive, negative and zero conjugation curvature.

\begin{theorem}
\label{thm:n=2_examples}
$BS(1,2)$ contains a positive density of elements $g$ with positive, negative and zero conjugation curvature $\kappa_1(g)$.
\end{theorem}

\begin{proof}
Define three sets of triples:
\begin{enumerate}
    \item $\mathcal{P}_2 = \{(2,w,\x) \, | \, \x = \p_2(1,0)\x'\s(0,1,2), \, \x' \in Q_2, \, w = \kx-1\}$.
    \item $\mathcal{Z}_2 = \{(2,w,\x) \, | \, \x = \p_2(1,0)\x'\s(0,1,2), \, \x' \in Q_2, \, w = \kx-2\}$.
    \item $\mathcal{N}_2 = \{(2,w,\x) \, | \, \x = \p_1(0)\x'\s(0,1,2), \, \x' \in Q_2, \, w = \kx-2\}$.
\end{enumerate}
Each of these is a subset of $\VV_2$, so it follows from Lemma~\ref{lemma:n=2_prefix_suffix}
in each case that $\x \in B_{\Sigma(\x)}^{u,w}$ is minimal.  
In addition, the prefix, suffix, and size of $w$ relative to the vector length
are all fixed within each set.  
It follows that $|\mathcal{A}(N)| = |Q_2(N+c)|$
for $\mathcal{A} = \mathcal{P}_2,\mathcal{Z}_2,\mathcal{N}_2$.  The value of $c$ varies between the subsets, but is constant for each one.
There is a map which realizes the bijection defined by taking
a triple $(2,w,\x)$ to the subvector $\x' \subseteq \x$. 
Hence by Lemma~\ref{lemma:n=2_prefix_suffix_growth_rate}, 
the set of geodesics
$\{\eta_{2,\Sigma(\x),w}(\x) \mid (2,w,\x) \in \mathcal{A}\}$
has positive density in $BS(1,n)$
for $\mathcal{A} = \mathcal{P}_2,\mathcal{Z}_2,\mathcal{N}_2$

It remains to shows that $\kappa(\eta_{u,\Sigma(\x),w}(\x))$
is positive, zero, and negative, for $(2,w,\x) \in \mathcal{P}_2,\mathcal{Z}_2$, and $\mathcal{N}_2$, respectively.
For an arbitrary such triple, let $g = \eta_{2,\Sigma(\x),w}(\x)$.
We must compare $l(g)$ to 
\[
\frac{1}{4}\left[l(g^t)+l(g^{\ti})+l(g^a)+l(g^{\ai})\right].
\]

As the first digit of $\p$ in all cases is nonzero, we know that $n$ does not divide $\Sigma(\x)$. 
As $\max(u,w) \leq \kx-1$, it follows from the first statement of Lemma~\ref{lemma:type1_facts} that $l(g) = l(g^t)=l(g^{\ti})$.

Following Remark~\ref{remark:conjugation}, $\x(a) = \rho_{u,-w}(\x)$, the vector 
in which we replace the digits $x_u$ and $x_w$, respectively, with $x_u+1$ and $x_w-1$.  Note that these digits always lie in $\p$ and $\s$, respectively.  For the remainder of this proof, we will write $u$, even though we always have $u=2$, for consistency.
Similarly, $\x(\ai) = \rho_{-u,w}(\x)$, the vector 
in which we replace the digits $x_u$ and $x_w$, respectively, with $x_u-1$ and $x_w+1$.

We consider the three possible cases: $(2,w,\x)$ in ${\mathcal P}_2$, ${\mathcal Z}_2$, and ${\mathcal N}_2$.  
Write $\rho_{u,-w}(\x) = \p_+ \x' \s_-$ and $\rho_{-u,w}(\x) = \p_- \x' \s_+$.  
Write $v = \Sigma(\x)$, and let $v_+ = n^u+v-n^w$ and $v_- = -n^u+v+n^w$.
It is not always true that 
$\rho_{u,-w}(\x)$ and $\rho_{-u,w}(\x)$ are minimal, or even elements of  ${\mathcal B}_{v_{\pm}}^{u,w}$.
However, in all cases, we add a small linear combination of basis
vectors $\wi$ in order to modify these vectors so that the triple $(u,w,\rho_{-u,w}(\x))$ is in $\VV_2$.
In order to avoid a plethora of notation, we will retain the same notation for the modified vectors.

Let $(2,w,\x) \in {\mathcal P}_2$.
\begin{enumerate}[itemsep=5pt]
    \item We have $\p_+ =(1,0,2,0,0,-1,0)$ and $\s_- = (0,0,-1,0,0,0,2)$, but $(u,w,\p_+ \x' \s_-)$ is not in $\VV_2$.
    Adding the basis vector ${\bf w}^{(2)}$ modifies the prefix to $\p_+ =(1,0,0,1,0,-1,0)$, and now  $(2,w,\p_+ \x' \s_-) \in {\mathcal V}_2$.
    \item We have $\p_- =(1,0,0,0,0,-1,0)$ and $\s_+ =(0,0,-1,0,0,2,2)$, but $(u,w,\p_-\x' \s_+)$ is not in $\VV_2$. 
    Adding the basis vector ${\bf w}^{(\kx-1)}$ modifies the suffix to $\s_+ =(0,0,-1,0,0,0,3)$, and now 
    $(2,w,\p_- \x' \s_+) \in {\mathcal V}_2$.
\end{enumerate}

Let $(2,w,\x) \in {\mathcal Z}_2$.
\begin{enumerate}[itemsep=5pt]
    \item We have $\p_+ =(1,0,2,0,0,-1,0)$ and $\s_- = (0,0,-1,0,-1,1,2)$, but $(u,w,\p_+ \x' \s_-)$ is not in $\VV_2$. 
    Adding the linear combination ${\bf w}^{(2)}+{\bf w}^{(\kx-2)}$ modifies the prefix  to $\p_+ =(1,0,0,1,0,-1,0)$ and the suffix to $\s_- = (0,0,-1,0,1,0,2)$.  It is now the case that $(2,w,\p_+ \x' \s_-) \in \VV_2$.
    \item We have $\p_- =(1,0,0,0,0,-1,0)$ and $\s_+ =(0,0,-1,0,1,1,2)$, but $(2,w,\p_- \x' \s_+)$ is not in ${\mathcal V}_2$.
    Adding the sum ${\bf w}^{(\kx-2)} + {\bf w}^{(\kx-1)}$ modifies the suffix to $\s_+ =(0,0,-1,0,-1,0,3)$, and
    now $(2,w,\p_- \x' \s_+) \in {\mathcal V}_2$.
\end{enumerate}

Let $(2,w,\x) \in {\mathcal N}_2$.
\begin{enumerate}[itemsep=5pt]
    \item We have $\p_+ =(1,0,1,0,-1,0,1,0)$ and $\s_- = (0,0,-1,0,-1,1,2)$, but  $(u,w,\p_+ \x' \s_-)$ is not in $\VV2$.  Adding ${\bf w}^{(\kx-2)}$ modifies the suffix to $\s_-=(0,0,-1,0,1,0,2)$, and now $(2,w,\p_+ \x' \s_-) \in \VV_2$.
    \item We have $\p_- =(1,0-1,0,-1,0,1,0)$ and $\s_+ =(0,0,-1,0,1,1,2)$, but  $(u,w,\p_- \x' \s_+)$ is not in $\VV_2$.  Adding ${\bf w}^{(\kx-2)} + {\bf w}^{(\kx-1)}$ modifies the suffix to $\s_+ =(0,0,-1,0,-1,0,3)$, and now $(2,w,\p_- \x' \s_+) \in \VV_2$.
\end{enumerate}

Notice that in all three cases, after suitable modification we
have $(2,w,\p_+ \x' \s_-), (2,w,\p_- \x' \s_+) \in \VV_2$,
so these vectors correspond to geodesic paths representing $g^a$ and $g^\ai$.
As $\max(u,w)<\kx$ for $\x$ and for these two vectors, we use the same word
length formula to compute $l(g), \ l(g^a)$ and $l(g^{\ai})$, so any difference in word length arises from a difference in $\ell^1$ norm between the corresponding vectors  $\x, \ \rho_{u,-w}(\x)$ and   $\rho_{-u,w}(\x)$.

When $(2,w,\x) \in {\mathcal P}_2$, 
inspection shows that $\Vert \rho_{u,-w}(\x) \Vert_1 = \Vert \x \Vert_1-1$ and  $\Vert \rho_{-u,w}(\x) \Vert_1 = \Vert \x \Vert_1-1$.
It follows that $l(g^a) = l(g^{\ai}) = l(g)-1$.
Together we see that $l(g^t)+l(g^{\ti})+l(g^a)+l(g^{\ai}) = 4l(g)-2$ and thus $\kappa_1(g)>0$.

When $(2,w,\x) \in {\mathcal Z}_2$, inspection shows that $\Vert \rho_{u,-w}(\x) \Vert_1 = \Vert \x \Vert_1$ and  $\Vert \rho_{-u,w}(\x) \Vert_1 = \Vert \x \Vert_1$.
It follows that $l(g^a) = l(g^{\ai}) = l(g)$.
Together we see that $l(g^t)+l(g^{\ti})+l(g^a)+l(g^{\ai}) = 4l(g)$ and thus $\kappa_1(g)=0$.

When $(2,w,\x) \in {\mathcal N}_2$, 
inspection shows that $\Vert \rho_{u,-w}(\x) \Vert_1 = \Vert \x \Vert_1 + 1$ and  $\Vert \rho_{-u,w}(\x) \Vert_1 = \Vert \x \Vert_1+2$.
It follows that $l(g^a) = l(g)+1$ and $l(g^{\ai}) = l(g)+2$.
Together we see that $l(g^t)+l(g^{\ti})+l(g^a)+l(g^{\ai}) > 4l(g)$ and thus $\kappa_1(g)<0$.
\end{proof}

\section{Technical Lemmas}
\label{sec:technical_lemmas}
In this section we include the proofs of the major propositions stated in Sections~\ref{section:geodesics_n_even} and ~\ref{section:geodesics_n_even_2}.

The first lemma completes the remaining case of Proposition~\ref{lemma:even_minimal_characterization}.

\begin{lemma}
\label{lemma:lastcase_lemma3.20}
Let $n \geq 4$ be even and $\x, \ \y \in \Bvuw$  with $\x$ not minimal and $\y$ minimal, so
$
\y = \x + \sum_{i=j}^\ell \alpha_i \wi.
$
Let $x_j = \fne$.  Then either
\begin{itemize}
    \item there is a run in $\x$ at which $\x$ can be reduced, or
    \item Proposition~\ref{lemma:adjacent_digits} applies to $\x$.
\end{itemize}
\end{lemma}

\begin{proof}
As $x_j = \fne$ and $\y \in \Bvuw$, the digit constraints on $\Bvuw$ force $\alpha_j = 1$ and thus $y_j = -\fne$.
We proceed to consider the subsequent digits of $\xx$ and $\y$.
Note that $y_{j+1} = x_{j+1} + 1 - \alpha_{j+1}n$.
For now, suppose that $|x_{j+1}|,|y_{j+1}| \le \fne$, so
\[
\left| 1 - \alpha_{j+1} n\right| \le n,
\]
and hence $\alpha_{j+1} \in \{0, 1\}$.  Suppose that
$\alpha_{j+1} = 1$; then we can make the same argument to conclude
that $x_{j+1} \in \{\fne,\fne-1\}$ and $\alpha_{j+2} \in \{0,1\}$.  We can continue in this way until
one of the following occurs, for some minimal index $l \geq j$:
\begin{enumerate}[itemsep=5pt]
\item[(a)] $\alpha_{\ell+1} = 0$, or
\item[(b)] $|x_{\ell+1}|,|y_{\ell+1}| > \fne$.
\end{enumerate}

First suppose that $\alpha_{\ell+1}=0$, which implies that $\alpha_i = 1$ for $j \leq i \leq l$ and thus
\[
\z = \sum_{i=j}^\ell \wi +  \sum_{i>\ell+1}\alpha_i \wi,
\]
so we have
\[
(z_j, z_{j+1}, \dots, z_{\ell}, z_{\ell+1}) = (-n, -(n-1), \dots, -(n-1), 1).
\]

The digit bounds on $\Bvuw$ force $x_i \in \{\fne, \fne-1\}$ for $j\le i \le \ell$.
That is, we have identified a run
$\r = (x_j, \dots, x_\ell)$ in $\x$ with $\epsilon_\r = 1$ such that
for the minimal vector $\y \in \Bvuw$, we have 
\[
\y = \x + \epsilon_\r\sum_{i=j}^\ell\wi
+ \sum_{i>\ell+1}\alpha_i \wi.
\]
It follows immediately from Lemma~\ref{lemma:single_run_from_multiple_runs} that $\x$ can be reduced at $\r$.

Next suppose that
$|x_{\ell+1}| = \fne + 1$ or $|y_{\ell+1}| = \fne + 1$.  These
imply, respectively, that $\kx=\ell+1$ or $\ky=\ell+1$.  
We have $y_{\ell+1} -  x_{\ell+1} = 1 - \alpha_{\ell+1}n$, so in fact it cannot
be that both $|x_{\ell+1}| = \fne + 1$ and $|y_{\ell+1}| = \fne + 1$,
since then the difference would be even.  Thus we have
\[
|1 - \alpha_{\ell+1}n| \le n + 1,
\]
so $\alpha_{\ell+1} \in \{-1,0,1\}$.  If $\alpha_{\ell+1} = 0$,
then the argument from case (a) applies.  Otherwise, we are in one of
the following cases, which we indicate with letters and
address below:
\begin{center}
\begin{tabular}{c|c|c}
&  $|x_{\ell+1}| = \fne + 1$  &  $|y_{\ell+1}| = \fne + 1$ \\
\hline
$\alpha_{\ell+1} = 1$   &   (A)       &  (C)      \\
$\alpha_{\ell+1} = -1$   &  (B)      &   (D)    \\
\end{tabular}
\end{center}
\begin{itemize}[itemsep=5pt]
\item[(A)]  Since $\alpha_{\ell+1} =1$, we must have $x_{\ell+1} = \fne + 1$, that is, this digit is positive and $\kx = l+1$.
Thus we can write
\[ \y = \xx + \sum_{i=j}^{\ell+1} w^{(i)} \]
and as $\y$ is minimal we have found a run at which $\x$ can be reduced.

\item[(B)] Since $\alpha_{\ell+1} = -1$, we conclude that $x_{\ell+1} = -(\fne+1)$ and $\kx = \ell+1 \geq \max(u,w)$
and $y_{\ell+1} = \fne$ and $\ky = \kx + 1$.
Thus the final two digits of $\y$ are $(\fne,-1)$.  
As $\ky > \kx \ge \max(u,w)$,
Lemma~\ref{lemma:even_reduction_end} applies to $\y$, contradicting its
minimality.  Thus this case does not occur.

\item[(C)]
Since $\alpha_{\ell+1} = 1$, we conclude $y_{\ell+1} = -(\fne + 1)$ and $\ky = \ell+1 \geq \max(u,w)$.
The digits $x_{\ell+1}$ and $x_{\ell+2}$ are then determined, namely $(x_{\ell+1},x_{\ell+2}) = (\fne-2,-1)$.  
Set $\r = (x_j, \dots, x_\ell)$; note that $\r$ is a run and we have
\[
\Vert \x \Vert_1 - \Vert \y \Vert_1 = \wt(\r) -2.
\]
As $\kx > \ky \ge \max(u,w)$, we use the second length formula in Lemma~\ref{lemma:length_formula} to compute the difference
\[
|\eta_{u,v,w}(\x)| - |\eta_{u,v,w}(\y)| = \wt(\r) -2 + 2(\kx - \ky) \ge \wt(\r).
\]

As $\y \ords \x$, we must have either $\wt(\r) > 0$ or $\wt(\r) = 0$
and $\y$ lexicographically smaller than $\xx$.  
If $\wt(\r) > 0$, then by
Lemma~\ref{lemma:adjacent_digits_from_weight} there must be a pair
of adjacent digits $\fne$ in $\x$.
Extending this pair to a maximal sequence of digits of the form $\fne$, we see that it must terminate before the final
digit in $\x$.
Moreover, the digit immediately following this sequence has absolute value less than $\fne$
because the sequence is maximal and $x_{\ell+1} = \fne -2$.
It follows immediately from
Lemma~\ref{lemma:even_reduction_adjacent_digits} that there is a run at which $\x$ can be reduced.

\smallskip
\noindent

If $\wt(\r) = 0$
but $\y$ precedes $\x$ in the lexicographic order, we must have $x_{j+1} = \fne$.  In particular $j+1 \le \ell$.
We have therefore found two adjacent digits $\fne$ and can repeat the argument
above and apply Lemma~\ref{lemma:even_reduction_adjacent_digits} to conclude that 
$\x$ contains a run at which it can be reduced.

\item[(D)] Since $\alpha_{\ell+1} = -1$, it is easily verified that $z_{\ell+1} = n+1$ and we must have $x_{\ell+1} = -\fne$.  It follows that $y_{\ell+1} = \fne+ 1$ and thus $\ky = \ell+1$.
As $y_{\ell+2}=0$, we must have $x_{\ell+2} = 1$, so $\kx \ge \ky + 1 \ge \max(u,w) + 1$.
If $\kx > \ell+2$, then $\y$ would have nonzero digits with index greater than $\ell+1$, contradicting $\ky = \ell+1$.  Thus the final two digits of $\x$ are $(-\fne,1)$, and Lemma~\ref{lemma:even_reduction_end} applies to $\x$.
\end{itemize}
These cases prove Lemma~\ref{lemma:lastcase_lemma3.20} and complete the proof of Proposition~\ref{lemma:even_minimal_characterization}.
\end{proof}

\subsection{The special case $n=2$.}
\label{sec:technical_n=2}

When $n=2$ the digit bounds on $\Bvuw$ allow for the possibility that if $\x \in \Bvuw$, then $x_{\kx} = \fne+2 = 3$.
As this case does not occur when $n \geq 4$ is even, we have additional cases and slight variations in approach when $n=2$.

We restate the lemma and propositions from Section~\ref{section:geodesics_n_even_2} which we prove below for easy reference.

\noindent
{\bf Lemma~\ref{lemma:n=2_110}.} \emph{ Let $n=2$ and suppose that $x \in \Bvuw$ and $\delta \in \{\pm 1\}$.
If any of the following occur, then there is a run at which
$\x$ can be reduced, and hence $\x$ is not minimal.
\begin{enumerate}[itemsep=5pt]
    \item $\x$ contains the digits $(\delta, -\delta\alpha)$ for $\alpha > 0$.
    \item $\kx \ne \max(u,w)$ and $\x$ ends in the digits $(\delta, \delta)$.
    \item $\x$ contains the digits $(\delta, \delta, \alpha)$ for any $\alpha$.
\end{enumerate} }
\begin{proof}
Let $j$ be the index in $\x$ of the first digit $\delta$ in any case above.

In case (1), note that in $\y=\x + \delta \wf{j}$  the digits $(\delta, -\delta\alpha)$ have been replaced by $(-\delta,-\delta(\alpha-1))$.  
Since  $\alpha \geq 1$, we see that $\Vert \x + \delta \wf{j} \Vert_1 =\Vert \x \Vert_1 -1$.
That is, $\x$ is not minimal.

If these digits are not the final digits of $\x$, or if they are but $\alpha >1$, then $\ky = \kx$.
Thus, regardless of the length formula used to compute $|\eta_{u,v,w}(\x)|$ and $|\eta_{u,v,w}(\x + \delta \wf{j})|$, the change in word length is 
determined by the change in $\ell^1$ norm, and thus $ \x + \delta \wf{j} \ords \x$.  

If these are the final digits of $\x$, and $\alpha = 1$, writing $\y = \x + \delta \wf{j}$, we have $\ky = \kx-1$.  
If $\kx \le \max(u,w)$ then any change in word length between  $|\eta_{u,v,w}(\x)|$ and $|\eta_{u,v,w}(\x + \delta \wf{j})|$ is determined by the change in $\ell^1$ norm, and hence $ \x + \delta \wf{j} \ords \x$.  
If $\kx > \max(u,w)$, then the fact that $\ky = \kx-1$ implies 
\[
|\eta_{u,v,w}(\x)| - |\eta_{u,v,w}(\x + \delta \wf{j})| = \Vert \x \Vert_1 - \Vert \x + \delta \wf{j}) \Vert_1 + 2(\kx - \ky) > \Vert \x \Vert_1 - \Vert \x + \delta \wf{j}) \Vert_1 >0.
\]
We conclude that $ \x + \delta \wf{j} \ords \x$, so $\x$ is not minimal.

In case (2) with $\kx > \max(u,w)$ consider $\y=\x - \delta\wf{j}$.  
We see that $\ky = \kx-1 \geq \max(u,w)$ and thus we use the second word length formula in Lemma~\ref{lemma:length_formula} to compute both $|\eta_{u,v,w}(\x)|$ and $|\eta_{u,v,w}(\y)|$.  
It is easily checked that $\Vert \y \Vert_1 = \Vert \x \Vert_1 +1$ and it follows that $|\eta_{u,v,w}(\y)| < |\eta_{u,v,w}(\x)|$, so $\x$ is not minimal.

In case (2) with $\kx < \max(u,w)$, let $\y = \x + \delta(\wf{j} + \wf{j+1})$.
It is easily checked that $\Vert \y \Vert_1 = \Vert \x \Vert_1$.  
Also note that there is a lexicographic decrease between $x_{j+1} = \delta$ and $y_{j+1} =0$.
As both geodesic lengths are computed using the first word length formula in Lemma~\ref{lemma:length_formula}, which does not depend on the length of the vector, we see that $\y \ords \x$, so $\x$ is not minimal.

In case (3), if $\sgn(\alpha) = -\sgn(\delta)$ then this reduces to case (1).
Without loss of generality, assume that $\delta = 1$ and $\sgn(\alpha) = \sgn(\delta)$.
\begin{itemize}[itemsep=5pt]
\item If $\alpha\in \{0,2\}$, let $\y = \x+\delta({\bf w}^{(j)}+{\bf w}^{(j+1)})$.
It is easily checked that $\Vert \y \Vert_1 = \Vert \x \Vert_1$
and that $|x_{j+1}| = 1$ while $|y_{j+1}| = 0$, demonstrating a lexicographic reduction from $\x$ to $\y$.
Note that $\kx=\ky$ for either value of $\alpha$ because the last digit of $\x$ cannot be $0$,
so regardless of the length formula used to compute both $|\eta_{u,v,w}(\y)|$ and  $|\eta_{u,v,w}(\x)|$, it follows that $\y \ords \x$.

\item If $\alpha = 3$, then $\alpha$ is the final digit in $\x$.  
Let $\y = \x + {\bf w}^{(\kx-2)}+{\bf w}^{(\kx-1)}+2{\bf w}^{(\kx)}$.  It is easily seen that $\y$ ends in the digits $(-1,0,0,2)$ and $\ky = \kx+1$.
As $x_{\kx}=3$ we must have $\kx \geq \max(u,w)$ and thus use the second formula in Lemma~\ref{lemma:length_formula} to compute both $|\eta_{u,v,w}(\y)|$ and  $|\eta_{u,v,w}(\x)|$.
As $\Vert \y \Vert_1 = \Vert \x \Vert_1 -2$, we see that
\[
|\eta_{u,v,w}(\x)| - |\eta_{u,v,w}(\y)| = 2+2(\kx - \ky) = 2-2=0.
\]
As $|x_{\kx-2}| = |y_{\kx-2}|$ while $x_{\kx-1}=1$ and $y_{\kx-1}=0$, there is a lexicographic reduction and hence $\y \ords \x$.

\item If $\alpha = 1$, then let $(x_j, \dots, x_\ell)$ be a maximal sequence consisting entirely of the digit $1$.  Note that $\ell-j \ge 2$.
First suppose that $\kx < \max(u,w)$.  Set $\y = \x + \sum_{i=j}^\ell \wi$.
Because $x_{\ell+1} \ne 1$ and $\kx < \max(u,w)$, we must have $x_{\ell+1} = 0$.

\smallskip

We can then compute the digits 
\[(x_j,x_{j+1}, \cdots ,x_{\ell+1}) = (\delta, \delta, \cdots ,\delta,0)
\]
of $\x$ and 
\[ (y_j,y_{j+1}, \cdots ,y_{\ell+1}) = (-\delta,0,\cdots ,0,1)\] of $\y$.
If $\kx = \ell$ then $\ky = \kx+1$.  If $\kx >\ell$ then $\ky = \kx$.  In either case, $\ky \le \kx +1$, and we use the first length formula to compute both $|\eta_{u,v,w}(\x)|$ and $|\eta_{u,v,w}(\y)|$, and conclude that
\[
|\eta_{u,v,w}(\x)| - |\eta_{u,v,w}(\y)| = \ell - j + 1 > 0.
\]
Therefore $\y \ords \x$, and $(x_j, \dots, x_\ell)$ is a run at which $\x$ can be reduced.

\smallskip

If $\kx \ge \max(u,w)$, then consider $x_{\ell+1}$. If $x_{\ell+1} \in \{-3,-2,-1,2,3\}$,
then this situation is covered by previous cases.  As we are assuming that we have a maximal
subsequence of digits $1$, we know $x_{\ell+1} \ne 1$.  Thus $x_{\ell+1} = 0$.
We now consider the cases $\kx = \ell$ and $\kx > \ell$.

\smallskip

If $\kx = \ell$, then set $\y = \x + \sum_{i=j}^{\ell-1} \wi$ and note that $\ky = \kx$.  Thus we use the second
length formula in Lemma~\ref{lemma:length_formula} to compute both $|\eta_{u,v,w}(\x)|$
and $|\eta_{u,v,w}(\y)|$, and so the difference between the geodesic lengths comes from the difference between the $\ell^1$ norms of the vectors.  As $\Vert \y \Vert_1 = \Vert \x \Vert_1$, we note that
$x_{j+1} = 1$ and $y_{j+1}=0$, so there is a lexicographic reduction from $\x$ to $\y$.
Thus $(x_j,\dots, x_{\ell-1})$ is a run at which $\x$ can be reduced.

\smallskip

If $\kx > \ell$, then set $\y = \x + \sum_{i=j}^{\ell} \wi$.  
The digits $(x_j, \cdots ,x_\ell)$ and $(y_j, \cdots ,y_\ell)$ are as computed above.  
As $\kx > \ell$ we know that $\ky = \kx$ and thus 
we use the second length formula in  Lemma~\ref{lemma:length_formula} to
compute both $|\eta_{u,v,w}(\x)|$ and $|\eta_{u,v,w}(\y)|$.  
It follows that $|\eta_{u,v,w}(\x)|-|\eta_{u,v,w}(\y)|>0$, so 
$(x_j,\dots, x_\ell)$ is a run at which $\x$ can be reduced.
\end{itemize}
\end{proof}

{\bf Proposition~\ref{lemma:lemma318_n=2}.} {\em Let $n=2$ and $\x \in \Bvuw$.  Then $\x$ is not minimal if and only if one of the following occurs.
\begin{itemize}[itemsep=5pt]
    \item There is a run at which $\x$ can be reduced.
    \item Lemma~\ref{lemma:n=2_notminimal} applies to $\x$.
\end{itemize}}

\begin{proof}
It is clear that if either condition applies to $\x \in \Bvuw$, then $\x$ is not minimal.  We must show the converse.

Let $\xx \in \Bvuw$ and $\y = \x+\z \in \Bvuw$ be minimal, for some 
$\z = \sum_{i=j}^\ell \alpha_i \wi \in \LL_0 $, where $j$ is the minimal index
such that $x_j \ne y_j$.  It follows from Lemma~\ref{lemma:less_than_6} that
$|\alpha_j| \in \{1,2\}$, equivalently that $|z_j| \in \{2,4\}$.

Without loss of generality, we assume that in the sum defining $\z$ all $\alpha_i \neq 0$ for $j \leq i \leq \ell$.
If this is not the case, so $j+1 \leq m < \ell$ is the minimal index with $\alpha_m = 0$, then let $\z' = \sum_{i=j}^{m-1} \alpha_i \wi$.  It follows from Lemma~\ref{lemma:single_run_from_multiple_runs} that $\x + \z' \ords \x$. Moreover, if $\kx \leq \ell$ then $\y$ is minimal as initially written. 
In the arguments below, unless  $\kx \leq \ell$, we use only the fact that $\y \ords \x$.  This allows us to assume for the remainder of the proof that $\alpha_i \neq 0$ for all $j \leq i \leq \ell$.

Assume without loss of generality that $x_j \ge 0$.  We consider the possible values of $\alpha_j \in \{\pm 1, \pm 2\}$.

If $\alpha_j =-2$ then $z_j=4$, hence $y_j\geq 4$.  So this case does not occur. 

If $\alpha_j = 2$ then $z_j=-4$.  We consider $x_j \in \{0,1,2,3\}$.
\begin{enumerate}[itemsep=5pt]
\item If $x_j = 0$ then $y_j=-4$, violating the digit bounds on $\Bvuw$.  So this case does not occur.

\item If $x_j \in \{2,3\}$ then $j = \kx \ge \max(u,w)$, and $\alpha_\kx \neq 0$.  It follows from Lemma~\ref{lemma:n=2_unequal_lengths} that $\ky >\kx$.
Thus when computing $|\eta_{u,v,w}(\x)|$ and $|\eta_{u,v,w}(\y)|$ we use the second length formula in Lemma~\ref{lemma:length_formula}.

\smallskip
\noindent
As $j=\kx$, for $j<i\leq \ky$ we have $y_i = z_i$, moreover for $j<i<\ky$ we have $|y_i| \leq 1$.  
Writing $\z = 2{\bf w}^{(\kx)} + \sum_{i=\kx+1}^{\ky-1} \alpha_i \wi$ we see that the conditions $y_i = z_i$ and $|y_i| \leq 1$ force $\alpha_i=1 $ for all $i$.  
So $\z$ has one of the following two forms:
\begin{itemize}
    \item $(-4,2)$, or
    \item $(-4,0,-1,-1, \cdots ,-1,1)$.
\end{itemize}

\smallskip
\noindent
When $x_j = 2$, so $y_j = -2$, it is immediate that $\Vert \y \Vert_1 > \Vert \x \Vert_1$.
When $x_j = 3$,  so $y_j = -1$, it is easily checked that $\Vert \y \Vert_1 \geq \Vert \x \Vert_1-1$.
As $\ky \geq \kx+1$, it follows from the second length formula in Lemma~\ref{lemma:length_formula} that  $|\eta_{u,v,w}(\x)|<|\eta_{u,v,w}(\y)|$, contradicting our assumption that $\y \ords \x$.  So these cases do not occur.

\item If $x_j=1$, then $y_j = -3$ and we conclude that $\ky = j$. 
Thus for $j<m \leq \ell+1$ we must have $x_m = -z_m$ and $|x_m| \leq 1$ for $m<\ell+1$.  
As in the previous case, we conclude that $\alpha_m = 1$ for $j+1 \leq m \leq \ell$, and hence $\z = (-4,2)$ or $\z = (-4,0,-1,-1, \cdots ,-1,1)$.

\smallskip

\begin{itemize}[itemsep=5pt]
    \item If $\z = (-4,2)$ then $\x$ ends in $(x_j,x_{j+1})=(1,-2)$.
    As $\y \ords \x$ we conclude that $(x_j) = (1)$ is a run at which $\x$ can be reduced.
    \item If $\z=(-4,0,1)$ then $\x$ ends in $(1,0,-1)$ in which case Lemma~\ref{lemma:n=2_notminimal} applies to $\x$.
    \item If $\z = (-4,0,-1,-1, \cdots ,-1,1)$ then $\x$ ends in $(x_j, \cdots ,x_{\ell+1}) = (1,0,1,1, \cdots ,1,-1)$.
    We show that $\r=(x_{j+2}, \cdots ,x_\ell) = (1,1, \cdots ,1)$ is a run at which $\x$ can be reduced.  
    Let ${\bf q} = \x+\sum_{i=j+2}^{\ell} \wi$.  Then by construction $k_{\bf q}<\kx$ and $\Vert {\bf q} \Vert_1 < \Vert \x \Vert_1$.   
    As $\max(u,w) \leq \ky < k_{\bf q}<\kx$, it then follows from the second length formula in Lemma~\ref{lemma:length_formula} that ${\bf q} \ords \x$, that is, $\r$ is a run at which $\x$ can be reduced.
\end{itemize}
\end{enumerate} 

If $\alpha_j = -1$ then $z_j = 2$.  We consider $x_j \in \{0,1,2,3\}$.

\begin{enumerate}[itemsep=5pt]
\item If $x_j \in  \{2,3\}$ then $|y_j| >3$, hence these cases do not occur.

\item If $x_j \in \{0,1\}$ then $y_j \in \{2,3\}$, so $j = \ky \geq \max(u,w)$.
Moreover, for $j+1 \leq m \leq \ell+1$ we have $x_m = -z_m$.
For $j+1 \leq m \leq \ell$ it is also true that $|x_m| \leq 1$, from which it follows that in the definition of $\z$ we have $\alpha_m =1$ for $j+1 \leq m \leq \ell$.
Thus $\z = (2,-1)$ or $\z = (2,1,1, \cdots ,1,-1)$.
\begin{itemize}[itemsep=5pt]
\item 
If $\z = (2,-1)$ it follows that $\x$ ends in either $(x_j,x_{j+1}) = (0,1)$, so  Lemma~\ref{lemma:n=2_notminimal} applies to $\x$, or $(1,-1)$, 
in which case $(x_j) = (1)$ is a run at which $\x$ can be reduced.
\item If $\z = (2,1,1, \cdots ,1,-1)$ it follows that $\x$ ends in  $(x_j,-1,-1, \cdots ,-1,1)$.
Then $(x_{j+1}, \cdots ,x_\ell) = (-1,-1, \cdots ,-1)$ is a run at which $\x$ can be reduced.
The proof is identical to case (3) when $\alpha_j = 2$ with a change of sign.
\end{itemize}
\end{enumerate}

If $\alpha_j = 1$ then $z_j = -2$.  We first make three observations about the form of $\z$.
First, for $j<m < \min(\kx,\ky)$ we have $z_m = \alpha_{m-1}-\alpha_m n$ and $|x_m| \leq 1$, which requires 
 $\alpha_m \in \{\pm 1\}$.  
If $\alpha_m \alpha_{m+1} = -1$ for $j+2\leq m+1 < \min(\kx, \ky)$ then $|z_{m+1}| = 3$, which requires either $\kx = m+1$ or $\ky = m+1$. However, $m+1 < \min(\kx, \ky)$ so we conclude that the sign of $\alpha_i$ is constant for $j+1 \leq i < \min(\kx, \ky)$.
Our assumption that $\alpha_j = 1$ allows us to write 
\[ \z = \sum_{i=j}^{\min(\kx,\ky)-1} \wi + \sum_{i=\min(\kx,\ky)}^{\ell} \alpha_i \wi.\]

The second observation is that for $\min(\kx,\ky)<m\leq\ell+1$ we have either $|x_m| = |z_m|$ or $|y_m| = |z_m|$.  
The digit bounds on $\Bvuw$ again force $|\alpha_m| = 1$
for $\min(\kx,\ky) \leq m \leq \ell$.
If $\alpha_m \alpha_{m+1} = -1$ for $m$ in this range, then $|z_{m+1}| = 3$ and hence either $|x_{m+1}| = 3$ or  $|y_{m+1}| = 3$.
In either case the digit bounds on $\Bvuw$ are violated and so this does not occur.
This allows us to write 
\[ \z = \sum_{i=j}^{\ell-1} \wi + \alpha_\ell {\bf w}^{(\ell)}.\]

Third, recall that we are not assuming that $\y$ is necessarily minimal, but simply that $\y \ords \x$.
However, if $\kx \le \ell$, then we can assume $\y$ is minimal.

We consider $x_j \in \{0,1,2,3\}$.
\begin{enumerate}[itemsep=5pt]
\item If $x_j = 0$ then $y_j = -2$ so $j = \ky \geq \max(u,w)$.
It follows that for $j+1 \leq m \leq \ell+1$ we have $x_m = -z_m$ and for $j+1 \leq m \leq \ell$ we have $|x_m| \leq 1$. 
The latter condition forces $\alpha_m = 1$ in the definition of $\z$ for $j+1 \leq m \leq \ell$.
Thus $\z = (-2,1)$ or $\z = (-2,-1, \cdots ,-1,1)$.

\smallskip

\begin{itemize}[itemsep=5pt]
\item If $\z = (-2,1)$ then $\x$ ends in $(0,-1)$ and Lemma~\ref{lemma:n=2_notminimal} applies to $\x$.
\item If $\z = (-2,-1, \cdots ,-1,1)$, then $(x_{j+1}, \cdots ,x_\ell) = (1,1, \cdots ,1)$ is a run at which $\x$ can be reduced.  
The proof of this is identical to the proof of case (3) when $\alpha_j=2$ with a change of sign.
\end{itemize}

\item If $x_j \in \{2,3\}$ then $y_j \in \{0,1\}$ and $j = \kx \geq \max(u,w)$.  

\smallskip

As $\alpha_{\kx} \neq 0$ it follows from Lemma~\ref{lemma:n=2_unequal_lengths} that $\ky > \kx \geq \max(u,w)$ and thus we use the second length formula in Lemma~\ref{lemma:length_formula} to compute both
$|\eta_{u,v,w}(\x)|$ and $|\eta_{u,v,w}(\y)|$.

\smallskip

As $j = \kx$, for $j+1 \leq m \leq \ky$ we have $y_m = z_m$, and for $j+1 \leq m < \ky$ we have $|y_m| \leq 1$.  
These conditions force $\alpha_m=1$ for all $j+1 \leq m \leq l$ in the definition of $\z$, and thus $\z = (-2,1)$ or $\z = (-2,-1, \cdots ,-1,1)$.

\smallskip

\begin{itemize}[itemsep=5pt]
\item If $\z = (-2,1)$ then $\y$ ends ether in $(0,1)$ or $(1,1)$.  We see that $\Vert \y \Vert_1 = \Vert \x \Vert_1 - 1$.
\item  If $\z = (-2,-1, \cdots ,-1,1)$ then $\y$ ends in $(x_j-2,-1,-1, \cdots ,-1,1)$ where the first digit is $0$ when $x_j = 2$ and $1$ when $x_j = 3$.
We see that $\Vert \y \Vert_1 \geq  \Vert \x \Vert_1$.
\end{itemize}

\smallskip

As $\ky > \kx \geq \max(u,w)$ it then follows from the second word length formula in Lemma~\ref{lemma:length_formula} that $\x \ords \y$, contradicting our assumption that $\y \ords \x$. 
So this case does not occur.

\item If $x_j = 1$, we consider the possible values for $\alpha_\ell \in \{\pm 1, \pm 2\}$.
In all cases, the digit bounds on $\Bvuw$ imply that  for $j \leq i \leq \ell-1$ we have $x_i \in \{0,1\}$.

\smallskip

\begin{itemize}[itemsep=5pt]
\item If $\alpha_\ell = 1$ and $x_\ell \ge 0$, then  by definition $(x_j, \cdots ,x_{\ell})$ is a run at which $\x$ can be reduced.
\item If $\alpha_\ell = 1$ and $x_\ell < 0$, we have $x_\ell \in \{-1,-2,-3\}$.
For all possible values of $x_\ell$, note that $y_\ell < -1$, so $\ky = \ell \ge \max(u,w)$.  Since $y_{\ell+1} =0$, we have $x_{\ell+1} = -1$.  Thus $x_\ell \notin \{-2,-3\}$.
If $x_\ell = -1$, then $\x$ ends with the digits $(-1,-1)$, and it follows from Lemma~\ref{lemma:n=2_110} that there is a run at which $\x$ can be reduced.
\item If $\alpha_{\ell} = 2$, then $z_\ell = -3$.
In order for $\y$ to satisfy the digit bounds on $\Bvuw$, 
we must have $x_\ell \ge 0$, and hence then $(x_j, \cdots ,x_{\ell})$ is a run at which $\x$ can be reduced.  
\item If $\alpha_\ell = -2$ then $(z_{\ell},z_{\ell+1}) = (5,-2)$.  
Hence $x_\ell \in \{-3,-2\}$, and $\kx = \ky =\ell$.  It follows that $x_{\ell+1} = y_{\ell+1} = 0$ which is impossible if $z_{\ell+1} = 2$.  So this case does not occur.
\item If $\alpha_\ell = -1$, the same reasoning as in the above cases shows that for $j \leq i \leq \ell-1$ we have $x_i \in \{0,1\}$.
As $\alpha_{\ell} = -1$, the final digits of $\z$ are $(z_\ell,z_{\ell+1}) = (3,-1)$ and either $\kx = \ell$ or $\ky = \ell$.

\smallskip

If $\ky = \ell$, then it follows from the second observation above that $\kx = \ky+1$ and $(x_{\ell},x_{\ell+1}) = (0,1)$ or $(-1,1)$.  In the first case Lemma~\ref{lemma:n=2_notminimal} applies to $\x$.  
In the second case, $(x_\ell) = (-1)$ is a run at which $\x$ can be reduced.  
It is straightforward to verify that regardless of which length formula from Lemma~\ref{lemma:length_formula} is used, $\x - {\bf w}^{(\ell)} \ords \x$.

\smallskip

If $\kx = \ell$, then $\ky = \kx+1$ and $(y_{\ell},y_{\ell+1})=(x_{\ell}+3, -1)$.
As $y_{\ell+1} \neq 0$ we must have $|y_\ell| \leq 1$, so $\x_\kx \in \{-3,-2\}$ and $\kx \geq \max(u,w)$.
As $\kx = \ell$, recall that we can assume $\y$ is minimal.
As $\ky > \kx  \geq \max(u,w)$, we use the second length formula in Lemma~\ref{lemma:length_formula} to compute both $|\eta_{u,v,w}(\x)|$ and $|\eta_{u,v,w}(\y)|$.
As $\y$ is minimal, we have
\begin{align*}
0 & \le |\eta_{u,v,w}(\x)| - |\eta_{u,v,w}(\y)|  \\
  & = \Vert \x \Vert_1 - \Vert \y \Vert_1 + 2(\kx -\ky) \\
  & = \wt(x_j,\dots,x_{\ell-1}) + |x_{\ell}| - |x_{\ell} +3| - 1 - 2.
\end{align*}
Hence $\wt(x_j,\dots,x_{\ell-1}) \ge 0$.  

\smallskip

Let $\r = (x_j,\dots, x_{\ell-1})$; note that we do not include $x_\ell$ in $\r$ as $x_\ell<-1$.  We show that $\y' = \x + \sum_{i=j}^{\ell-1}{\bf w}^{(i)} \ord \x$, that is, $\x$ can be reduced at $\r$.
First note that as $\ky = \kx+1 = \ell+1$ the fact that the maximal index in $\r$ is $\ell-1$ ensures that $k_{\y'} = \kx$.
Consequently, we use the same length formula to compute $|\eta_{u,v,w}(\x)|$ and $|\eta_{u,v,w}(\y')|$, so any difference between the lengths of the geodesics arises from the change in $\ell^1$ norm between the two vectors.
We compute
\begin{align*}
|\eta_{u,v,w}(\x)| - |\eta_{u,v,w}(\y')| &= \wt(x_j,\dots,x_{\ell-1}) + |x_{\ell}| - |x_{\ell} +1| \\
                                         &> \wt(x_j,\dots,x_{\ell-1}) + |x_{\ell}| - |x_{\ell} +3| - 1 - 2 \\
                                         &\ge 0,
\end{align*}
where the second, strict, inequality holds because $x_{\kx} = x_{\ell} \in \{-3,-2\}$.
Thus $\x$ can be reduced at $\r$.
\end{itemize}
\end{enumerate}
Considering all possible combinations of values of $\alpha_j$ and $x_j$, we have either shown that the combination does not arise,  Lemma~\ref{lemma:n=2_notminimal} applies to $\x$, or $\x$ contains a run at which it can be reduced, proving the lemma.
\end{proof}

{\bf Proposition~\ref{lemma:n=2_adjacent_digits}.} {\em Let $n=2$ and $\x \in \Bvuw$ and $\kx < \max(u,w)$.
Then $\x$ is not minimal if and only if
$\x$ contains the digits  $(\delta, \delta)$
or $(\delta, -\delta)$, for $\delta \in \{\pm 1\}$. }

\begin{proof}
Suppose that $\x$ contains one of the above sequences of digits. As $\kx < \max(u,w)$, it follows from Lemma~\ref{lemma:n=2_110} that $\x$ is not minimal.

It remains to show the converse.
Suppose that $\x$ is not minimal.  Since $\kx < \max(u,w)$, it follows from 
Proposition~\ref{lemma:lemma318_n=2}  that Lemma~\ref{lemma:n=2_notminimal} does not apply to $\x$, so there must be a run $\r = (x_j, \dots, x_\ell)$
at which $\x$ can be reduced.  
Without loss of
generality, we assume that $\epsilon_\r = 1$.  That is, there
is $\z = \sum_{i=j}^{\ell-1}\wf{i} + \alpha_{\ell}\wf{\ell}$ with
$\alpha_\ell \in \{1,2\}$ so that $\y = \x + \z \ords \x$.
As $\kx < \max(u,w)$ and $\ky \leq \kx+1 \leq \max(u,w)$, we use the first length formula in Lemma~\ref{lemma:length_formula} to compute both 
$|\eta_{u,v,w}(\x)|$ and $|\eta_{u,v,w}(\y)|$.
This formula does not depend on the the values of $\kx$ and $\ky$.

As $\y \ords \x$, we have
\begin{align*}
|\eta_{u,v,w}(\x)| - |\eta_{u,v,w}(\y)|  &= \wt(\r) + |x_{\ell+1}| - |y_{\ell+1}| \ge 0 .\\
\end{align*}
Moreover, if 
the above difference is zero,
there must be a lexicographic reduction from $\x$ to $\y$.
Recall that $y_{\ell+1} = x_{\ell+1} + \alpha_\ell$ and $\alpha_\ell \in \{1,2\}$.  As $\kx < \max(u,w)$ we know that $|x_{l+1}| \leq 1$.

First suppose that $x_{\ell+1}\ge 0$.
If $\alpha_\ell = 2$, then $ |x_{\ell+1}| - |y_{\ell+1}| = -2$, so $\wt(\r) - 2 \ge 0$.
Recalling the definition of $\wt(\r)$, we see that $\r$ must contain at least three more $1$'s than $0$'s.
If $\alpha_\ell = 1$, then $ |x_{\ell+1}|-|y_{\ell+1}| = -1$,
and $\r$ must contain at least two more $1$'s than $0$'s. 
In either case, it follows from Remark~\ref{remark:digit_sequences} that $\r$ contains either the digits $(1,1,0)$ or ends in $(1,1)$.
In either case, $\r$ contains $(1,1)$.

Next suppose that $x_{\ell+1} < 0$, which implies $x_{\ell+1}=-1$.  We consider the possibilities for $x_l$.
\begin{enumerate}[itemsep=5pt]
        \item If $x_\ell = 1$, then $\x$ contains the digits $(1,-1)$.
        \item If $x_\ell = 0$ and $\alpha_\ell = 2$, then $\z$ ends
        with the digits $(-3,2)$ forcing $\y$ to contain the digits $(-3,1)$, contradicting the digit restrictions on $\Bvuw$.  So this case does not occur.
        \item If $x_\ell = 0$ with $\alpha_\ell = 1$, then $|x_{\ell+1}|-|y_{\ell+1}| = 1$ and we have $\wt(\r) +1 \geq 0$.  
        
        \smallskip
        
        \begin{enumerate}[itemsep=5pt]
        \item If $\wt(\r) = -1$, so $|\eta_{u,v,w}(\x)| = |\eta_{u,v,w}(\y)|$,
        there must be a lexicographical reduction from $\x$ to $\y$.  
        Since $x_\ell=0$ and the first digit of $\r$  is $x_j = 1$, we conclude that $\r$ has length at least 2. 
        To ensure the lexicographic reduction requires $x_{j+1} = 1$ and thus $\r$
        begins with the digits $(1,1)$. 
        \item If $\wt(\r) +1 > 0$, there must be at least one more occurrence of the digit 1 than the digit 0 in $\r$.  Since we know that $x_j = 1$ and $x_\ell = 0$, this condition forces $\r$ to contain the digits $(1,1)$.
        \end{enumerate}
\end{enumerate}
\end{proof}

\bibliographystyle{plain}
\bibliography{curvature}

\begin{thebibliography}{10}

\bibitem{BDK}
A.~{Bar-Natan}, M.~{Duchin}, and R.~{Kropholler}.
\newblock {Conjugation curvature for Cayley graphs}.
\newblock {\em Journal of Topology and Analysis}, to appear.

\bibitem{BE}
Jos\'{e} Burillo and Murray Elder.
\newblock Metric properties of {B}aumslag-{S}olitar groups.
\newblock {\em Internat. J. Algebra Comput.}, 25(5):799--811, 2015.

\bibitem{CEG}
D.~J. Collins, M.~Edjvet, and C.~P. Gill.
\newblock Growth series for the group {$\langle x,y|\ x^{-1}yx=y^l\rangle$}.
\newblock {\em Arch. Math. (Basel)}, 62(1):1--11, 1994.

\bibitem{DL}
Volker Diekert and J\"{u}rn Laun.
\newblock On computing geodesics in {B}aumslag-{S}olitar groups.
\newblock {\em Internat. J. Algebra Comput.}, 21(1-2):119--145, 2011.

\bibitem{Elder}
Murray Elder.
\newblock A linear-time algorithm to compute geodesics in solvable
  {B}aumslag-{S}olitar groups.
\newblock {\em Illinois J. Math.}, 54(1):109--128, 2010.

\bibitem{EH}
Murray Elder and Susan Hermiller.
\newblock Minimal almost convexity.
\newblock {\em J. Group Theory}, 8(2):239--266, 2005.

\bibitem{FM1}
Benson Farb and Lee Mosher.
\newblock A rigidity theorem for the solvable {B}aumslag-{S}olitar groups.
\newblock {\em Invent. Math.}, 131(2):419--451, 1998.
\newblock With an appendix by Daryl Cooper.

\bibitem{freden}
Eric~M. Freden, Teresa Knudson, and Jennifer Schofield.
\newblock Growth in {B}aumslag-{S}olitar groups {I}: subgroups and rationality.
\newblock {\em LMS J. Comput. Math.}, 14:34--71, 2011.

\bibitem{MRUV}
A.~Myasnikov, V.~Roman'kov, A.~Ushakov, and A.~Vershik.
\newblock The word and geodesic problems in free solvable groups.
\newblock {\em Trans. Amer. Math. Soc.}, 362(9):4655--4682, 2010.

\bibitem{YO1}
Yann Ollivier.
\newblock Ricci curvature of metric spaces.
\newblock {\em C. R. Math. Acad. Sci. Paris}, 345(11):643--646, 2007.

\bibitem{YO2}
Yann Ollivier.
\newblock Ricci curvature of {M}arkov chains on metric spaces.
\newblock {\em J. Funct. Anal.}, 256(3):810--864, 2009.

\bibitem{YO3}
Yann Ollivier.
\newblock A survey of {R}icci curvature for metric spaces and {M}arkov chains.
\newblock In {\em Probabilistic approach to geometry}, volume~57 of {\em Adv.
  Stud. Pure Math.}, pages 343--381. Math. Soc. Japan, Tokyo, 2010.

\bibitem{YO4}
Yann Ollivier.
\newblock A visual introduction to {R}iemannian curvatures and some discrete
  generalizations.
\newblock In {\em Analysis and geometry of metric measure spaces}, volume~56 of
  {\em CRM Proc. Lecture Notes}, pages 197--220. Amer. Math. Soc., Providence,
  RI, 2013.

\bibitem{TW_growth}
Jennifer Taback and Alden Walker.
\newblock A new proof of the growth rate of the solvable {B}aumslag-{S}olitar
  groups, 2020.
\newblock In preparation.

\end{thebibliography}

\end{document}